\documentclass[12pt,reqno]{amsart}

\usepackage{amssymb}
\usepackage[T1]{fontenc}
\usepackage[latin1]{inputenc} 
\usepackage{dsfont, mathrsfs}
\usepackage{a4wide}  
\usepackage{stmaryrd}
\usepackage{amsthm}
\usepackage{graphicx}
\usepackage{color,soul}  
\usepackage{xspace}
\usepackage{yhmath, epigraph}
\usepackage[
        colorlinks,
        hyperfootnotes=true,
        pagebackref,
        hyperindex,
        linkcolor=blue]{hyperref}
\usepackage{bbm}
\usepackage[bbgreekl]{mathbbol}
\DeclareSymbolFontAlphabet{\mathbb}{AMSb}
\DeclareSymbolFontAlphabet{\mathbbl}{bbold}

\allowdisplaybreaks

\theoremstyle{plain}
\newtheorem{theorem}{Theorem}[section]
\newtheorem{lemma}[theorem]{Lemma}

\theoremstyle{definition}

\theoremstyle{remark}
\newtheorem{remark}[theorem]{Remark}

\newcommand{\Ee}{\mathbb{E}}

\newcommand{\Nn}{\mathbb{N}}
\newcommand{\Pp}{\mathbb{P}}

\newcommand{\Rr}{\mathbb{R}}

\newcommand{\Un}{\mathds{1}}

\newcommand{\Ce}{\mathcal{C}}

\newcommand{\Fe}{\mathcal{F}}

\newcommand{\Le}{\mathcal{L}}
\newcommand{\Me}{\mathcal{M}}

\newcommand{\Ye}{\mathcal{Y}}
\newcommand{\Ze}{\mathcal{Z}}

\newcommand{\Teb}{{\boldsymbol{\mathcal{T}}}}

\newcommand{\Ns}{\mathscr{N}}

\newcommand{\Us}{\mathscr{U}}

\newcommand{\Zg}{\mathfrak{Z}}

\newcommand{\deltab}{\boldsymbol{\delta}}
\newcommand{\gammab}{\boldsymbol{\gamma}}
\newcommand{\kappab}{\boldsymbol{\kappa}}

\newcommand{\lambdab}{\boldsymbol{\lambda}}

\newcommand{\Gammab}{\boldsymbol{\Gamma}}

\newcommand{\Thetab}{\boldsymbol{\Theta}}

\newcommand{\Fb}{{\boldsymbol{F}}}

\newcommand{\Mb}{{\boldsymbol{M}}} 
\newcommand{\Pb}{{\boldsymbol{P}}}
\newcommand{\Qb}{{\boldsymbol{Q}}}
\newcommand{\Rb}{{\boldsymbol{R}}}
\newcommand{\Tb}{{\boldsymbol{T}}}

\newcommand{\Xb}{{\boldsymbol{X}}}
\newcommand{\Zb}{{\boldsymbol{Z}}}

\newcommand{\xb}{{\boldsymbol{x}}}

\newcommand{\ensemble}[1]{ \left\lbrace #1 \right\rbrace } 
\newcommand{\prth}[1]{\!\left( #1 \right) }
\newcommand{\crochet}[1]{\!\left[ #1 \right] }  
\newcommand{\intcrochet}[1]{\llbracket #1 \rrbracket} 
\newcommand{\abs}[1]{\left| #1 \right|}  
\newcommand{\norm}[1]{\left| \! \left| #1 \right| \! \right|}

\newcommand{\Esp}[1]{ \Ee  \prth{ #1 } }  
\newcommand{\Prob}[1]{ \Pp \prth{ #1 } }  
\newcommand{\Espr}[2]{ \Ee_{#1} \prth{ #2 } } 

\def\inv{^{-1}}

\def\longlongrightarrow{\hspace{+0.1ex} - \hspace{-1.1ex} - \hspace{-1.1ex} - \hspace{-1.1ex}\longrightarrow  } 

\newcommand{\tendvers}[2]{ \underset{#1 \rightarrow #2}{\longlongrightarrow} }  
\newcommand{\cvlaw}[2]{\stackrel{\Le}{\underset{#1 \, \rightarrow \, #2}{\longlongrightarrow}}}  

\newcommand{\equivalent}[1]{ {\underset{#1 }{\sim} } }

\newcommand{\Unens}[1]{ \Un_{ \ensemble{#1} } }
\def\eqlaw{\stackrel{\Le}{=}}
\newcommand{\pe}[1]{\left[ #1 \right] }

\newcommand{\centered}[1]{\stackrel{ _\circ}{#1}\!\!}

\def\Var{{ \operatorname{Var} }}

\def\FM{{\operatorname{FM}}}

\def\dKol{ d_{\operatorname{Kol}} }

\def\Ber{ \operatorname{Ber} }

\def\Poisson{ \operatorname{Po} }

\def\Bin{\operatorname{Bin}}

\def\Argtanh{ \operatorname{Argtanh} }

\def\geq{\geqslant}
\def\leq{\leqslant}

\let\oldforall\forall
\def\forall{\oldforall\,} 

\let\oldexists\exists
\def\exists{\oldexists\,}

\definecolor{pink}{RGB}{219, 48, 122}
\definecolor{purple}{RGB}{128, 0, 128}
\definecolor{rougeclair}{rgb}{1,.65,.65}

\newcommand{\blue}[1]{\color{black}#1\color{black}} 
\newcommand{\purple}[1]{\color{black}#1\color{black}}

\title[The Curie-Weiss model through exchangeability surrogates]{ A surrogate by exchangeability approach \\ to the Curie-Weiss model\vspace{-0.2cm}}
\author[Y. Barhoumi-Andr\'eani]{Yacine Barhoumi-Andr\'eani}
\author[M. Butzek]{Marius Butzek}
\author[P. Eichelsbacher]{Peter Eichelsbacher}
\date{\today}

\subjclass[2020]{60E99, 82B05, 82B20, 60G09}

\begin{document}
\begin{abstract}
We introduce a new general concept of surrogate random variable, the ``surrogate by exchangeability'' that allows to study the class of random variables that can be decomposed by means of an independent randomisation. 

As an example, we treat the case of the Curie-Weiss model using the explicit construction of its De Finetti measure of exchangeability. Writing the magnetisation as a sum of i.i.d.'s randomised by the underlying De Finetti random variable, the surrogate study shows that the appearance of a phase transition can be understood as a competition between these two sources of randomness, the Gaussian regime corresponding to a marginally relevant disordered system.
\end{abstract}
\maketitle

\tableofcontents


\section{Introduction}

\subsection{Motivations, history and main result}

The Curie-Weiss model of $n$ spins at inverse temperature $ \beta \geq 0$ is the law of the random variables $ (X_k^{(\beta)})_{1 \leq k \leq n} $ defined by 
\begin{align}\label{Def:CurieWeiss}
\Pp_n^{(\beta)} \equiv \Pp_{ \prth{ X_1^{(\beta)}\!, \dots,\, X_n^{(\beta)} } } := \frac{e^{ \frac{\beta}{2n} S_n^2 } }{ \Esp{ e^{\frac{\beta}{2n} S_n^2 } } } \bullet \Pp_{ (X_1, \dots, X_n) },
\end{align}
where $ f \bullet \Pp $ denotes the bias/penalisation/tilting of $ \Pp $ by $f$ and
\begin{align*}
(X_k)_{1 \leq k \leq n}\sim \textrm{i.i.d. Ber}_{\ensemble{\pm 1}}(1/2), \qquad\qquad S_n := \sum_{k = 1}^n X_k.
\end{align*}

One classical modification of this model consists in adding an additional parameter, the external magnetic field $ \mu $, i.e.
\begin{align}\label{Eq:CurieWeissGeneral}
\Pp_{ \prth{ X_1^{(\beta, \mu)}\!, \dots,\, X_n^{(\beta, \mu)} } } := \frac{e^{ \frac{\beta}{2n} S_n^2 + \mu S_n} }{ \Esp{ e^{\frac{\beta}{2n} S_n^2 + \mu S_n } } } \bullet \Pp_{ (X_1, \dots, X_n) }.
\end{align}

In this article, we will only be concerned with the case $ \mu = 0 $.

\medskip

This simple model of statistical mechanics was originally introduced by Pierre Curie in 1895 \cite{CurieThese} and refined by Pierre-Ernest Weiss in 1907 \cite{WeissOnCurie} as an exactly solvable model of ferromagnetism: the ferromagnetic alloys have the property of spontaneously changing their magnetic behaviour when heated, once a certain critical temperature threshold is reached.

Nowadays, it is presented as a \textit{mean-field} approximation of the more refined Ising model, i.e. as the replacement of an interaction with nearest neighbour $ \sum_{i \sim j} X_i X_j $ by an interaction with all other spins $ \sum_{i, j} X_i X_j = S_n^2 $ (see e.g. \cite[ch. 2]{FriedliVelenik}). Such approximations are frequently performed in probability theory in general and in statistical mechanics in particular. Replacing a complex model with a simpler one whose overall behavior may be examined through explicit computations allows to get an intuition of the features that can be inferred from the original model, sometimes with no alteration. This is the case of the Curie-Weiss model when it comes to spontaneous magnetisation and metastability. In particular, it does exhibit a phase transition with three distinct behaviours at high, critical and low temperature.

We refer to \cite[ch. 2]{FriedliVelenik} for a friendly introduction to its main properties, or the more classical references \cite{Brout, KacMecaStat, StanleyMecaStat, Thompson}.

\medskip

Of particular interest is the law of the \textit{unnormalised} magnetisation
\begin{align}\label{Def:UnnormalisedMagnetisation}
M_n^{(\beta)} := \sum_{k = 1}^n X_k^{(\beta)},
\end{align}
given, for all continous bounded function $f$, by 
\begin{align*}
\Esp{ f\prth{ M_n^{(\beta)} } } = \frac{ \Esp{ e^{ \frac{\beta}{2n} S_n^2 } f(S_n) } }{ \Esp{ e^{ \frac{\beta}{2n} S_n^2 } } }.
\end{align*}

This random variable contains all the information of the model, as the law of every spin is defined by means of $ M_n^{(\beta)} $.

\medskip

The difference of behaviour of the system when the inverse temperature $ \beta $ varies can be summarised in the following theorem that can be found e.g. in the books
\cite{Brout, FriedliVelenik, KacMecaStat, StanleyMecaStat, Thompson} or in the papers \cite{EllisNewman, EllisNewmanRosen}. In the Bernoulli case that we consider, the case $\beta=1$ is, to the best of the authors' knowledge, first due to Simon and Griffiths \cite[thm. 1]{SimonGriffiths}.

\begin{theorem}[Fluctuations of the \textit{unnormalised} magnetisation]\label{Theorem:FluctuationsMagClassical}
We have :
\begin{enumerate}

\medskip
\item If $ \beta < 1 $, 
\begin{align*}
\frac{1}{\sqrt{n}} \, M_n^{(\beta )} \cvlaw{n}{+\infty} \Ns\prth{ 0, \frac{1}{1 - \beta } }. 
\end{align*}

\medskip
\item If $ \beta = 1 $, let $ \gammab(a) \sim \Gammab(a) $ i.e. $ \Prob{ \gammab(a) \in dx} = \Unens{x > 0} x^a e^{-x} \frac{dx}{x} $ for $ a > 0 $, and $ B_{\pm 1} \sim \Ber_{\ensemble{\pm 1} }(1/2) $ independent of $ \gammab(a) $; then,
\begin{align*}
\frac{1}{n^{3/4}} \, M_n^{(1)} \cvlaw{n}{+\infty} B_{ \pm 1} \, \gammab(1/4)^{1/4} \sim e^{- \frac{x^4}{12 }} \frac{dx}{\Ze_0} .
\end{align*}

\medskip
\item If $ \beta = 1 - \frac{\gamma}{\sqrt{n} } $ with $ \gamma \in \Rr $ fixed, let $ \Fb_{\!\! \gamma} \sim  e^{ - \gamma \frac{x^2}{2} - \frac{x^4}{12} } \frac{dx}{\Ze_\gamma } $; then,
\begin{align*}
\frac{1}{n^{3/4}} \, M_n^{(1 - \gamma / \sqrt{n} )} \cvlaw{n}{+\infty} \Fb_{\!\! \gamma}.
\end{align*}

\medskip
\item If $ \beta > 1 $, let $ B_{\pm 1} \sim \Ber_{\ensemble{\pm 1} }(1/2) $; then,
\begin{align*}
\frac{1}{n} \, M_n^{(\beta )} \cvlaw{n}{+\infty} t_\beta\, B_{\pm 1}, \qquad \mbox{where } t_\beta = \tanh(\beta t_\beta).
\end{align*}
\end{enumerate}
\end{theorem}

\medskip

We note that the left transition $ \gamma > 0 $ and the right transition $ \gamma < 0 $ give the same limiting law, even though the graph of the density displays a very different behaviour, with two different modes that announce the case $ \beta > 1 $ in the second case; see Remark~\ref{Rk:beta>1:conditioning}. This \textit{continuous phase transition} is characteristic of the case $ \mu = 0 $, see e.g. \cite[\S~2.5.3]{FriedliVelenik}.

\medskip

Several modifications and follow-ups to Theorem~\ref{Theorem:FluctuationsMagClassical} can be made: universality of the limits when the law of the $ (X_k)_k $ is changed \cite{EllisNewman, EllisNewmanRosen}, dynamical spin-flip version \cite{KulskeLeNy}, concentration properties of the spins around the limit in the case $ \beta > 1 $ \cite{ChatterjeeStein07, ChatterjeeDey09}, moderate and large deviations \cite{EichelsbacherLoeweMDP, EllisBook}, modification of the Hamiltonian leading to the Curie-Weiss-Potts model \cite{EichelsbacherMartschink}, the inhomogeneous Curie-Weiss model \cite{DommersEichelsbacher}, the $p$-spin Curie-Weiss model \cite{MukherjeeSonBhattacharya}, the multi-group Curie-Weiss model \cite{FleermannKirschToth, KirschToth}, etc. 

The methods to prove the previous results are diverse and varied: the concentration property in \cite{ChatterjeeStein07, ChatterjeeDey09} uses exchangeable pairs to build a functional satisfying a simple sub-linear inequality ameanable to the Herbst argument (that consists in bounding the Laplace transform of the random variable by means of a differential inequality). The methods in \cite{EllisNewman, EllisNewmanRosen} use the Laplace/Fourier transform in the same way the Central Limit Theorem is proven, and the technique used in \cite{FleermannKirschToth, KirschToth} uses the method of moments and, in the case of joint limiting Gaussian distribution, computes the limiting correlation between sums of different families of spins by means of a Hubbard-Stratonovich transformation\footnote{This transformation amounts to the computation of the De Finetti measure in Lemma~\ref{Lemma:DeFinettiCWspins}; compare with \cite[prop. 31]{KirschToth}.} estimated with the Laplace method.

These results allow the variety of techniques used in probability theory to express their power and illustrate a form of richness of the field, both in the questions asked and in the responses that follow.

\medskip

The goal of this paper is to pursue this trend by giving yet another proof of this old and respectable theorem with the additional result of the speed of convergence in Kolmogorov and Fortet-Mourier distance; more precisely:

\medskip
\begin{theorem}[Summary of the main results]\label{Theorem:MainIntro}
$ $

\begin{enumerate}

\item In the Fortet-Mourier norm defined by 
\begin{align}\label{Def:dFM}
d_\FM(X, Y) := \sup_{\norm{h}_\infty \leq 1, \norm{h'}_\infty \leq 1}\abs{ \Esp{h(X)} - \Esp{h(Y)} } 
\end{align}
one has for $h$ with $\norm{h}_\infty \leq 1, \norm{h'}_\infty \leq 1$:
\begin{enumerate}

\item If \underline{$ \beta < 1 $} (Theorem~\ref{Theorem:MagnetisationBetaSmaller1} in the sequel)
\begin{align*}
& \abs{ \Esp{ h\prth{ \frac{M_n^{(\beta )}}{\sqrt{n}} } } -  \Esp{ h\prth{ \Zb_{\!\beta} } } } \leq C \frac{ \norm{h'}_\infty }{\sqrt{n} }  +  \frac{1}{n} \prth{ \frac{ \beta }{ 1 - \beta }  \norm{h'}_\infty  +  C(\beta)  \norm{h}_\infty } \\
\Longrightarrow \qquad & d_\FM\prth{  \frac{M_n^{(\beta )}}{\sqrt{n}}, \, \Ns\prth{ 0, \frac{1}{1 - \beta} }  } 
               \leq  \frac{C}{\sqrt{n}} + \frac{D(\beta)}{n}
\end{align*}
for constants $ C, D(\beta) > 0 $, with $ \Zb_{\!\beta} \sim \Ns(0, \frac{1}{1 - \beta}) $.

\medskip
\item If \underline{$ \beta = 1 $} (Theorem~\ref{Theorem:MagnetisationBetaEqual1} in the sequel)
\begin{align*}
& \abs{ \Esp{ h\prth{ \frac{M_n^{(1 )}}{n^{3/4}}  } } -  \Esp{ h( \Fb_{\!0} ) } } \leq  \prth{ \frac{  C }{ \sqrt{n} } + O\prth{ \frac{ 1 }{ n^{3/4} } } } \prth{ \vphantom{a^{a^a}} \norm{h}_\infty +  \norm{h'}_\infty } \\
\Longrightarrow \qquad &
d_\FM\prth{  \frac{M_n^{(1)}}{n^{3/4}}, \, \Fb_{\!0}  } 
               \leq  \frac{C}{\sqrt{n}} + \frac{E}{n^{3/4}}
\end{align*}
for constants $ C, E > 0 $.

\medskip
\item If \underline{$ \beta = 1 - \frac{\gamma}{\sqrt{n}} $} (Theorem~\ref{Theorem:MagnetisationBeta_n} in the sequel)
\begin{align*}
& \abs{ \Esp{ h\prth{ \frac{M_n^{(\beta_n )}}{n^{3/4}}  } } -  \Esp{ h( \Fb_{\!\! \gamma} ) } } \leq  \prth{ \frac{  C }{ \sqrt{n} } + O\prth{ \frac{ 1 }{ n^{3/4} } } } \prth{ \vphantom{a^{a^a}} \norm{h}_\infty +  \norm{h'}_\infty }\\
\Longrightarrow \qquad &
 d_\FM\prth{  \frac{M_n^{(1 + \frac{\gamma}{n})}}{n^{3/4}}, \, \Fb_{\!\gamma}  } 
               \leq  \frac{C}{\sqrt{n}} + \frac{E(\gamma)}{n^{3/4}}
\end{align*}
for constants $ C, E(\gamma) > 0 $.

\medskip
\item If \underline{$ \beta > 1 $} (Theorem~\ref{Theorem:MagnetisationBetaBigger1} in the sequel)
\begin{align*}
& \abs{ \Esp{ h\prth{ \frac{M_n^{(\beta )}}{n } } } -  \Esp{ h\prth{ \Teb^{(\beta)} } } } \leq \prth{  \frac{C }{\sqrt{n} }  + O_\beta\prth{\frac{1}{n}} } \norm{h'}_\infty \\
\Longrightarrow 	\qquad &
d_\FM\prth{  \frac{M_n^{(\beta)}}{n }, \, t_\beta B_{\pm 1} } 
               \leq  \frac{C}{\sqrt{n}}  + \frac{F(\beta)}{n}
\end{align*}
for constants $ C, F(\beta) > 0 $, with $ \Teb^{(\beta)} \sim t_\beta B_{\pm 1} $.

\end{enumerate}

\medskip

\item In the Kolmogorov norm defined by 
\begin{align}\label{Def:dKol}
\dKol(X, Y) := \sup_{t \in \Rr}\abs{\Prob{X \leq t} - \Prob{Y \leq t} } 
\end{align}
one has:
\begin{enumerate}

\item If \underline{$ \beta < 1 $} (Theorem~\ref{Theorem:MagnetisationBetaSmaller1:dKol} in the sequel)
\begin{align*}
\dKol\prth{  \frac{M_n^{(\beta )}}{\sqrt{n}}, \, \Ns\prth{ 0, \frac{1}{1 - \beta} }  } 
               \leq  \frac{C'}{\sqrt{n}} + \frac{D'(\beta)}{n}
\end{align*}
for constants $ C', D'(\beta) > 0 $.

\medskip
\item If \underline{$ \beta = 1 $} (Theorem~\ref{Theorem:MagnetisationBeta=1:dKol} in the sequel)
\begin{align*}
\dKol\prth{  \frac{M_n^{(1)}}{n^{3/4}}, \, \Fb_{\!0}  } 
               \leq  \frac{C'}{\sqrt{n}} + \frac{E'}{n^{3/4}}
\end{align*}
for constants $ C', E' > 0 $.

\medskip
\item If \underline{$ \beta = 1 - \frac{\gamma}{\sqrt{n}} $} (Theorem~\ref{Theorem:MagnetisationBeta_n:dKol} in the sequel)
\begin{align*}
\dKol\prth{  \frac{M_n^{(1 + \frac{\gamma}{n})}}{n^{3/4}}, \, \Fb_{\!\gamma}  } 
               \leq  \frac{C'(\gamma)}{\sqrt{n}} + \frac{E'(\gamma)}{n^{3/4}}
\end{align*}
for constants $ C'(\gamma), E'(\gamma) > 0 $.

\medskip
\item If \underline{$ \beta > 1 $} (Theorem~\ref{Theorem:MagnetisationBetaBigger1:dKol} in the sequel)
\begin{align*}
\dKol\prth{  \frac{M_n^{(\beta)}}{n }, \, t_\beta B_{\pm 1} } 
               \leq  \frac{C'}{\sqrt{n}}  + \frac{F'(\beta)}{n}
\end{align*}
for constants $ C', F'(\beta) > 0 $.

\end{enumerate}

\end{enumerate}
\end{theorem}

The results in the Kolmogorov norm can be compared with the results in Theorem 2.1 in \cite{ChatterjeeShao} and in Theorem 3.7, 3.8 and 3.2 in \cite{EichelsbacherLoewe}. In both papers Stein's method for exchangeable pairs was developed for a rich class of distributional
approximations. Given a random variable $W$, Stein's method
is based on the construction of another variable $\widehat{W}$ (some coupling) such that the pair $(W, \widehat{W})$ is exchangeable, i.e. their
joint distribution is symmetric. Suppose that $X = (X_i^{(\beta)})_i$ is drawn from $\Pp_n^{(\beta)}$, then let $I$ be a uniformly distributed random variable on the set $\{1, \ldots, n \}$,
independent of $X$. Given $I=i$, $\widehat{X} $ is constructed by taking one step in the heat-bath Glauber dynamics: coordinate $X_i^{(\beta)}$ is replaced by $ \widehat{X}_i^{(\beta)} $ drawn from the conditional distribution of $X_i^{(\beta)}$ given $(X_j^{(\beta)})_{j \not= i}$. We denote by $ \widehat{M}_n^{(\beta)} $ the version of $ M_n^{(\beta)} $ with $X_I^{(\beta)}$ replaced by $ \widehat{X}_I^{(\beta)} $, i.e.
$ \widehat{M}_n^{(\beta)} = M_n^{(\beta)} + \widehat{X}_I^{(\beta)}  - X_I^{(\beta)} $. The main observation is the linear regression property 
\begin{align*}
\Esp { M_n^{(\beta)} - \widehat{M}_n^{(\beta)} \Big\vert M_n^{(\beta)}} = \Esp{X_I^{(\beta)} -  \widehat{X}_I^{(\beta)} \Big\vert M_n^{(\beta)}} = \frac{M_n^{(\beta)}}{n} - \frac 1n \sum_{i=1}^n \Esp{ \widehat{X}_i^{(\beta)} \Big\vert M_n^{(\beta)}}.
\end{align*}

The heart of the method in \cite{ChatterjeeShao} and \cite{EichelsbacherLoewe} is therefore to obtain $\Esp { \widehat{X}_i^{(\beta)} \Big\vert M_n^{(\beta)}} = \tanh ( \beta m_i(X))$ with $ m_i(X) := \frac 1n \sum_{k=1, k \not= i}^n X_k^{(\beta)}$.
For further details see \cite{ChatterjeeShao} and \cite{EichelsbacherLoewe}.
\medskip

The proof of Theorem~\ref{Theorem:MainIntro} will use a new perspective that shall shed new light on the nature of the limits and their speeds, in addition to fitting in a more conceptual way into the field of statistical mechanics, by highlighting the articulation of a structural property of the model that is often unnoticed in the main textbooks on the topic, but nevertheless well-studied in some parts of the literature.. 

We hope that future works will expand this philosophy to the variety of available statistical mechanical models in addition to go beyond the scope of the sole fluctuations.

\medskip
\subsection{Literature and philosophy of proof}

\subsubsection{State of the art}

In a domain as venerable as the Curie-Weiss model, older than 100 years, it seems very difficult to innovate, especially with the original model of Bernoulli spins. In addition to the classical studies of Ellis-Newman-Rosen \cite{EllisNewman, EllisNewmanRosen} that use the Laplace transform, one must add the classical tools of probability theory when concerned with distributional approximation such as Stein's method of exchangeable pairs \cite{ChatterjeeStein07, ChatterjeeDey09, ChatterjeeShao, EichelsbacherLoewe}.

A very important feature of the Curie-Weiss random variables~\eqref{Def:CurieWeiss} that is not present in e.g. the Ising model is the existence of an \textit{exchangeability measure}. While exchangeable pairs have been used thoroughly by means of Stein's method, see \cite{ChatterjeeShao},\cite{EichelsbacherLoewe}, writing the spins as i.i.d. random variables conditionally to a measure of mixture is a particularly strong peculiarity that was taken advantage of in several works on the Curie-Weiss model, for instance:
\begin{itemize}

\item Jeon \cite{JeonCW} extends Theorem~\ref{Theorem:FluctuationsMagClassical} (with more general distributions than the Bernoulli ones) into a functional central limit theorem towards a Gaussian process for the rescaled partial sums of sub-critical spins $ ( \sum_{k = 1}^{\pe{nt}} X_k^{(\beta < 1)})_{t \in [0, 1]} $ and proves that the critical process rescaled in the same way is non Gaussian. He moreover considers a wider class of penalisations given by the exponential of the cumulant generating series of a particular random variable (with conditions on the existence of a unique global minimiser for the difference between this series and the Gaussian one), a class already considered by Ellis and Newman \cite{EllisNewman} and called the \textit{class L}.

\item Chaganty and Sethuraman \cite{ChagantySethuraman1987} extend Theorem~\ref{Theorem:FluctuationsMagClassical} to a larger class of penalisations and distribution functions using a methodology based on large deviations/local CLT type of manipulations, building on the works \cite{ChagantySethuraman1985, OkamotoProba}. In particular, \cite[ex. 4.4]{ChagantySethuraman1987} is specifically concerned with the Bernoulli case. The same class \textit{L} of random random variables is here again considered.

\item Papangelou \cite{Papangelou} extends the Jeon results by giving a functional central limit theorem towards a Brownian bridge for the process of partial sums of critical spins $ ( \sum_{k = 1}^{\pe{nt}} X_k^{(\beta = 1)})_{t \in [0, 1]} $ rescaled with a random shift, implying that \cite[p. 266]{Papangelou}
\begin{quote}
``\textit{there is considerable Gaussian structure in the internal fluctuations of the critical model}'',
\end{quote}
an assertion that will be confirmed here, although only for $t = 1$.

\item Liggett, Steif and Toth \cite{LiggettSteifToth} study the problem of extension to an infinite system of spins, following general works on extendibility of symmetric probability distributions such as \cite{AldousStFlour, DiaconisExchSurvey, Scarsini, Spizzichino}. In addition to several statistical mechanical models such as the Potts, Curie-Weiss-Heisenberg and Curie-Weiss clock model that are generalisations of the Curie-Weiss model, they prove the infinite extendibility for the ferromagnetic Curie-Weiss model (i.e. the model with $ \beta > 0 $ which is the only one that we consider in this article), and, in the antiferromagnetic case where such an extension is impossible ($ \beta < 0 $), ``how much'' one can extend  (see \cite[def. 1.13]{LiggettSteifToth} for a precise meaning). 
\end{itemize}

The articles \cite{ChagantySethuraman1987, JeonCW, LiggettSteifToth, Papangelou} use in an extensive way the existence of a De Finetti measure of exchangeability for the Curie-Weiss spins to tackle natural probabilistic questions (functional CLT, extension to infinite exchangeability, etc.). None of these problems will be treated here, though; we will focus exclusively on a new approach to Theorem~\ref{Theorem:FluctuationsMagClassical}, leaving open the question of a functional CLT or a change of penalisation to future studies (see \S~\ref{Sec:Conclusion} for a list of interesting problems).

\medskip
\subsubsection{A quick overview of the philosophy}

In a broad sense, exchangeability is concerned with the action of a particular group or algebra, usually infinite. This action leaves invariant the law of the sequence of random variables, which, by duality between functions and measures, is nothing but a probabilistic manifestation of representation theory: the measure underlying the sequence can be disintegrated with the extremal points of the convex set of measures in a similar way representations are decomposed by means of irreducible representations\footnote{
Characters of representations (i.e. their traces) are harmonic functions for a relevant Laplacian, and a mixture of measures can be seen as a decomposition of harmonic functions on the Martin boundary of an underlying Markov field, see e.g. \cite{FollmerMartin}.
}. Problems having such group actions hidden in their core are famously known to be solvable; this is the theory of \textit{integrable systems}. The recent field of integrable probability is another instance of such a phenomenon in probability theory\footnote{
As noted by A. Borodin and V. Gorin \cite{BorodinGorinSurvey}
\begin{quote}
(...) the ``solvable'' or \textit{integrable} members of this class [of models] should be viewed as projections of much more powerful objects whose origins lie in representation theory. In a way, this is similar to integrable systems that can also be viewed as projections of representation theoretic objects; and this is one reason we use the words integrable probability to describe the phenomenon.
\end{quote}
}. 

\medskip

In some sense that will be made precise in this article, this phenomenon will be reproduced: the knowledge of a structure of mixture will be used to decompose the magnetisation with two levels of randomness, one of the random variables being a sum of i.i.d.'s conditionally to the other one, as in the field of random walks in random environment. We will then perform a probabilistic approximation that will, quite surprisingly, preserve all the information of the phase transition picture, and even more at the level of the speed of convergence (in a continuous or Kolmogorov distance). The Curie-Weiss model is already the mean-field approximation of the Ising model that retains its phase transition aspect; the surrogate random variables that we will introduce constitute a last piece of approximation (arguably, \textit{the} last piece of approximation) that will again conserve this picture.

\medskip

We will describe in detail this philosophy in \S~\ref{Subsec:Theory:GeneralSurrogate}, but we can already give an overview of the procedure. Consider a mixture
\begin{align*}
\Pp_{(X_1, \dots, X_n)} = \int \Pp_{(Y_1(p), \dots, Y_n(p))} \nu_n(dp), \qquad \nu_n(dp) = \Prob{ \Pb_{\!\!n} \in dp}
\end{align*}
with mixing measure $ \nu_n $ and independent random variables $ (Y_k(p))_k $ for all $p$. Then, one has
\begin{align*}
\Pp_{ X_1 + \cdots + X_n} = \int \Pp_{Y_1(p) + \cdots + Y_n(p)} \nu_n(dp)
\end{align*}
and one can study the additional structure that consists in the sum of i.i.d.'s $ S_n(p) := \sum_{k = 1}^n Y_k(p) $ and the ``random environment'' given by the mixture measure $ \nu_n $, or equivalently, the randomisation $ \Pb_{\!\!n} \sim \nu_n $. Suppose now that, for fixed $p$, $ S_n(p) $ can be approximated in distribution by another random variable $ \widetilde{S}_n(p) $, for instance $ \mu_n(p) + \sigma_n(p) G $ with $ G \sim \Ns(0, 1) $ in the continuous setting, or $ \Poisson(\mu_n(p)) $ in the discrete setting (with $ \mu_n(p) := \Esp{ S_n(p) } $, etc.). If one can make sense, quantitatively, of the approximation 
\begin{align*}
\int \Pp_{S_n(p)} \nu_n(dp) \approx \int \Pp_{\widetilde{S}_n(p)} \nu_n(dp)
\end{align*}
then, one can study instead the replacement random variable coming from the CLT, or \textit{CLT surrogate} $ \widetilde{S}_n(\Pb_{\!\! n}) \sim \int \Pp_{\widetilde{S}_n(p)} \nu_n(dp) $. As one can see, this philosophy is very general and applies to a broad class of models enjoying a mixture property (see e.g. a list in \S~\ref{Sec:Conclusion}).

\medskip

The rationale for Theorem~\ref{Theorem:FluctuationsMagClassical} will then take the form of an interplay between the CLT approximation $ S_n(p) \approx \widetilde{S}_n(p) $ and the mixture approximation $ \Pb_{\!\! n} \approx \Pb'_{\!\! n} $ for some $ \Pb'_{\!\! n} $:
\begin{enumerate}

\item when the randomisation $ \Pb_{\!\! n} $ is ``weak'', the CLT approximation dominates and the limit is Gaussian; this corresponds to the case $ \beta < 1 $. In this particular regime, the randomisation is also shown to converge in law to a Gaussian, and the algebraic structure underlying the construction of the CLT surrogate entails that the limiting random variable is the \textit{sum} of these two Gaussians. 

The limiting Gaussian is thus made out of two independent Gaussians, one coming from the CLT for sums of i.i.d.'s and another one coming from the randomisation, a dichotomy that went unnoticed in all previous proofs of Theorem~\ref{Theorem:FluctuationsMagClassical}. This subtle mixture of two sources of randomness with survival at the limit of a part of the disorder (i.e. the randomisation considered as a random environment) is named, in the language of statistical mechanics, a \textit{marginally relevant disordered system}. It is a reminiscence of the standard energy-entropy competition, the randomness being naturally associated to the entropy, see \cite{EllisBook, EllisNewman, EllisNewmanRosen}.

One classical such case concerns the convergence of the rescaled KPZ equation towards the Edwards-Wilkinson universality class, see Remark~\ref{Rk:DominationCLT:Smooth} and \cite{CaravennaSunZygourasUniv, ZygourasKPZmarginally}. To the best of the authors' knowledge, this is the first time that the Gaussian limit in the Curie-Weiss model is analysed through the lens of such a framework which lies at the heart of statistical mechanics. 

\item The case $ \beta \geq 1 $ being non Gaussian, it can be assumed that the randomisation will no longer be marginally relevant (since the sum of i.i.d.'s is approximately Gaussian), and indeed, the surrogate behaviour is dominated by the randomisation with a total disapperance of the Gaussian sum.

\item The transition around $ \beta = 1 $ consists in a competition between CLT and randomisation, i.e. between the disordered environment and the random walk given by the sum of i.i.d.'s. The proportion of both sources of randomness is particularly visible in their contribution to the speed of convergence (in any type of norm), and one has to design a more refined type of surrogate to analyse it correctly, using the known Bernoulli coupling in $p$. We refer to \S~\ref{SubSec:Smooth:Beta=1} for a precise description.
\end{enumerate}


In order to control exactly the error originated from the replacement of the original random variable by its surrogate, the approximation in law must be quantitative. This is the difference between Theorem~\ref{Theorem:FluctuationsMagClassical} and Theorem~\ref{Theorem:MainIntro}: we will not only be concerned with the limits in law but also with their speed of convergence, in the Fortet-Mourier norm whose test function admit a certain degree of smoothness and in the Kolmogorov norm whose test functions are indicators of half infinite intervals of the real line. 

Here, not only can one save a considerable computational effort by using the classical Berry-Ess\'een theorem for sums of i.i.d.'s, but in addition comes an unexpected bonus that arises as a byproduct of the use of such a surrogate: by integrating a fraction of the randomness of the surrogate, the indicator functions in the Kolmogorov distance are replaced by smooth functions. This transfert from randomness to smoothness is a very agreable surprise that reduces the discontinuous norm estimate to a smooth one for a related random variable, allowing thus to bypass the usual pathologies of discontinuous test functions distances (see \S~\ref{Sec:Kol}). This is one of the advantages of the surrogate by exchangeability approach: it does not differentiate between the discontinuous and the continuous probability norms.

\subsection{Organisation of the paper}

The plan of the paper is as follows:
\begin{itemize}

\item In \S~\ref{Sec:Theory}, we describe the De Finetti measure of the Curie-Weiss model (with a proof in Annex~\ref{Sec:IdentityInLaw}) and the general notion of surrogate in probability theory (with early examples).

\item In \S~\ref{Sec:Smooth}, we prove Theorem~\ref{Theorem:FluctuationsMagClassical} with the additional computation of the speed of convergence in a smooth norm.  We perform the same analysis in the Kolmogorov distance in \S~\ref{Sec:Kol}.

\item We conclude with questions of interest and extensions of this approach to Large Deviations or other types of statistical mechanical models in \S~\ref{Sec:Conclusion}.

\item The Annex~\ref{Sec:Estimates} gives some estimates required in the proofs and Annex~\ref{Sec:IdentityInLaw} computes the De Finetti measure of the Curie-Weiss model. We encourage the reader to start with this result to get to know this De Finetti measure.

\end{itemize}

\medskip
\section{General theory}\label{Sec:Theory}

\subsection{Reminders on the De Finetti measure of the Curie-Weiss model}

The celebrated De Finetti theorem \cite{DeFinetti} can be stated as follows:

\begin{theorem}[De Finetti, 1969]
Every exchangeable infinite sequence of random variables is a mixture of an i.i.d. sequence.
\end{theorem}

Here, exchangeability for an infinite sequence is understood in the sense of the action of the inductive limit of the symmetric group, i.e. every \textit{finite} sub-sequence of the sequence is invariant by a (finite) permutation. This notion has been extended in several ways, in a finite-dimensional version by Diaconis-Freedman \cite{DiaconisFreedmanDeFinetti}, using other groups such as the projective limit of symmetric groups by Gnedin-Olshanski \cite{GnedinOlshanski} and orthogonal groups by Olshanski-Vershik \cite{OlshanskiVershik} and Schoenberg \cite{Schoenberg}, etc. See e.g. Accardi \cite{AccardiOnDeFinetti} for an account of subtleties on this theorem and further references, or Diaconis-Freedman \cite{DiaconisFreedmanDeFinetti} for a presentation of the topic in relation with the computation of total variation distances.

\medskip

The Curie-Weiss spins are clearly exchangeable since the measure \eqref{Def:CurieWeiss} is invariant by permutation of the spins. Since $n$ is fixed, the De Finetti theorem \textit{a priori} does not apply, but Lemma~\ref{Lemma:DeFinettiCWspins} shows nevertheless that this is the case. There exists a measure $ \widetilde{\nu}_{n, \beta} : \crochet{0, 1} \to \crochet{0, 1} $ such that 
\begin{align}\label{Eq:DeFinettiCW:Meas:P}
\Pp_n^{(\beta)} \equiv \Pp_{(X_1^{(\beta)}, \dots, X_n^{(\beta)} ) } = \int_{\crochet{0, 1}} \Pp_{ ( X_1(p), \dots, X_n(p) ) } \, \widetilde{\nu}_{n, \beta}(dp),
\end{align}
where $ (X_k(p))_{1 \leq k \leq n} \sim i.i.d. \Ber_{\ensemble{\pm 1}}(p) $.

Classically writing 
\begin{align}\label{Def:TotallyDependentCouplingBer:0}
	X_k(p) = 2 \Unens{U_k < p} - 1, \qquad U_k \sim \Us([0, 1]),
\end{align}
one can rewrite \eqref{Eq:DeFinettiCW:Meas:P} in a more probabilistic way using a random variable $ \widetilde{V}_{n, \beta} \sim \widetilde{\nu}_{n, \beta} $ independent of $ (X_k(p))_{1 \leq k \leq n, p \in [0, 1]}  $. We then have the randomisation equality
\begin{align}\label{Eq:DeFinettiCW:RV:P}
\hspace{-0.4cm} (X_1^{(\beta)}, \dots, X_n^{(\beta)} )   
              \eqlaw ( X_1(\widetilde{V}_{n, \beta}), \dots, X_n(\widetilde{V}_{n, \beta}) ) 
               = \prth{2 \Un_{\{ U_1 < \widetilde{V}_{n, \beta} \} } - 1, \dots, 2 \Un_{\{ U_n < \widetilde{V}_{n, \beta} \} } - 1  }.
\end{align}
%
%
%

The measure $ \widetilde{\nu}_{n, \beta} $ (or the random variable $ \widetilde{V}_{n, \beta} $) is well-known for the Curie-Weiss model. It is given by (see \cite[thm. 5.6, (164), (165)]{KirschMoments} or Lemma~\ref{Lemma:DeFinettiCWspins} for a derivation)
\begin{align}\label{Eq:DeFinettiMeasureP}
\widetilde{\nu}_{n, \beta}(dp) = \widetilde{f}_{n, \beta}(p) dp, \qquad \widetilde{f}_{n, \beta}(p) := \frac{1}{\Ze_{n, \beta} }  e^{ -\frac{n}{2\beta} \Argtanh(2p - 1)^2 - (n/2 + 1) \ln(1 - (2p - 1)^2) },
\end{align}
with an explicit renormalisation constant $ \Ze_{n, \beta} $ defined by the equality $ \int_0^1 \widetilde{f}_{n, \beta}(p) dp = 1 $. We recall here that $ \Argtanh(x) = \frac{1}{2} \log\abs{\frac{1 + x}{1 - x}} $ for $ \abs{x} < 1 $. For the sake of completeness, we give a proof of this result in Annex~\ref{Sec:IdentityInLaw}.

\medskip

If instead of considering the parameter $ p \in [0, 1] $ of the Bernoulli random variables, one encodes the De Finetti measure with the expectation parameter $ t := 2p - 1 \in [-1, 1] $, one gets 
\begin{align}\label{Eq:DeFinettiMeasureT}
\nu_{n, \beta}(dt) = f_{n, \beta}(t) dt, \qquad f_{n, \beta}(t) := \frac{1}{\Ze_{n, \beta} }  e^{ -\frac{n}{2\beta} \Argtanh(t)^2 - (n/2 + 1) \ln(1 - t^2) },
\end{align}
and a randomisation by an independent random variable $ V_{n, \beta} \sim \nu_{n, \beta} $.

\subsection{Surrogate random variables in probability theory}\label{Subsec:Theory:GeneralSurrogate}

Consider the following classical problem in extreme value theory: compute the fluctuations of $ \Mb_{\! n} := \max_{1 \leq k \leq n} Z_k $ for independent random variables $ (Z_k)_{k \geq 1} $ when $ n \to +\infty $. One way to proceed is to note that
\begin{align}\label{Eq:MaxToSum}
\ensemble{ \Mb_{\! n} \leq x } = \ensemble{ \forall k \leq n, \, Z_k \leq x } = \ensemble{ \sum_{k = 1}^n \Unens{ Z_k > x} = 0 }.
\end{align}

The problem amounts thus to analyse the fluctuations of the parametric random variable 
\begin{align*}
S_n(x_n) := \sum_{k = 1}^n \Unens{ Z_k > x_n}, \qquad x_n := \mu_n  + \sigma_n x, 
\end{align*}
for given numbers $ \mu_n, \sigma_n $ that one has to tune in order to get the correct result. Since $ S_n(x_n) $ is a sum of independent $ \ensemble{0, 1} $-Bernoulli random variables, one can proceed to a Poisson approximation $ S_n(x_n) \approx \Poisson(f(x)) $, using for instance the Chen-Stein method \cite{Feidt}, resulting in:
\begin{align*}
\Prob{ \Mb_{\! n} \leq x_n } = \Prob{ S_n(x_n) = 0 } = \Prob{ \Poisson(f(x)) = 0 } + o(1) = e^{- f(x) } + o(1).
\end{align*}

By doing so, one has replaced the estimation of a non linear functional of $ (Z_k)_k $, a maximum, by a simpler problem, the estimation of a sum. Such a sum is a \textit{surrogate} random variable. Its study is equivalent to the original problem while being arguably simpler. 

The equality \eqref{Eq:MaxToSum} allows to use a strict equality to replace the original problem by the surrogate problem, and the approximation is only performed at the level of $ S_n(x_n) $, but one could reverse the steps or add an additional approximation step in between (replacing the equality \eqref{Eq:MaxToSum} by an approximation) as long as the original problem is not fundamentally impacted. This is what we perform now.

\subsection{Surrogate magnetisation inequalities}

We recall the following inequality obtained with Stein's method and zero-bias transform (see \cite[(13)]{GoldsteinReinert}\footnote{There seems to be a typo in this equation. This is corrected in \cite[thm. 3.29]{Ross}. Note also that we have done a linear rescaling of the test function.} or \cite[thm. 3.29]{Ross}) valid for all $ h \in \Ce^1 $ with $ \norm{h}_\infty, \norm{h'}_\infty < \infty $
\begin{align}\label{Ineq:BerryEsseen:Smooth}
\abs{\Esp{h(S_n) } - \Esp{h(\sigma_n G + \mu_n) } } \leq C \norm{h'}_\infty,
\end{align}
where $ S_n = \sum_{k = 1}^n Z_k $ with i.i.d. $ (Z_k)_k $ satisfying $ \Esp{\abs{Z}^3} < \infty $, $ \sigma_n^2 := \Var(S_n) = n \Var(Z) $, $ \mu_n := \Esp{S_n} = n \Esp{Z}, G \sim \Ns(0, 1) $ here and throughout this paper, and $C$ is an absolute constant. We see in particular that rescaling linearly $ h \leftarrow h(\frac{\cdot}{\sqrt{n} }) $ gives a speed of convergence in $ O(1/\sqrt{n}) $ (Berry-Ess\'een theorem).

Note that other identities using stronger conditions on the functional space defining the norm and stronger moments conditions allow for stronger speed of convergence (see e.g. \cite[cor. 3.1]{GoldsteinReinert}).

\medskip

In the particular case of Bernoulli $ \ensemble{\pm 1} $ random variables $ (B_k)_k $ of parameter $ p := \Prob{B = 1} $, one has $ \Esp{B} = 2p - 1 $ and $ \Var(B) = 4 p(1 - p) $, thus
\begin{align*}
\abs{\Esp{h(S_n) } - \Esp{h\prth{ \sqrt{n} \, 2\sqrt{p(1 - p)} G + (2p - 1)n } } } \leq C \norm{h'}_\infty.
\end{align*}

Now, taking $ p $ at random with a law $ \nu $ and writing $ S_n(p) $ to mark the dependency, one gets
\begin{align*}
\delta_n(h) & := \abs{\int_{[0, 1]}\Esp{h(S_n(p)) } \nu(dp) - \int_{[0, 1]}\Esp{h\prth{ \sqrt{n} \, 2\sqrt{p(1 - p)} G + (2p - 1)n } }\nu(dp) } \\
            & \leq \int_{[0, 1]} \abs{\Esp{h(S_n(p)) }  - \Esp{h\prth{ \sqrt{n} \, 2\sqrt{p(1 - p)} G + (2p - 1)n } } } \nu(dp) \\
            & \leq C \norm{h'}_\infty \int_{[0, 1]} \nu(dp) = C \norm{h'}_\infty. 
\end{align*}

Taking $p$ distributed as 
\begin{align}\label{Def:PnBeta}
\Pb_n^{(\beta)} \sim \widetilde{\nu}_{n, \beta}
\end{align}
one finally gets with \eqref{Eq:DeFinettiCW:Meas:P} and \eqref{Eq:DeFinettiCW:RV:P}
\begin{align*}
\abs{ \Esp{ h\prth{ M_n^{(\beta)} } } - \Esp{ h\prth{  \sqrt{n} \,  G \times 2 \sqrt{\Pb_n^{(\beta)}(1 - \Pb_n^{(\beta)})}   +  n \times (2\Pb_n^{(\beta)} - 1)  } } } \leq C \norm{h'}_\infty. 
\end{align*}

One can perform a last bit of \textit{massaging} to this expression. Define the random variable
\begin{align}\label{Def:TnBeta}
\Tb_{\! n}^{(\beta)} := 2\Pb_n^{(\beta)} - 1  \sim \nu_{n, \beta}
\end{align}
which corresponds to the parametrisation of the De Finetti measure $ \nu_{n, \beta} $ defined in \eqref{Eq:DeFinettiMeasureT} as opposed to the one defined in \eqref{Eq:DeFinettiMeasureP}. Noting that $ p(1 - p) = \frac{1 - t^2}{4} $, one defines the \textit{surrogate magnetisation} by
\begin{align}\label{Def:SurrogateMnBeta}
\Me_n^{(\beta)} := \sqrt{n} \, G \, \sqrt{1 - ( \Tb_{\! n}^{(\beta)} )^2 } + n \,  \Tb_{\! n}^{(\beta)}, 
\end{align}
so that
\begin{align}\label{Ineq:RandomisedBerryEsseen}
\abs{ \Esp{ h\prth{ M_n^{(\beta)} } } - \Esp{ h\prth{ \Me_n^{(\beta)} } } } \leq C \norm{h'}_\infty.
\end{align}

Using \eqref{Ineq:RandomisedBerryEsseen} in conjunction with the triangle inequality yields with $ \Zb_{\! n, \beta} \sim \Ns(0, n/(1 - \beta) ) $
\begin{align*}
\abs{ \Esp{ h\prth{ M_n^{(\beta)} } } - \Esp{ h\prth{  \Zb_{\! n, \beta} } } }  & \leq  \abs{ \Esp{ h\prth{ M_n^{(\beta)} } } - \Esp{ h\prth{ \Me_n^{(\beta)} } } }  + \abs{ \Esp{ h\prth{ \Me_n^{(\beta)} } } - \Esp{ h\prth{ \Zb_{\! n, \beta} } } } \\
                 & \leq C \norm{h'}_\infty  + \abs{ \Esp{ h\prth{ \Me_n^{(\beta)} } } - \Esp{ h\prth{ \Zb_{\! n, \beta} } } }.
\end{align*}

It is thus enough to control the convergence of the surrogate random variable $ \Me_n^{(\beta)} $ towards its limit (up to a rescaling) to obtain the speed of convergence of the original random variable.

The justification for Theorem~\ref{Theorem:FluctuationsMagClassical} relies then entirely on the fact that the structure of $ \Me_n^{(\beta)} $ defined in \eqref{Def:SurrogateMnBeta} is particularly simple to understand since $ G $ is of order $1$ and $ \Tb_{\! n}^{(\beta)} =: \cos ( \Thetab_{\! n}^{(\beta)} ) \in [-1, 1] $:
\begin{enumerate}

\item when $ \Tb_{\! n}^{(\beta)} $ converges to $ 0 $, the first term is approximately Gaussian after rescaling by $ \sqrt{n} $ and one needs to study the behaviour of $ \Tb_{\! n}^{(\beta)} \sqrt{n} $ which will be shown to be Gaussian too (this corresponds to $ \beta < 1 $), 

\item when $ \Tb_{\! n}^{(\beta)} $ tends to $\pm 1$, one needs to rescale by $n$ and the limit will come from the last term in \eqref{Def:SurrogateMnBeta} (this corresponds to $ \beta > 1 $),

\item last, when both terms are of the same order, i.e. $ \sqrt{n} \sin( \Thetab_{\! n}^{(\beta)} )  \approx n \cos (\Thetab_{\! n}^{(\beta)})  $ or equivalently $ \tan(\Thetab_{\! n}^{(\beta)})  = O_\Pp(\sqrt{n}) $, the analysis has to be refined and a non standard limit can emerge.

\end{enumerate}


Compared with the expository case of \S~\ref{Subsec:Theory:GeneralSurrogate}, one has defined the surrogate by means of an initial (fundamental) inequality and added an additional (non fundamental) inequality to use it, the approximation coming then next. 


\begin{remark}\label{Rk:PoissonSurrogate}
In the proof of Theorem~\ref{Theorem:MagnetisationBetaEqual1}, we will modify the surrogate \eqref{Def:SurrogateMnBeta} using the classical coupling \eqref{Def:TotallyDependentCouplingBer:0} between Bernoulli random variables, due to a discrepancy in the desired speed of convergence. This shows that several surrogates are available, depending on the goal to reach. In particular, since one deals with randomised sums of $ \ensemble{0, 1} $-Bernoulli random variables, one can also replace the fundamental inequality \eqref{Ineq:BerryEsseen:Smooth} with the following Poisson approximation inequality \cite[\S~4]{Ross}
\begin{align*}
\abs{\Esp{ h(S_n) } - \Esp{ h( \Poisson( n p ) ) } } \leq C \norm{  h(\cdot + 1) - h }_\infty, 
\end{align*}
in which case the associated surrogate random variable becomes
\begin{align*}
\widehat{\Me}_n^{(\beta)} := \Poisson\prth{ n \Pb_n^{(\beta)} }.
\end{align*}

Such a surrogate random variable with values in $ \Nn $ could prove useful to obtain a local CLT for the total magnetisation in the vein of \cite{FleermannKirschToth}.
\end{remark}

\medskip
\section{Application to the Curie-Weiss magnetisation in smooth distance}\label{Sec:Smooth}
\medskip

\subsection{The case $ \beta < 1 $}\label{SubSec:Smooth:Beta<1}

Define $  \Zb_{\!\beta} \sim \Ns\prth{ 0, \frac{1}{1 - \beta} } $. 


\begin{theorem}[Fluctuations of the \textit{unnormalised} magnetisation for $ \beta < 1 $]\label{Theorem:MagnetisationBetaSmaller1}
If $ \beta < 1 $, one has for all $ h \in \Ce^1 $ with $ \norm{h}_\infty, \norm{h'}_\infty < \infty $
\begin{align}\label{Eq:Beta<1:Speed}
\abs{ \Esp{ h\prth{ \frac{M_n^{(\beta )}}{\sqrt{n}} } } -  \Esp{ h\prth{ \Zb_{\!\beta} } } } \leq C \frac{ \norm{h'}_\infty }{\sqrt{n} }  +  \frac{1}{n} \prth{ \frac{ \beta }{ 1 - \beta }  \norm{h'}_\infty  +  C(\beta)  \norm{h}_\infty } 
\end{align}
for explicit constants $ C, C(\beta) > 0$. 
\end{theorem} 


\begin{proof} 
Rescaling $ M_n^{(\beta)} $ by $ \sqrt{n} $ amounts to do $ h \leftarrow h(\frac{\cdot}{\sqrt{n}}) $, thus $ h' \leftarrow \frac{1}{\sqrt{n} } h'(\frac{\cdot}{\sqrt{n}}) $. Substituting in \eqref{Ineq:RandomisedBerryEsseen} yields
\begin{align*}
\abs{ \Esp{ h\prth{ \frac{M_n^{(\beta)}}{\sqrt{n} } } } - \Esp{ h\prth{ \frac{\Me_n^{(\beta)}}{\sqrt{n} } } } } \leq \frac{C}{\sqrt{n} } \norm{h'}_\infty
\end{align*}
and the triangle inequality implies then
\begin{align}\label{Ineq:MagSpeedWithSurrogate}
\abs{ \Esp{ h\prth{ \frac{M_n^{(\beta)}}{\sqrt{n} } } } - \Esp{ h\prth{ \Zb_{\!\beta} } } } \leq \frac{C}{\sqrt{n} } \norm{h'}_\infty + \abs{ \Esp{ h\prth{ \frac{\Me_n^{(\beta)}}{\sqrt{n} } } } - \Esp{ h\prth{ \Zb_{\!\beta} } } }.
\end{align}

Note that we do not get better than the usual Berry-Ess\'een bound for the magnetisation due to the term $ \frac{C}{\sqrt{n} } \norm{h'}_\infty $. So far, we are focused on the validity of the approximation, i.e. we want to prove that $ \Me_n^{(\beta)} / \sqrt{n} $ converges in law to a Gaussian, with speed in the Fortet-Mourier norm of order at least $ \frac{1}{\sqrt{n}} $. We thus define
\begin{align}\label{Def:SurrogateDisanceToGaussian}
\widetilde{\delta}_n(h)  := & \abs{ \Esp{ h\prth{ \frac{\Me_n^{(\beta)}}{\sqrt{n} } } } - \Esp{ h\prth{ \Zb_{\!\beta} } } }. 
\end{align}

Define 

$$
\Xb_{\! n, \beta} :=  \sqrt{n} \, \Tb_{\! n}^{(\beta)}, \quad 
G_\beta \sim \Ns\prth{ 0, \frac{\beta}{1 - \beta} }.
$$

Supposing that
\begin{align}\label{CvLaw:XnBeta}
\Xb_{\! n, \beta} \cvlaw{n}{+\infty} G_\beta,  
\end{align}
one gets $ G \sqrt{1 - (\Xb_{\! n, \beta})^2/n } \sim_{n \to +\infty } G \, \sqrt{1 - G_\beta^2/n } \to G $. Hence the final decomposition of $ \Zb_{\!\beta} $
\begin{align}\label{Eq:MarginallyRelevantDecomposition}
G + G_\beta \eqlaw \Zb_{\!\beta},
\end{align}
with $ G $ independent of $ G_\beta $ since $ G $ and $ \Xb_{\! n, \beta}$ are independent (hence so are $ G $ and $ G_\beta $). 
The law of $ \Xb_{\! n, \beta} := \sqrt{n} \, \Tb_{\! n}^{(\beta)} $ is given by the rescaling of $\nu_{n, \beta}$ in \eqref{Eq:DeFinettiMeasureT}:
\begin{align*}
\Prob{ \Xb_{\! n, \beta} \in dt} & = f_{n, \beta}\prth{ \frac{t}{\sqrt{n} } } \frac{dt}{\sqrt{n} }  \\
              & =  \frac{1}{ \sqrt{n}\,\Ze_{n, \beta} }  e^{ -\frac{n}{2\beta} \Argtanh(t/\sqrt{n})^2 - (n/2 + 1) \ln(1 - t^2/n) } \Unens{ \abs{t} \leq \sqrt{n} } dt.
\end{align*}

Using a Taylor expansion in $0$, one easily gets
%
%
%
$$
\frac{n}{2\beta} \Argtanh\prth{ \frac{t}{\sqrt{n}} }^2 + \prth{ \frac{n}{2} + 1 } \ln\prth{1 - \frac{t^2}{n} }  = \prth{ \frac{1 - \beta }{\beta} } \frac{t^2}{2} + O\prth{ \frac{t^4}{n} }.
$$ 

One can moreover show that (see \cite{KirschMoments})
\begin{align*}
\sqrt{n}\,\Ze_{n, \beta} \tendvers{n}{+\infty} \sqrt{2\pi } \times \sqrt{ \frac{\beta}{1 - \beta} },
\end{align*}
which implies \eqref{CvLaw:XnBeta} and would imply 
\begin{align}\label{EqConj:MagSurrogateBeta<1}
\frac{\Me_n^{(\beta)} }{\sqrt{n} } =  G \, \sqrt{1 - ( \Tb_{\! n}^{(\beta)} )^2 } + \sqrt{n} \,  \Tb_{\! n}^{(\beta)} \cvlaw{n}{+\infty} \Zb_{\!\beta}  
\end{align}
with an additional dominated convergence. This is what we prove now. We have
\begin{align*}
\widetilde{\delta}_n(h)  := & \abs{ \Esp{ h\prth{ \frac{\Me_n^{(\beta)}}{\sqrt{n} } } } - \Esp{ h\prth{ G + G_\beta } } } \\
              = \, & \abs{ \Esp{ h\prth{ G \sqrt{ 1 - \frac{ (\Xb_{\! n, \beta})^2}{n} } + \Xb_{\! n, \beta} } } - \Esp{ h\prth{ G + G_\beta } } } \\
              = \, &\abs{ \int_{[-\sqrt{n}, \sqrt{n}]}  \Esp{ h\prth{ G \sqrt{1 - \frac{x^2}{n}} + x  } } f_{n, \beta}\prth{ \frac{x}{\sqrt{n} } } \frac{dx}{\sqrt{n} }  - \Esp{ h\prth{ G + G_\beta } } }  \\
              =: & \abs{ \int_{[-\sqrt{n}, \sqrt{n}]} h_n(x) g_n(x) dx - \int_\Rr h_G(x) g_\beta(x) dx } 
\end{align*}
with 
$$
h_n(x)  := \Esp{ h\prth{ G \sqrt{1 - \frac{x^2}{n}} + x  } },\,
h_G(x)  := \Ee \, \prth{ \vphantom{a^{a^a} } \,  h\, \prth{ \, G + x  } }
$$
and
$$
g_n(x)  :=  \frac{1}{\sqrt{n} } f_{n, \beta}\prth{ \frac{x}{\sqrt{n} } }, \,
g_\beta(x)  :=  \sqrt{ \frac{1 - \beta}{2 \pi \beta} } e^{ - \frac{1 - \beta}{ \beta} \frac{x^2}{2} }.
$$ 

Note that we have used a coupling of $ \Me_n^{(\beta) } $ and $ \Zb_{\!\beta} $ by supposing that the random variable $ G $ that is used in each random variable is the same. Such a coupling is always possible, and this is an important feature of the proof.

\medskip

One has moreover 
\begin{align*}
\abs{ \int_{[-\sqrt{n}, \sqrt{n}]} h_n g_n  - \int_\Rr h_G g_\beta } & \leq \abs{ \int_{[-\sqrt{n}, \sqrt{n}]} ( h_n g_n  - h_G g_\beta )  } + \abs{ \int_{ \Rr \setminus [-\sqrt{n}, \sqrt{n}] } h_G g_\beta }
\end{align*}
with 
\begin{align*}
\abs{ \int_{ \Rr \setminus [-\sqrt{n}, \sqrt{n}] } h_G g_\beta } & \leq \norm{h_G}_\infty \int_{ \Rr \setminus [-\sqrt{n}, \sqrt{n}] } g_\beta \leq  \norm{h}_\infty \int_{ \Rr \setminus [-\sqrt{n}, \sqrt{n}] } g_\beta \\
               & = \norm{h}_\infty \Prob{ \abs{G_\beta} \geq \sqrt{n} } \\
               & \leq \norm{h}_\infty \Esp{ \abs{G_\beta}^{2k} } n^{-k}, \qquad \forall k \geq 1
\end{align*}
using Markov's inequality. Note that the true value of the Gaussian tail is $ \Prob{ \abs{ G_\beta } \geq x } \leq \frac{ 1 }{ \sqrt{ 2\pi x \sigma_\beta^2 } } \exp\prth{ -\frac{x^2}{2 \sigma_\beta^2 } } $, hence $ \Prob{ \abs{G_\beta} \geq \sqrt{n} } = O(e^{-n/2}) $ since $ \beta < 1 $, and this power bound is small enough for our purposes.

The first integral can be estimated writing
\begin{align*}
\abs{ \int_{[-\sqrt{n}, \sqrt{n}]} ( h_n g_n  - h_G g_\beta )  } & = \abs{ \int_{[-\sqrt{n}, \sqrt{n}]} \crochet{ \, h_n (g_n - g_\beta) +  (h_n - h_G) g_\beta }  } \\
                & \leq  \int_{[-\sqrt{n}, \sqrt{n}]} \abs{  h_n (g_n - g_\beta) } +  \abs{ \purple{ \int_{[-\sqrt{n}, \sqrt{n}]}  (h_n - h_G) g_\beta } }  \\
                & \leq \blue{\norm{h_n}_\infty \int_{[-\sqrt{n}, \sqrt{n}]} \abs{ g_n - g_\beta }  } + \norm{ h_n - h_G }_\infty \int_{[-\sqrt{n}, \sqrt{n}]} g_\beta.
\end{align*}

It is clear that $ \abs{h_n(x)} \leq \norm{h}_\infty $. One can thus bound $ \norm{h_n - h_G}_\infty $ by $ 2 \norm{h}_\infty $ but since $ \int_{[-\sqrt{n}, \sqrt{n}]} g_\beta = \Prob{ \abs{G_\beta} \leq \sqrt{n} } = O( 1 - e^{-n/2} ) = O(1) $ which does not tend to $0$, one must work on $ \abs{ \purple{  \int_{[-\sqrt{n}, \sqrt{n}]} (h_n - h_G) g_\beta } } $ direclty. Since one has the same random variables $ G $ and $ G_\beta $, one has a coupling that allows to write 
\begin{align*}
\purple{  \int_{[-\sqrt{n}, \sqrt{n}]} (h_n - h_G) g_\beta  } & = \Esp{ h\prth{ G \sqrt{1 - \frac{G_\beta^2}{n}} + G_\beta  }\Unens{ \abs{G_\beta } \leq \sqrt{n} } - h(G + G_\beta) \Unens{ \abs{G_\beta} \leq \sqrt{n}  } } \\
              & \hspace{-1cm}= \Esp{  G \prth{ \! \sqrt{1 - \tfrac{G_\beta^2}{n}} -1} \Unens{ \abs{G_\beta } \leq \sqrt{n}  }  h'\prth{ \! G_\beta + G + U  \, G \prth{\! \sqrt{1 - \tfrac{G_\beta^2}{n} } \, -1 } \! }  \! }
\end{align*}
with $ U \sim \Us([0, 1]) $ independent of $ (G, G_\beta) $. Using $ x_+ := \max\{x, 0\} = x\Unens{x \geq 0} $ and $ \ensemble{ \abs{G_\beta} \leq \sqrt{n} } = \ensemble{ 1 - G_\beta^2/n \geq 0 } $, one then gets
\begin{align*}
\abs{\int_{[-\sqrt{n}, \sqrt{n}]} (h_n - h_G) g_\beta} & \leq \Esp{ \abs{ G \prth{\sqrt{\prth{ 1 - \frac{G_\beta^2}{n} }_+ }-1 } } }\, \norm{h'}_\infty  \\
             & =  \Esp{ \abs{ G } }\, \norm{h'}_\infty \times \Esp{\abs{\sqrt{\prth{ 1 - \frac{G_\beta^2}{n} }_+ } -1 } } \\
             & \leq \frac{  \Esp{ G_\beta^2 } }{ n} \, \norm{h'}_\infty, 
\end{align*}
where we have used $ 1 - \sqrt{1 - x}  \leq x $ for $ 0 \leq x \leq 1 $ and $ \Esp{ \abs{G} } \leq \sqrt{\Esp{ G^2 } } = 1 $.

%

\medskip

The important quantity to bound is thus
\begin{align*}
\blue{ \norm{h_n}_\infty \int_{[-\sqrt{n}, \sqrt{n}]} \abs{ g_n - g_\beta }  } \leq \norm{h }_\infty \int_{[-\sqrt{n}, \sqrt{n}]} \abs{ g_n - g_\beta }. 
\end{align*}

One thus needs to estimate carefully
\begin{align*}
\delta(g_n, g_\beta) & := \int_{[-\sqrt{n}, \sqrt{n}]} \abs{ g_n - g_\beta } \\
                     & = \int_{[-\sqrt{n}, \sqrt{n}]} \abs{ \frac{1}{\sqrt{n} } f_{n, \beta}\prth{ \frac{x}{\sqrt{n} } }   - \sqrt{ \frac{1 - \beta}{2 \pi \beta} } e^{ - \frac{1 - \beta}{ \beta} \frac{x^2}{2} } } dx \\
                     & = 2 \int_{[0, \sqrt{n}]} \abs{ \frac{1}{\sqrt{n} } f_{n, \beta}\prth{ \frac{x}{\sqrt{n} } }   - \sqrt{ \frac{1 - \beta}{2 \pi \beta} } e^{ - \frac{1 - \beta}{ \beta} \frac{x^2}{2} } } dx 
\end{align*}
by symmetry of $ g_n $ and $ g_\beta $.

Define
%
%
%
$$
\varphi_n(x) := \frac{n}{2\beta} \Argtanh\prth{ \frac{x}{\sqrt{n}} }^2 + \prth{ \frac{n}{2} + 1 } \ln\prth{1 - \frac{x^2}{n} }, \quad
C_\beta  :=   \frac{1 - \beta}{ \beta},  
$$
%
%
%
so that
$$ 
f_{n, \beta}\prth{ \frac{x}{\sqrt{n} } } = \frac{1}{\Ze_{n, \beta} } e^{ - \varphi_n(x)  },  \quad
g_\beta (x)  =  \sqrt{ \frac{C_\beta}{2\pi} } e^{ - C_\beta \frac{x^2}{2} }.
$$ 

One then has 
\begin{align*}
\delta(g_n, g_\beta) & \leq \abs{ \frac{1}{\Ze_{n, \beta} \sqrt{n} } - \sqrt{ \frac{ C_\beta }{2 \pi } } } \int_{[-\sqrt{n}, \sqrt{n}]} e^{-\varphi_n(x)} dx + \sqrt{ \frac{C_\beta}{2 \pi  } } \int_{[-\sqrt{n}, \sqrt{n}]} \abs{ e^{ - \varphi_n(x) } - e^{ - C_\beta \frac{x^2}{2} } } dx.
\end{align*}

The first quantity is
\begin{align*}
\abs{ \sqrt{ \frac{ C_\beta }{2 \pi } } \times \Ze_{n, \beta} \sqrt{n} - 1 } \int_{[-\sqrt{n}, \sqrt{n}]} g_n & = \abs{ \sqrt{ \frac{ C_\beta }{2 \pi } } \times \Ze_{n, \beta} \sqrt{n} - 1 }  \Prob{ \abs{X_n} \leq \sqrt{n} } \\
               & \leq \abs{ \sqrt{ \frac{ C_\beta }{2 \pi } } \times \Ze_{n, \beta} \sqrt{n} - 1 }
\end{align*}
and an analysis of its speed of convergence to 0 is performed in Lemma~\ref{Lemma:AsymptoticRenormConstant}.

The important quantity is the second one. Define
%
%
%
$$
\gamma_n  := \int_{ (-\sqrt{n}, \sqrt{n}) } \abs{ e^{ - \varphi_n(x) } - e^{ - C_\beta \frac{x^2}{2} } } dx  = \int_{ (-\sqrt{n}, \sqrt{n}) } \abs{ e^{ - ( \varphi_n(x) - C_\beta \frac{x^2}{2} ) } - 1 } e^{ - C_\beta \frac{x^2}{2} } dx.
$$ 

Supposing that the quantity inside the absolute value in the second line was bounded by a constant $ D_n $ on $ (-\sqrt{n}, \sqrt{n} ) $, one would have $ \gamma_n = O\prth{ D_n  \int_{(-\sqrt{n}, \sqrt{n})} e^{ - C_\beta \frac{x^2}{2} } dx } = O\prth{D_n  \Prob{ \abs{G_\beta} \leq \sqrt{n} } } = O\prth{ D_n } $. This is not the case as $ \Argtanh^2( \frac{x}{\sqrt{n}} ) \to +\infty $ when $ x \to \sqrt{n} $; nevertheless, the integral is still definite. By symmetry, we will now work on $ [0, \sqrt{n}) $ and use a factor $2$. 

Let $ \varepsilon \in (0, \sqrt{n}) $ to be choosen later. We split $ [0, \sqrt{n}) $ according to  
\begin{align*}
[0,\sqrt{n} ) :=  [0, \sqrt{n} - \varepsilon ) \cup [\sqrt{n} - \varepsilon, \sqrt{n})
\end{align*}
and estimate the integral on each of these subintervals.

%
%
%
%
%
%
%

\medskip

\noindent\textbf{$ \bullet $ \underline{Main interval:}} On $ [0, \sqrt{n} - \varepsilon ) $, the function $ \widetilde{\kappa}_n : x \mapsto \varphi_n(x) - C_\beta \frac{x^2}{2} $ is not monotone. This is mainly due to the fact that the second derivatives of the two functions occuring in the difference do not match in $0$. Indeed, one has $ \varphi_n(0) = \varphi_n'(0) = 0 $ but $ \varphi_n''(0) = C_\beta - \frac{2}{n} $. 

Up to comparing with a triangle inequality $ G_\beta $ and $ G_{\beta, n} \sim \Ns(0, (C_\beta - \frac{2}{n})\inv ) $, one defines then for $ n \geq n_0(\beta) := \pe{ 2 C_\beta\inv } $ 
\begin{equation}\label{Def:CbetaNandKappaN}
C_{\beta, n} := C_\beta - \frac{2}{n}, \quad
\kappa_n(x) :=  \varphi_n(x) - C_{\beta, n} \frac{x^2}{2}.
\end{equation}

The replacement of $ G_\beta \sim \Ns(0, C_\beta\inv) $ by $ G_{\beta, n} \sim \Ns(0, C_{\beta, n}\inv) $ can be done up to $ O(1/n) $. Indeed,
\begin{align*}
\gamma_n & = \int_{ (-\sqrt{n}, \sqrt{n}) } \abs{ e^{ - \varphi_n(x) } - e^{ - C_\beta \frac{x^2}{2} } } dx \\
                & = \int_{ (-\sqrt{n}, \sqrt{n}) } \abs{ e^{ - \varphi_n(x) } - e^{ - C_{\beta,n} \frac{x^2}{2} } + e^{ - C_{\beta,n} \frac{x^2}{2} } - e^{ - C_\beta \frac{x^2}{2} } } dx \\
                & \leq \int_{ (-\sqrt{n}, \sqrt{n}) } \abs{ e^{ - \varphi_n(x) } - e^{ - C_{\beta,n} \frac{x^2}{2} } } dx + \int_{ (-\sqrt{n}, \sqrt{n}) } \abs{ e^{ - C_{\beta,n} \frac{x^2}{2} } - e^{ - C_\beta \frac{x^2}{2} } } dx \\
                & =: 2  \sqrt{ \frac{ 2\pi }{ C_{\beta, n} }}  \, \widetilde{\gamma}_n  + \int_{ (-\sqrt{n}, \sqrt{n}) } \abs{ e^{ - C_{\beta,n} \frac{x^2}{2} } - e^{ - C_\beta \frac{x^2}{2} } } dx
\end{align*}
with
$$ 
\widetilde{\gamma}_n :=  \sqrt{ \frac{ C_{\beta, n} }{ 2\pi } } \int_0^{ \sqrt{n} } \abs{ e^{ - \varphi_n(x) } - e^{ - C_{\beta,n} \frac{x^2}{2} } } dx 
                  = \sqrt{ \frac{ C_{\beta, n} }{ 2\pi } } \int_0^{ \sqrt{n} } \abs{ 1 -  e^{ - \kappa_n(x) }   }  e^{ - C_{\beta,n} \frac{x^2}{2} } dx
$$ 
and
\begin{align*}
\int_{ -\sqrt{n}} ^{\sqrt{n} } \abs{ e^{ - C_{\beta,n} \frac{x^2}{2} } - e^{ - C_\beta \frac{x^2}{2} } } dx & = \int_{ -\sqrt{n}}^{ \sqrt{n} } \abs{ e^{ \frac{x^2}{n} }- 1 } e^{ - C_\beta \frac{x^2}{2} } dx \\
               & \leq 3 \int_\Rr \frac{x^2}{n}   e^{ - C_\beta \frac{x^2}{2} } dx \\
               & = \frac{3}{n}  \frac{\sqrt{2\pi}}{C_\beta^{3/2}}.
\end{align*}

Here, we have used $ \abs{ e^x - 1 } \leq e^{\abs{x} } - 1 \leq  \abs{x} e^{\abs{x} } $ for all $ x \in \Rr $, $ e^{ \frac{x^2}{n} } \leq 3 $ for $ x \in (-\sqrt{n}, \sqrt{n}) $ and the second moment of a Gaussian.

\medskip

The function $ \kappa_n $ thus defined in \eqref{Def:CbetaNandKappaN} is now strictly increasing and positive on $ [0, \sqrt{n} ) $, hence, so is $ 1 - e^{-\kappa_n} $. Moreover, thanks to the matching of the derivatives and the fact that $ \varphi_n'''(0) = 0 $, the Taylor formula with integral remainder gives at the fourth order
\begin{align*}
\kappa_n(x)  =  \frac{x^4}{6} \int_0^1 (1 - \alpha)^3 \kappa_n^{(4)}( \alpha x ) d\alpha = \frac{x^4}{6} \int_0^1 (1 - \alpha)^3 \varphi_n^{(4)}( \alpha x ) d\alpha,
\end{align*}
where $ \kappa_n^{(4)}(x) := \prth{\frac{d}{dx}}^4 \kappa_n(x) $. The only singularity of all the derivatives of $ \kappa_n $ is in $ \sqrt{n} $, which is not in $ (0, \sqrt{n} - \varepsilon) $. One can thus write for all $ x \in [0, \sqrt{n}) $
\begin{align*}
0 \leq \kappa_n(x)  \leq   \frac{x^4}{6} \int_0^1 (1 - \alpha)^3  d\alpha \norm{ \kappa_n^{(4)} \Un_{ [0, \sqrt{n} - \varepsilon)  } }_\infty = \frac{x^4}{24} \norm{ \kappa_n^{(4)} \Un_{ [0, \sqrt{n} - \varepsilon)  } }_\infty =: M_n(\varepsilon) \frac{x^4}{24}.
\end{align*}

Since $ \kappa_n^{(4)} $ is positive and increasing on $ [0, \sqrt{n} - \varepsilon) $, one has
\begin{align*}
M_n(\varepsilon) = \kappa_n^{(4)}(\sqrt{n} - \varepsilon),
\end{align*}
hence
\begin{align*}
\int_{ [0, \sqrt{n}-\varepsilon) } \abs{ 1 - e^{-\kappa_n(x)} } e^{ - C_{\beta, n} \frac{x^2}{2} } dx & \leq \int_{ [0, \sqrt{n}-\varepsilon) } \prth{ 1 - e^{ - M_n(\varepsilon) \frac{x^4}{24} } } e^{ - C_{\beta, n} \frac{x^2}{2} } dx \\
                & \leq \frac{M_n(\varepsilon) }{24} \int_{ [0, \sqrt{n}-\varepsilon) } x^4 e^{ - C_{\beta, n} \frac{x^2}{2} } dx \\
                & \leq \frac{M_n(\varepsilon) }{24} \sqrt{ \frac{ 2\pi }{ C_{\beta, n} } } \, \Esp{ (G_{\beta, n})^4 }   \\
                & = \frac{M_n(\varepsilon) }{8} \sqrt{ \frac{ 2\pi }{ C_{\beta, n} } } \, C_{\beta, n}^{-2},   
\end{align*}
where we have used the scaling of the Gaussian and its fourth moment equal to $ 3 $.

\medskip

Last, a computation with SageMath \cite{SageMath} gives
\begin{align*}
\kappa_n^{(4)}(x) = \varphi_n^{(4)}(x) = \frac{2 }{n \beta} \frac{ P (x/\sqrt{n} ) \Argtanh(x/\sqrt{n}) - Q_{n, \beta}(x/\sqrt{n}) }{ (1 - (x/\sqrt{n})^2)^4 }
\end{align*}
%
%
%
%
%
with explicit polynomials 
\begin{align*}
P(x)  & := 12\, x ( x^2 + 1 ), \\
Q_{n, \beta}(x) & := 3\beta (1 + 2 n\inv) x^4 - 18 ( 1 - \beta  - 2n\inv ) x^2  - (4 - 3\beta - 6 \beta n\inv)  \\
                & =: Q_{ \beta}(x)  + \frac{1}{n} \widetilde{Q}_\beta(x), \\
Q_\beta(x) & := 3\beta   x^4 - 18 ( 1 - \beta    ) x^2  - (4 - 3\beta  ), \\
\widetilde{Q}_\beta(x) & :=  6( \beta x^4 + 6 x^2  + 6 \beta ).
\end{align*}

Define
%
%
%
$t := \frac{\varepsilon}{\sqrt{n} }  \in (0, 1)$, then

\begin{align}\label{Eq:MnEps}
M_n(\varepsilon) & = \varphi_n^{(4)}(\sqrt{n} - \varepsilon) = \frac{2 }{n \beta} \frac{ P\prth{ 1 - t } \Argtanh(1 - t) - Q_{n, \beta}(1 - t) }{ (1 - (1 - t)^2)^4 } \notag \\
             & = \frac{2 }{n \beta} \frac{ P\prth{ 1 - t } \Argtanh(1 - t) - Q_\beta(1 - t) }{ t^4(2 - t)^4  } - \frac{2 }{n^2 \beta} \frac{ \widetilde{Q}_\beta(1 - t) }{ t^4(2 - t)^4 }.
\end{align}

%
%
%
%
%
%
%
%
%
%
%

\medskip

\noindent\textbf{$ \bullet $ \underline{Remaining interval:}} As $ \kappa_n $ is positive on $  [\sqrt{n} - \varepsilon, \sqrt{n}) $, we have $ 1 - e^{-\kappa_n(x) } \leq 1 $ and
\begin{align*}
\int_{ [\sqrt{n} - \varepsilon, \sqrt{n}) } \abs{ 1 - e^{-\kappa_n(x)} } e^{ -  C_{\beta, n} \frac{x^2}{2} } dx & \leq  \int_{ [\sqrt{n} - \varepsilon, \sqrt{n}) }  e^{ -  C_{\beta, n} \frac{x^2}{2} } dx \\
               & =  \varepsilon \int_0^1   e^{ - C_{\beta, n} \frac{(\sqrt{n }- \varepsilon u)^2}{2} } du  \\  
               & \leq  \varepsilon \, e^{-C_{\beta, n} \frac{\sqrt{n} }{2}} \leq 3 \varepsilon \,    e^{-C_\beta  \frac{ \sqrt{n} }{2}},
\end{align*}
where we have used $ (\sqrt{n }- \varepsilon u)^2 = n (1 - t u)^2 \geq \sqrt{n} $ for $ n $ big enough and for all $ t, u \in [0, 1] $, in addition to $ C_{\beta, n} \frac{\sqrt{n}}{2} = C_\beta \frac{\sqrt{n}}{2} - \frac{1}{\sqrt{n}} $ and $ e^{ 1 / \sqrt{n} } \leq 3 $.
%
%
%
%
%

\medskip

\noindent\textbf{$ \bullet $ \underline{General contribution:}} One finally gets 
\begin{align*}
\sqrt{ \frac{ 2\pi }{ C_{\beta, n} } } \, \widetilde{\gamma}_n := \int_{ (0, \sqrt{n}) } \abs{ 1 - e^{-\kappa_n(x)} } e^{ -  C_{\beta, n} \frac{x^2}{2} }   dx  \leq M_n(\varepsilon) \frac{ C_{\beta, n}^{-2} }{8}  +  3 \sqrt{ \frac{ C_{\beta, n} }{2\pi} } \varepsilon \,    e^{-C_\beta  \frac{ \sqrt{n} }{2}}.
\end{align*}

It remains to choose $ \varepsilon $ in order to get the order of convergence. In view of \eqref{Eq:MnEps}, one can take for instance $ t = \frac{3}{4} $, yielding
\begin{align*}
M_n(\tfrac{3}{4} \sqrt{n}) & = \frac{2 }{n \beta} \frac{ P\prth{ \tfrac{1}{4}} \Argtanh(\tfrac{1}{4}) - Q_\beta(\tfrac{1}{4}) }{ (3/4)^4(3/2 - 1)^4  } - \frac{2 }{n^2 \beta} \frac{ \widetilde{Q}_\beta(\tfrac{1}{4}) }{ (3/4)^4(3/2 - 1)^4  } \\
                & =: \frac{K_1(\beta)}{n} + \frac{K_2(\beta)}{n^2}.
\end{align*}

As a result, one gets for explicit constants $ K_3(\beta), K_4(\beta) > 0 $
\begin{align*}
\sqrt{ \frac{ 2\pi }{ C_{\beta, n} } } \, \widetilde{\gamma}_n   \leq   \frac{K_3(\beta)}{n} + \frac{K_4(\beta)}{n^2} + O\prth{ \sqrt{n} \, e^{- C_\beta \frac{ \sqrt{n} }{2}} } = O_\beta\prth{ \frac{1}{n} }.
\end{align*}

\medskip 

\noindent\textbf{$ \bullet $ \underline{Conclusion:}} We have for all $ k \geq 1 $
\begin{align*}
\widetilde{\delta}_n(h)  \leq \frac{ \Esp{ \abs{G_\beta}^{2k} }}{n^k}   \norm{h}_\infty + \frac{  \Esp{ G_\beta^2 } }{ n } \, \norm{h'}_\infty +  \delta(g_n, g_\beta) \norm{h}_\infty
\end{align*}
and using Lemma~\ref{Lemma:AsymptoticRenormConstant}, one has
\begin{equation}\label{Ineq:EstimateDeltaNg}
\delta(g_n, g_\beta)  \leq \abs{ \sqrt{ \frac{ C_\beta }{2 \pi } } \times \Ze_{n, \beta} \sqrt{n} - 1 } + \gamma_n 
                 \leq \frac{C_\beta^{-9/2} }{4 \, n} + 2 \sqrt{ \frac{ 2\pi }{ C_{\beta, n} }} \, \widetilde{\gamma}_n + \frac{3}{n}  \frac{\sqrt{2\pi}}{C_\beta^{3/2}} \notag 
                 = O_\beta\prth{\frac{1}{n} },
\end{equation}
which gives the result.
\end{proof}


\begin{remark}\label{Rk:DominationCLT:Smooth}
The surrogate approach here defined allows to understand in a better way the apparition of the limiting Gaussian random variable. In the case $ \beta < 1 $, the Gaussian CLT is present through the random variable $ G $, and it is the adjunction of the random variable $ G_\beta $ coming from the fluctuations of the randomisation that finally gives $ \Zb_{\!\beta} = G + G_\beta $. It is thus a subtle mixture of the two structures, sums of i.i.d.'s and randomisation, that gives the final distribution in this case. In the language of statistical mechanics of phase transitions, when a disorder is present in a statistical system and has a marginal effect, one talks about a \textit{marginally relevant disordered system} (see e.g. \cite{CaravennaSunZygourasUniv, ZygourasKPZmarginally} in the context of the KPZ equation or random polymers). It is typically the case here with the decomposition $ \Zb_{\!\beta} = G + G_\beta $ (analogous to the convergence of the rescaled KPZ equation to the Edwards-Wilkinson universality class which is Gaussian).

Nevertheless, when looking at the speed in \eqref{Ineq:MagSpeedWithSurrogate} and \eqref{Eq:Beta<1:Speed}, one sees that it is only the Berry-Ess\'een bound for sums of i.i.d.'s that gives its footprint to the first order and not at all the randomisation in this case (the speed coming from the randomisation only appears at the second order). This will be the opposite in the next case $ \beta = 1 - \frac{\gamma}{\sqrt{n}}$ and in particular its special case $\beta = 1$ . The regime $ \beta < 1 $ can thus be considered as the regime where the independent CLT dominates at the level of the speed (and the randomisation is marginally relevant); the regime  $ \beta = 1 - \frac{\gamma}{\sqrt{n}}$ will be the one where the randomisation dominates at the level of the fluctuations (non Gaussian behaviour). From this perspective, the transition is interesting: there is a competition between randomisation and sums of independent random variables. 
\end{remark}

\subsection{The case $ \beta_n = 1 \pm \frac{\gamma}{\sqrt{n}} $, $ \gamma > 0 $}\label{SubSec:Smooth:BetaTrans}

\begin{theorem}[Fluctuations of the \textit{unnormalised} magnetisation for $ \beta_n = 1 - \frac{\gamma}{\sqrt{n}} $, $ \gamma \in \Rr $] \label{Theorem:MagnetisationBeta_n}
Let $ \Fb_{\!\! \gamma} $ be a random variable of law given by
\begin{align*}
\Prob{\Fb_{\!\! \gamma} \in dx } :=  \frac{1}{\Ze_{\Fb_{\! \gamma}}} e^{-\frac{x^4}{12} - \gamma \frac{ x^2}{2}} dx, \qquad  \Ze_{\Fb_{\! \gamma}} := \int_{\Rr}  e^{-\frac{x^4}{12} - \gamma\frac{ x^2}{2}} dx
\end{align*}

Then, for all $ h \in \Ce^1 $ with $ \norm{h}_\infty, \norm{h'}_\infty < \infty $ it holds that
\begin{align}\label{Eq:Beta_n:Speed}
\abs{ \Esp{ h\prth{ \frac{M_n^{(\beta_n )}}{n^{3/4}}  } } -  \Esp{ h( \Fb_{\!\! \gamma} ) } } \leq  \prth{ \frac{  C }{ \sqrt{n} } + O\prth{ \frac{ 1 }{ n^{3/4} } } } \prth{ \vphantom{a^{a^a}} \norm{h}_\infty +  \norm{h'}_\infty },  
\end{align}
where $ C > 0 $ is an explicit constant.
\end{theorem}  


\begin{proof} 
With $ \Tb_{\! n}^{(\beta)} $ defined in \eqref{Def:TnBeta}, set
\begin{align}\label{Def:FnGamma}
\Fb_{\!\! n, \gamma}  :=  n^{1/4} \, \Tb_{\! n}^{(\beta_n)}, \qquad \beta_n := 1 - \tfrac{\gamma}{\sqrt{n}}.
\end{align}

Then, the law of $ \Fb_{\!\! n, \gamma} $ is given by 
\begin{align}\label{Def:Law:FnGamma}
\begin{aligned}
\Prob{ \Fb_{\!\! n, \gamma} \in dt} & = f_{n, \beta_n}\prth{ \frac{t}{n^{1/4} } } \frac{dt}{n^{1/4} } \\
                &  =  \frac{1}{ n^{1/4} \,\Ze_{n, \beta_n} }  e^{ -\frac{n}{2 \beta_n} \Argtanh(t/n^{1/4})^2 - (n/2 + 1) \ln(1 - t^2/\sqrt{n}) } \Unens{ \abs{t} \leq n^{1/4} } dt.
\end{aligned}
\end{align}

Randomising \eqref{Ineq:RandomisedBerryEsseen} and rescaling it by a factor $ n^{3/4} $ (which amounts to do $ h \leftarrow h(\frac{\cdot}{n^{3/4}}) $ hence $ h' \leftarrow \frac{1}{n^{3/4} } h'(\frac{\cdot}{n^{3/4}}) $) yields
\begin{align}\label{Eq:RandomisedBeta=1}
\abs{ \Esp{ h\prth{ \frac{M_n^{(\beta_n)}}{n^{3/4} }  } } - \Esp{ h\prth{ \frac{\Me_n^{(\beta_n)}}{n^{3/4} } } } } \leq \frac{C}{n^{3/4} } \norm{h'}_\infty,
\end{align}
with
\begin{align*}
\frac{\Me_n^{(\beta_n)} }{n^{3/4} } =  \frac{G}{ n^{1/4} } \, \sqrt{1 - ( \Tb_{\! n}^{(\beta_n)} )^2 } + n^{1/4} \,  \Tb_{\! n}^{(\beta_n)}, 
\end{align*}
or rewritten to
\begin{align}\label{Eq:MagSurrogateBeta=1}
\frac{\Me_n^{(\beta_n)} }{n^{3/4} } =  \frac{G}{ n^{1/4} } \, \sqrt{1 - \frac{\Fb_{\!\! n, \gamma}^2}{\sqrt{n} } } + \Fb_{\!\! n, \gamma},
\end{align}
where $ G \sim \Ns(0, 1) $ independent of $ M_n^{(\beta_n)} $ and $ \Fb_{\!\! \gamma} $, and the triangle inequality gives then the analogue of \eqref{Ineq:MagSpeedWithSurrogate}
\begin{align}\label{Ineq:MagSpeedWithSurrogate2}
\abs{ \Esp{ h\prth{ \frac{M_n^{(\beta_n)}}{n^{3/4} } } } - \Esp{ h(\Fb_{\!\!\gamma}) } } \leq \frac{C}{n^{3/4} } \norm{h'}_\infty + \abs{ \Esp{ h\prth{ \frac{\Me_n^{(\beta_n)}}{n^{3/4} } } } - \Esp{ h(\Fb_{\!\!\gamma})} }.
\end{align}


Nevertheless, one sees from the expression of \eqref{Eq:MagSurrogateBeta=1} that if $ \Fb_{\!\! n, \gamma} \to \Fb_{\!\!\gamma} $ in law, 
\begin{align*}
\frac{\Me_n^{(\beta_n)} }{ n^{3/4} } & =  \frac{G}{ n^{1/4} }  \sqrt{1 - \frac{\Fb_{\!\! n, \gamma}^2}{\sqrt{n} } }   + \Fb_{\!\! n, \gamma} \\
              & \approx \Fb_{\!\!\gamma} + \frac{G}{ n^{1/4} } \prth{  1 - \frac{\Fb_{\!\!\gamma}^2}{2 \sqrt{n} }  } \\
              & \approx \Fb_{\!\!\gamma} + \frac{G}{n^{1/4} } + O_\Pp\prth{ \frac{1}{\sqrt{n} } }.
\end{align*}

As a result, one will always have at best $ \Esp{ h\prth{ \frac{\Me_n^{(\beta_n))}}{n^{3/4} } } } - \Esp{ h(\Fb_{\!\!\gamma})} = O\prth{ \frac{ \norm{h'}_\infty }{n^{1/4} } } $ which is incompatible with the results of \cite{ChatterjeeShao, EichelsbacherLoewe} that give a speed in $ O\prth{\frac{1}{\sqrt{n} } } $ for the case $ \gamma = 0 $. Such a discrepancy between this result and \eqref{Eq:RandomisedBeta=1} shows that we have used the ``wrong'' random variable to compare to, when using the triangle inequality. One should instead incorporate another random variable at a distance $ \frac{1}{\sqrt{n} } $ to decrease the distance in $ \frac{1}{n^{1/4} } $, possibly at the cost of increasing the distance in $ \frac{1}{n^{3/4}} $ in \eqref{Eq:RandomisedBeta=1}. Such a replacement can be performed by introducing another related surrogate.

\medskip

Recall that for i.i.d.'s $ (B_k)_k $, one has $ S_n(p) := \sum_{k = 1}^n B_k(p) $. A classical representation of $ B_k(p) $ is given by 
\begin{align}\label{Def:TotallyDependentCouplingBer}
B_k(p) := \Unens{U_k < p} - \Unens{U_k > p} = 2 \, \Unens{U_k < p} - 1, \qquad (U_k)_{k \geq 1} \sim \textrm{i.i.d.}\Us\prth{ [0, 1] },
\end{align}
and this representation allows in addition to visualise the randomisation of the parameter $p$ in a functional way.

One can thus use a coupling of $ S_n(p) $ and $ S_n(q) $ using these uniform random variables. This very coupling is said to be \textit{totally dependent} in the sense that these are the same uniform random variables that are used (i.e. $ B_k(q) $ is a measurable function of $ B_k(p) $ and vice versa).

Define for $ p, q \in [0, 1] $
\begin{align*}
\lambda_n & := 2 n (p - q) = \Esp{ S_n(p) - S_n(q) }, \\
S_n(p, q) & := S_n(p) - S_n(q) - \lambda_n, \\
\sigma(p)^2 & := 4 p(1 - p).
\end{align*}

Then,  
\begin{align*}
\abs{ \Esp{ h\prth{ S_n(p) } } - \Esp{ h\prth{ S_n(q) + \lambda_n } } } & \leq \norm{h'}_\infty \Esp{ \abs{ S_n(p, q) } } \\
                & \leq \norm{h'}_\infty \sqrt{ \Esp{ \abs{ S_n(p, q) }^2 } } \\
                & =: \norm{h'}_\infty \sqrt{ n} \sqrt{ \sigma(p)^2 + \sigma(q)^2 - 2 \rho(p, q) },
\end{align*}
with  
\begin{align*}
\centered{B}_k(p) := B_k(p) - \Esp{B_k(p)} = 2\prth{ \Unens{U_k < p} - p }
\end{align*}
and   
\begin{align*}
\rho(p, q) & := \Esp{  \centered{B}(p) \! \centered{B}(q) } =  4 \, \Esp{  \prth{ \Unens{U < p} - p } \prth{ \Unens{U < q} - q } } \\
              & = 4 \, \Esp{ \Unens{ U < p \wedge q } - p \Unens{U < q} - q \Unens{U < p} + p q } = 4 ( p \wedge q - pq), 
\end{align*}
with $p\wedge q := \min\{p, q\}$. We thus have
\begin{align*}
\sigma(p)^2 + \sigma(q)^2 - 2 \rho(p, q) & = \Esp{ \prth{ \centered{B}(p) - \centered{B}(q) }^2 } \\
                & = 4 \prth{ p(1 - p) + q(1 - q) - 2 (p\wedge q - pq)  } \\
                & = 4 \prth{ p + q - 2 \prth{p \wedge q} - \crochet{  p^2 + q^2 - 2pq } } \\
                & = 4 \prth{ \abs{ p - q } - \abs{ p - q }^2 } = 4 \abs{ p - q } (1 - \abs{ p - q }).
\end{align*}

We can now write with $ \lambdab_n := 2 n (\Pb - \Qb) $ and $ (\Pb, \Qb) $ choosen at random independently from $ (U_k)_k $
\begin{align*}
\deltab_n^\gamma(h) & := \abs{ \Esp{ h\prth{ \frac{S_n(\Pb)}{n^{3/4} } } } - \Esp{ h(\Fb_{\!\!\gamma}) } } \\
               & \leq \, \, \abs{ \Esp{ h\prth{ \frac{S_n(\Pb)}{n^{3/4} } } } - \Esp{ h\prth{ \frac{S_n(\Qb) + \lambdab_n }{n^{3/4} } } } } \\
               & \hspace{+1.5cm} + \abs{ \Esp{ h\prth{ \frac{S_n(\Qb) + \lambdab_n }{n^{3/4} } } } - \Esp{ h\prth{ \frac{ 2\, \Qb (1 - \Qb) \sqrt{n} \, G + n (2\Qb - 1) + \lambdab_n }{n^{3/4} } } } } \\
               & \hspace{+4cm} + \abs{ \Esp{ h\prth{ \frac{ 2\, \Qb (1 - \Qb) \sqrt{n} \, G + n (2\Qb - 1) + \lambdab_n }{n^{3/4} } } } - \Esp{ h(\Fb_{\!\!\gamma}) } } \\
               & =: \deltab^{(\gamma, 1)}_n(h) + \deltab^{(\gamma, 2)}_n(h) + \deltab^{(\gamma, 3)}_n(h). 
\end{align*}

For $ \Fb_{\!\! n, \gamma} $ with law given by \eqref{Def:Law:FnGamma}, we define a random variable $ \Fb_{\!\! n, \gamma}' $ with a law and a dependency to $ \Fb_{\!\! n, \gamma} $ to be defined later. Set moreover
$$
2\Pb - 1  := n^{-1/4} \Fb_{\!\! n, \gamma}', \quad
2\Qb - 1 := n^{-1/4} \Fb_{\!\! n, \gamma}. 
$$

Then, one has
\begin{align*}
\frac{ 2\, \Qb (1 - \Qb) \sqrt{n} \, G + n (2\Qb - 1) + \lambdab_n }{n^{3/4} } & = \frac{G}{n^{1/4} } \sqrt{ 1 - \frac{ \Fb_{\!\! n, \gamma}^2 }{\sqrt{n} } } + \Fb_{\!\! n, \gamma} + \frac{ \lambdab_n }{n^{3/4} } \\
                & = \frac{G}{n^{1/4} } \sqrt{ 1 - \frac{ \Fb_{\!\! n, \gamma}^2 }{\sqrt{n} } } + \Fb_{\!\! n, \gamma} + (\Fb_{\!\! n, \gamma}' - \Fb_{\!\! n, \gamma} ),
\end{align*}
since
\begin{align*}
\lambdab_n = 2 n (\Pb - \Qb) =  n \times n^{-1/4} (\Fb_{\!\! n, \gamma}' - \Fb_{\!\! n, \gamma} ).
\end{align*}

One would then want to couple $ (\Fb_{\!\! n, \gamma} , \Fb_{\!\! n, \gamma}')  $ by setting $ \Fb'_{\!\! n, \gamma} - \Fb_{\!\! n, \gamma} = - \frac{G}{ n^{1/4} } $ with the same $ G $ that defines the surrogate magnetisation in $ \deltab^{(\gamma, 3)}_n(h) $, nevertheless, one also needs to remember that $ p, q \in [0, 1] $, hence that $ (p - q) \in [-1, 1] $. One thus sets
\begin{align}\label{Def:CouplingFFprime}
\Fb_{\!\! n, \gamma}' - \Fb_{\!\! n, \gamma} := - \frac{G}{ n^{1/4} } \Unens{ \abs{G} \leq \sqrt{n} },
\end{align}
which defines the law of $ \Fb_{\!\! n, \gamma}' $ with a convolution of the law \eqref{Def:Law:FnGamma} and the truncated Gaussian law. Notice that \eqref{Def:CouplingFFprime} yields
\begin{align*}
2(\Pb - \Qb) = \frac{1}{n^{1/4} } \prth{ \Fb_{\!\! n, \gamma}' - \Fb_{\!\! n, \gamma} } = - \frac{ G }{ \sqrt{n} } \Unens{ \abs{G} \leq \sqrt{n} }.
\end{align*}

\medskip
\noindent $ \bullet $ \textbf{\underline{Bound on $ \deltab^{(\gamma, 1)}_n $:}} One has
\begin{align*}
\deltab^{(\gamma, 1)}_n(h) & \leq  \frac{ \norm{h'}_\infty }{n^{3/4} } \times 2\sqrt{n} \sqrt{ \Esp{ \abs{ \Pb - \Qb } (1 - \abs{ \Pb - \Qb }) } } \\
                  & =  \frac{ \norm{h'}_\infty }{ \sqrt{2 n} } \prth{ 1 + O\prth{ \frac{1}{\sqrt{n} } } }.
\end{align*}

\medskip
\noindent $ \bullet $ \textbf{\underline{Bound on $ \deltab^{(\gamma, 2)}_n $:}} Setting $ g := h( \cdot + \lambdab_n/n^{3/4}) $ and using the Berry-Ess\'een bound \eqref{Ineq:RandomisedBerryEsseen} rescaled by a factor $ n^{3/4} $ still gives
\begin{align*}
\deltab^{(\gamma, 2)}_n(h) \leq  \frac{ C }{ n^{3/4} } \Esp{ \norm{g'}_\infty } \leq \frac{ C }{ n^{3/4} } \norm{h'}_\infty. 
\end{align*}

\medskip
\noindent $ \bullet $ \textbf{\underline{Bound on $ \deltab^{(\gamma, 3)}_n $:}} Now, one has 
\begin{align*}
\frac{ 2\, \Qb (1 - \Qb) \sqrt{n} \, G + n (2\Qb - 1) + \lambdab_n }{n^{3/4} }  = \frac{G}{n^{1/4} } \prth{  \sqrt{ 1 - \frac{ \Fb_{\!\! n, \gamma}^2 }{\sqrt{n} } } - 1 } + \Fb_{\!\! n, \gamma}  + \frac{G}{n^{1/4} } \Unens{ \abs{G} > \sqrt{n} }.
\end{align*}

Set
\begin{align*}
\Fe_{n, \gamma} & := \frac{G}{n^{1/4} } \prth{  \sqrt{ 1 - \frac{ \Fb_{\!\! n, \gamma}^2 }{\sqrt{n} } } - 1 }, \\
\deltab^{(\gamma, 4)}_n(h) & := \abs{ \Esp{ h\prth{ \Fe_{n, \gamma} + \Fb_{\!\! n, \gamma} } } - \Esp{ h(\Fb_{\!\! n, \gamma}) } }, \\
\deltab^{(\gamma, 5)}_n(h) & := \abs{ \Esp{ h\prth{ \Fb_{\!\!n, \gamma} } } - \Esp{ h(\Fb_{\!\!\gamma}) } }, 
\end{align*}
so that
\begin{align*}
\deltab^{(\gamma, 3)}_n(h) & = \abs{ \Esp{ h\prth{ \Fe_{n,\gamma} + \Fb_{\!\! n, \gamma} + \frac{G}{n^{1/4} } \Unens{ \abs{G} > \sqrt{n} } } } - \Esp{ h(\Fb_{\!\!\gamma}) } } \\
                  & \leq \norm{h'}_\infty \Esp{ \frac{\abs{G} }{n^{1/4} } \Unens{ \abs{G} > \sqrt{n} } }  + \deltab^{(\gamma, 4)}_n(h) + \deltab^{(\gamma, 5)}_n(h) \\
                  & \leq \norm{h'}_\infty \frac{ \sqrt{ \Esp{ \abs{G}^2 } } }{n^{1/4}  }  \sqrt{ \Prob{ \abs{G} > \sqrt{n} } } + \deltab^{(\gamma, 4)}_n(h) + \deltab^{(\gamma, 5)}_n(h) \\
                  & \leq \norm{h'}_\infty \frac{ \sqrt{ \Esp{ G^{ 2k } } } }{ n^{k + 1/4} } + \deltab^{(\gamma, 4)}_n(h) + \deltab^{(\gamma, 5)}_n(h)
\end{align*}
for all $ k \geq 1 $ using the triangle, Taylor, Cauchy-Schwarz and Markov inequalities.

\medskip
\noindent $ \bullet $ \textbf{\underline{Bound on $ \deltab^{(\gamma, 4)}_n $:}}
\begin{align*}
\deltab^{(\gamma, 4)}_n(h) & = \abs{ \Esp{ h\prth{ \Fe_{n, \gamma} + \Fb_{\!\! n, \gamma} } } - \Esp{ h(\Fb_{\!\! n, \gamma}) } } \\
                  & \leq \norm{h'}_\infty \Esp{ \abs{\Fe_{n, \gamma}} } \\
                  & \leq \norm{h'}_\infty \frac{ \Esp{ \abs{G} } \Esp{ \Fb_{\!\! n, \gamma}^2 } }{  n^{3/4} }  \leq \norm{h'}_\infty \frac{ \Esp{ \Fb_{\!\! n, \gamma}^2 } }{ n^{3/4} },
\end{align*}
using $ \abs{ 1 - \sqrt{1 - u} } \leq u $ for $ u \in (0, 1) $ and $ \Esp{\abs{G}} \leq \sqrt{\Esp{G^2}} = 1 $.

To achieve to bound $ \deltab^{(4)}_n(h) $, we will now show that $ \Fb_{\!\! n, \gamma}^2 \to \Fb_{\!\!\gamma}^2 $ in law and in $ L^1 $, hence that $  \Esp{ \Fb_{\!\! n, \gamma}^2 } = \Esp{ \Fb_{\!\!\gamma}^2 } + o(1) $.

\medskip
\noindent $ \bullet $ \textbf{\underline{Convergence in law $\Fb_{\!\! n, \gamma}^2 \to \Fb_{\!\!\gamma}^2 $:}} The law of $ \Fb_{\!\! n, \gamma} := n^{1/4} \, \Tb_{\! n}^{(\beta_n)} $ is given by the rescaling of $\nu_{n, \beta_n}$ in \eqref{Eq:DeFinettiMeasureT}~:
\begin{align*}
\Prob{ \Fb_{\!\! n, \gamma} \in dt} & = f_{n, \beta_n}\prth{ \frac{t}{n^{1/4} } } \frac{dt}{n^{1/4} }  \\
                   & =  \frac{1}{ n^{1/4} \,\Ze_{n, \beta_n} }  e^{ -\frac{n}{2\beta_n} \Argtanh(t/n^{1/4})^2 - (n/2 + 1) \ln(1 - t^2/\sqrt{n}) } \Unens{ \abs{t} \leq n^{1/4} } dt.
\end{align*}

Lemma~\ref{Lemma:AsymptoticRenormConstant3} gives
\begin{align*}
n^{1/4} \,\Ze_{n, \beta_n}  = \Ze_{\Fb_{\!\gamma}} + O\prth{ \frac{1}{\sqrt{n} } }
\end{align*}
and a Taylor expansion in $0$ yields with SageMath
%
%
%
%
%
%
%
%
%
%
%
\begin{align}\label{Eq:TaylorRandomisation:beta_n}
\Phi_{n,\gamma}(t) & := \frac{n}{2 \beta_n} \Argtanh\prth{ \frac{t}{n^{1/4}} }^2 + \prth{ \frac{n}{2} + 1 } \ln\prth{1 - \frac{t^2}{\sqrt{n}} } \nonumber \\
                                & = \prth{\frac{1 - \beta_n}{\beta_n} \sqrt{n} - \frac{2}{\sqrt{n}}} \frac{t^2}{2} + \frac{t^4}{12}\prth{ \frac{4 - 3 \beta_n}{\beta_n} - \frac{6}{n} } + O\prth{ \frac{t^6}{n^{3/2}} }  \nonumber \\
                                & =  \prth{ \frac{\gamma}{\beta_n} - \frac{2}{\sqrt{n}}} \frac{t^2}{2} + \frac{t^4}{12}\prth{ \frac{4 - 3 \beta_n}{\beta_n} - \frac{6}{n} } + O\prth{ \frac{t^6}{n^{3/2}} }.  
\end{align}

Since $ \beta_n \to 1 $, this implies the convergence in law $ \Fb_{\!\! n, \gamma} \to \Fb_{\!\!\gamma} $ by looking at the densities, and the result by square integrability of $ \Fb_{\!\!\gamma} $.

\medskip
\noindent $ \bullet $ \textbf{\underline{Bound on $ \deltab^{(\gamma, 5)}_n $:}} In the same vein as for $ \beta < 1 $, one has for all $ \varepsilon \in (0, 1) $ and setting $ \overline{\varepsilon} := 1 - \varepsilon $
\begin{align*}
\deltab^{(\gamma, 5)}_n(h) & := \abs{ \Esp{ h\prth{ \Fb_{\!\! n, \gamma} } } - \Esp{ h(\Fb_{\!\!\gamma}) } } \\
                   & \leq \int_{(-\overline{\varepsilon} n^{1/4}, \overline{\varepsilon} n^{1/4} ) } \abs{h} \abs{ f_{\Fb_{\! n, \gamma} } - f_{\Fb_{\!\!\gamma}} } + \abs{ \Esp{ h(\Fb_{\!\!\gamma}) \Unens{ \abs{\Fb_{\! \gamma}} > \overline{\varepsilon} n^{1/4} } } } \\
                   & \leq \norm{h}_\infty \prth{ \norm{ f_{\Fb_{\! n, \gamma} } - f_{\Fb_{\! \gamma}}  }_{ L^1([-\overline{\varepsilon} n^{1/4}, \overline{\varepsilon} n^{1/4} ]) } + \Prob{ \abs{\Fb_{\!\!\gamma}} > \overline{\varepsilon} n^{1/4} } } \\
                   & \leq \norm{h}_\infty \prth{ \norm{ f_{\Fb_{\! n, \gamma} } - f_{\Fb_{\! \gamma}}  }_{ L^1([-\overline{\varepsilon} n^{1/4}, \overline{\varepsilon} n^{1/4} ] )  } +  \frac{ \Esp{\Fb_{\!\!\gamma}^{4k}} }{(1 - \varepsilon)^{4k} \, n^k } },
\end{align*}
using the triangle inequality and the Markov inequality for all $ k \geq 1 $.

\medskip

We now show that
\begin{align}\label{Ineq:FnToF}
\norm{ f_{\Fb_{\! n, \gamma} } - f_{\Fb_{\!\gamma}}  }_{ L^1([-\overline{\varepsilon} n^{1/4}, \overline{\varepsilon} n^{1/4} ] ) } 
                \leq \deltab^{(\gamma, 6)}_n + \frac{ 2 \gamma \Esp{\Fb_{\!\! \gamma}^2} + 3(\gamma^2 - 2) \Esp{\Fb_{\!\! \gamma}^4}}{ 6 \sqrt{n} } + O\prth{\frac{1}{n} },
\end{align}
where $ \deltab^{(\gamma, 6)}_n $ will be defined in \eqref{Def:DeltaN_gamma_6}.

\medskip

In view of \eqref{Eq:TaylorRandomisation:beta_n}, we introduce the random variable $ \widehat{\Fb}_{\!\! n, \gamma} $ defined by the density
\begin{align*}
f_{ \widehat{\Fb}_{\! n,\gamma} }(x) & := \frac{1}{\widehat{\Ze}_{n,\gamma} } e^{- \widehat{\Phi}_{n,\gamma}(x)}, \qquad \widehat{\Ze}_{n,\gamma} := \int_\Rr e^{- \widehat{\Phi}_{n,\gamma} },  \\
\widehat{\Phi}_{n,\gamma}(x) & := \alpha_n^{(1)} \frac{x^4}{12} + \alpha_n^{(2)} \frac{x^2}{2 }, \quad \alpha_n^{(1)} := \frac{4 - 3 \beta_n}{\beta_n} - \frac{6}{n}, \qquad \alpha_n^{(2)} := \frac{\gamma}{\beta_n} - \frac{2}{\sqrt{n}}.   
\end{align*}

Recalling that $ \Phi_{n, \gamma} $ is defined in \eqref{Eq:TaylorRandomisation:beta_n}, we also define
\begin{align*}
\Phi_\gamma & := \widehat{\Phi}_{\infty,\gamma} = \Phi_{\infty,\gamma} : x \mapsto \frac{x^4}{12} + \gamma \frac{x^2}{2}.
\end{align*}

Then, one can replace $ \Fb_{\!\! \gamma} $ by $ \widehat{\Fb}_{\! n,\gamma} $ up to $ O(1/\sqrt{n}) $ by writing 
\begin{align*}
\norm{ f_{\Fb_{\!  n, \gamma} } - f_{\Fb_{\! \gamma}}  }_{ L^1([-\overline{\varepsilon} n^{1/4}, \, \overline{\varepsilon} n^{1/4} ] ) } 
                    \leq 
                          \norm{ f_{\Fb_{\! n,\gamma} } - f_{\widehat{\Fb}_{\! n,\gamma}} }_{ L^1([-\overline{\varepsilon} n^{1/4}, \, \overline{\varepsilon} n^{1/4} ] )  } 
                         + \norm{ f_{\widehat{\Fb}_{\! n,\gamma}} - f_{\Fb_{\! \gamma}} }_{ L^1(\Rr) }.
\end{align*}

Set
\begin{align}\label{Def:DeltaN_gamma_6}
\deltab_n^{(\gamma,6)} & := \norm{ f_{\Fb_{\! n,\gamma} } - f_{\widehat{\Fb}_{\! n,\gamma}} }_{ L^1([-\overline{\varepsilon} n^{1/4}, \, \overline{\varepsilon} n^{1/4} ] )  }.   
\end{align}

One has
\begin{align*}
\norm{ f_{\widehat{\Fb}_{\! n,\gamma}} - f_{\Fb_{\! \gamma}} }_{ L^1(\Rr) } 
               & := \int_\Rr \abs{ \frac{1}{\widehat{\Ze}_{n, \gamma} } e^{- \widehat{\Phi}_{n, \gamma} } - \frac{1}{\Ze_{\Fb_{\! \gamma}} } e^{- \Phi_\gamma  } } =  \frac{1}{\Ze_{\Fb_{\! \gamma}} } \int_\Rr \abs{ \frac{\Ze_{\Fb_{\! \gamma}}}{\widehat{\Ze}_{n, \gamma} } e^{- \widehat{\Phi}_{n, \gamma} + \Phi_\gamma } - 1}  e^{- \Phi_\gamma  } \\
               & = \frac{1}{\Ze_{\Fb_{\! \gamma}} } \int_\Rr \abs{ \prth{ \frac{\Ze_{\Fb_{\! \gamma}}}{\widehat{\Ze}_{n,\gamma} } - 1} e^{- \widehat{\Phi}_{n, \gamma} + \Phi_\gamma } + e^{- \widehat{\Phi}_{n, \gamma} + \Phi_\gamma } - 1}  e^{- \Phi_\gamma  } \\
               & \leq \abs{ \frac{\Ze_{\Fb_{\! \gamma}} }{\widehat{\Ze}_{n, \gamma} } - 1} \frac{1}{\Ze_{\Fb_{\! \gamma}} } \int_\Rr e^{- \widehat{\Phi}_{n, \gamma} + \Phi_\gamma } e^{- \Phi_\gamma  } + \frac{1}{\Ze_{\Fb_{\! \gamma}} } \int_\Rr \abs{   e^{- \widehat{\Phi}_{n, \gamma} + \Phi_\gamma } - 1}  e^{- \Phi_\gamma  } \\
               & = \frac{1}{\Ze_{\Fb_{\! \gamma}}} \abs{  \Ze_{\Fb_{\! \gamma}} - \widehat{\Ze}_{n, \gamma}  }  
                       + \Esp{ \abs{ e^{ \,\prth{1 - \alpha_n^{(1)}} \frac{ \Fb_{\! \gamma}^4}{12} \, + \,  \prth{\gamma - \alpha_n^{(2)}}\frac{\Fb_{\!\gamma}^2}{ 2 } } - 1 } } \\
               & = \frac{1}{\Ze_{\Fb_{\! \gamma}}} \abs{ \int_\Rr \prth{e^{- \widehat{\Phi}_{n, \gamma}} - e^{-\Phi_\gamma} }  } 
                        + \Esp{ \abs{ e^{ \,\prth{1 - \alpha_n^{(1)}} \frac{ \Fb_{\! \gamma}^4}{12} \, + \, \prth{\gamma - \alpha_n^{(2)}}\frac{\Fb_{\!\gamma}^2}{ 2 } } - 1 } } \\
               & \leq 2\,  \Esp{ \abs{ e^{ \,\prth{1 - \alpha_n^{(1)}} \frac{ \Fb_{\! \gamma}^4}{12} \, + \, \prth{\gamma - \alpha_n^{(2)}}\frac{\Fb_{\!\gamma}^2}{ 2 } } - 1 } }.
\end{align*}

Using  
\begin{align*}
1 - \alpha_n^{(1)} & = \frac{6}{n} + 4 \prth{1 - \frac{1}{\beta_n}} = \frac{4 \gamma}{\sqrt{n}} + O\prth{ \frac{1}{n} }, \\
\gamma - \alpha_n^{(2)} & = \gamma \prth{1 - \frac{1}{\beta_n}} - \frac{2}{\sqrt{n}} = \frac{\gamma^2 - 2}{\sqrt{n}} +  O\prth{ \frac{1}{n} },
\end{align*}
we get
\begin{align*}
\Esp{ \abs{ e^{ \,\prth{1 - \alpha_n^{(1)}} \frac{ \Fb_{\! \gamma}^4}{12} \, + \, \prth{\gamma - \alpha_n^{(2)}}\frac{\Fb_{\!\gamma}^2}{ 2 } } - 1 } }
               & = \Esp{ \abs{ 1 + \frac{4 \gamma}{\sqrt{n}}  \frac{\Fb_{\!\! \gamma}^4}{12} + \frac{(\gamma^2 - 2)}{\sqrt{n}} \frac{\Fb_{\!\! \gamma}^2}{2} - 1 + O\prth{ \frac{\Fb_{\!\! \gamma}^4}{n} } } }  \\
               & =   \frac{ 2 \gamma \Esp{\Fb_{\!\! \gamma}^2} + 3(\gamma^2 - 2) \Esp{\Fb_{\!\! \gamma}^4}}{ 6 \sqrt{n} } +  O\prth{ \frac{1}{n} }.
\end{align*}

We thus have
\begin{align*}
\deltab_n^{(\gamma, 5)}(h)
              \leq \norm{h}_\infty \prth{ \deltab_n^{(\gamma, 6)}(h)  
                   + \frac{ 2 \gamma \Esp{\Fb_{\!\! \gamma}^2} + 3(\gamma^2 - 2) \Esp{\Fb_{\!\! \gamma}^4}}{ 6 \sqrt{n} } +  O\prth{ \frac{1}{n} } 
                   + \frac{ \Esp{\Fb_{\!\! \gamma}^{4k}} }{(1 - \varepsilon)^{4k} \, n^k } }.
\end{align*}
and we now estimate the remaining term.

\medskip
\noindent $ \bullet $ \textbf{\underline{Bound on $ \deltab^{(\gamma, 6)}_n $:}} We have

\begin{align*}
\deltab_n^{(\gamma,6)} & := \norm{ f_{\Fb_{\! n,\gamma} } - f_{\widehat{\Fb}_{\! n,\gamma}} }_{ L^1([-\overline{\varepsilon} n^{1/4}, \, \overline{\varepsilon} n^{1/4} ] )  }   \\
                & = \int_{ -\overline{\varepsilon} n^{1/4} }^{ \overline{\varepsilon} n^{1/4} } \abs{ f_{\Fb_{\!  n, \gamma} } - f_{\widehat{\Fb}_{\!\! n, \gamma}} } = \int_{ -\overline{\varepsilon} n^{1/4} }^{ \overline{\varepsilon} n^{1/4} } \abs{ 1 - \frac{f_{\Fb_{\!  n, \gamma} }}{ f_{\widehat{\Fb}_{\!\! n, \gamma}} } } f_{\widehat{\Fb}_{\!\! n, \gamma}}  \\
                & = \int_{ -\overline{\varepsilon} n^{1/4} }^{ \overline{\varepsilon} n^{1/4} } \abs{ 1 - \frac{ \widehat{\Ze}_{n, \gamma} }{n^{1/4}\Ze_{n, \beta_n} }   e^{- (\Phi_{n, \gamma} - \widehat{\Phi}_{n, \gamma} ) } } f_{\widehat{\Fb}_{\!\! n, \gamma}}  \\
                & = \int_{ -\overline{\varepsilon} n^{1/4} }^{ \overline{\varepsilon} n^{1/4} } \abs{ 1 - e^{- (\Phi_{n, \gamma} - \widehat{\Phi}_{n, \gamma} ) } + \prth{ 1 - \frac{ \widehat{\Ze}_{n, \gamma} }{n^{1/4}\Ze_{n, \beta_n} } }   e^{- (\Phi_{n, \gamma} - \widehat{\Phi}_{n, \gamma} ) } } f_{\widehat{\Fb}_{\!\! n, \gamma}}  \\
                & \leq \abs{  1 - \frac{ \widehat{\Ze}_{n, \gamma} }{n^{1/4}\Ze_{n, \beta_n} } }  \int_{ -\overline{\varepsilon} n^{1/4} }^{ \overline{\varepsilon} n^{1/4} }   e^{- (\Phi_{n, \gamma} - \widehat{\Phi}_{n, \gamma} ) }   f_{\widehat{\Fb}_{\!\! n, \gamma}}    +   \int_{ -\overline{\varepsilon} n^{1/4} }^{ \overline{\varepsilon} n^{1/4} } \abs{ 1 - e^{- (\Phi_{n, \gamma} - \widehat{\Phi}_{n, \gamma} ) }  } f_{\widehat{\Fb}_{\!\! n, \gamma}}   \\
                & \leq \abs{  1 - \frac{ \widehat{\Ze}_{n, \gamma} }{n^{1/4}\Ze_{n, \beta_n} } } \frac{n^{1/4}\Ze_{n, \beta_n} }{ \widehat{\Ze}_{n, \gamma} }  +   \int_{ -\overline{\varepsilon} n^{1/4} }^{ \overline{\varepsilon} n^{1/4} } \abs{ 1 - e^{- (\Phi_{n, \gamma} - \widehat{\Phi}_{n, \gamma} ) }  } f_{\widehat{\Fb}_{\!\! n, \gamma}}.
\end{align*}

One has in addition with $\gamma_n  := \frac{\alpha_n^{(2)}}{ \sqrt{ \alpha_n^{(1)} } }   - \gamma$
\begin{align*}
\frac{\widehat{\Ze}_{n,\gamma}}{\Ze_{\Fb_{\! \gamma}}} & := \frac{1}{\Ze_{\Fb_{\! \gamma}} } \int_\Rr e^{-\widehat{\Phi}_{n,\gamma}} \\
                   & = \frac{1}{\Ze_{\Fb_{\! \gamma}} } \int_\Rr e^{- \frac{x^4}{12}\alpha_n^{(1)} + \alpha_n^{(2)} \frac{x^2}{2 } } dx 
                   =  \frac{1}{\Ze_{\Fb_{\! \gamma}} \prth{ \alpha_n^{(1)} }^{\! 1/4 } } \int_\Rr e^{ - \frac{x^4}{12} - \gamma \frac{x^2}{2} +  \gamma_n \frac{x^2}{2} } dx \\ 
                   & = \frac{1}{ \prth{ \alpha_n^{(1)} }^{\! 1/4 } } \Esp{ e^{ \gamma_n \Fb_{\!\! \gamma}^2 } } 
                   = 1 + \frac{\gamma + (\gamma^2 + 2) \Esp{ \Fb_{\!\! \gamma}^2 }}{\sqrt{n} } + O\prth{\frac{1}{n} },
\end{align*}
since  
\begin{align*}
\gamma_n & := \frac{\alpha_n^{(2)}}{ \sqrt{ \alpha_n^{(1)} } }   - \gamma   = \gamma \prth{ \prth{\frac{1}{1 +  \frac{\gamma}{\sqrt{n}}}  + \frac{2}{\gamma \sqrt{n}} } \prth{ \frac{1}{1 - \frac{4\gamma}{\sqrt{n}} +        O\prth{ \frac{1}{n} } } }^{\! 1/2} - 1 } \\
                  &  = \frac{\gamma^2 + 2}{\sqrt{n}} + O\prth{ \frac{1}{n} } 
\end{align*}
and 
\begin{align*}
\frac{1}{ \prth{ \alpha_n^{(1)} }^{\! 1/4 } } & = \frac{1}{ \prth{1 - \frac{4 \gamma}{\sqrt{n}} + O\prth{ \frac{1}{n} } }^{1/4} } 
                 = 1 + \frac{ \gamma}{\sqrt{n}} + O\prth{ \frac{1}{n} }.
\end{align*}

Lemma~\ref{Lemma:AsymptoticRenormConstant3} gives  
\begin{align*}
\frac{\Ze_{\Fb_{\! \gamma}}}{n^{1/4}\Ze_{n, \beta_n} } = 1 + O\prth{ \frac{1}{\sqrt{n} } },
\end{align*}
implying
\begin{align*}
 1 - \frac{ \widehat{\Ze}_{n,\gamma} }{n^{1/4}\Ze_{n, \beta_n} }   =   1 - \frac{ \widehat{\Ze}_{n,\gamma} }{\Ze_{\Fb_{\! \gamma}} } \frac{\Ze_{\Fb_{\! \gamma}}}{n^{1/4}\Ze_{n, \beta_n} }   = O\prth{\frac{1}{\sqrt{n} } }.
\end{align*}

Define
\begin{align*}
\deltab_n^{(\gamma,7)} 
                 := \int_{ -\overline{\varepsilon} n^{1/4} }^{ \overline{\varepsilon} n^{1/4} } \abs{ 1 - e^{- (\Phi_{n,\gamma} - \widehat{\Phi}_{n,\gamma} ) }  } f_{\widehat{\Fb}_{\! n,\gamma}}
                 = \Esp{ \abs{ 1 - e^{- (\Phi_n - \widehat{\Phi}_n )(\widehat{\Fb}_{\!\! n, \gamma}) }  } \Unens{ \abs{\widehat{\Fb}_{\!\! n, \gamma}} \leq \overline{\varepsilon} n^{1/4} } }
\end{align*}

Then,
\begin{align}\label{Ineq:DeltaN_gamma_6:Intermediate}
\deltab_n^{(\gamma,6)}   
                   \leq   \int_{ -\overline{\varepsilon} n^{1/4} }^{ \overline{\varepsilon} n^{1/4} } \abs{ 1 - e^{- (\Phi_{n,\gamma} - \widehat{\Phi}_{n,\gamma} ) }  } f_{\widehat{\Fb}_{\! n,\gamma}} 
                   + \abs{  \frac{n^{1/4}\Ze_{n, \beta_n} }{ \widehat{\Ze}_{n,\gamma} } - 1 }    
                 =  \deltab_n^{(\gamma,7)} +  O\prth{\frac{1}{\sqrt{n} } }.
\end{align}

\medskip
\noindent $ \bullet $ \textbf{\underline{Bound on $ \deltab^{(\gamma, 7)}_n $:} }
Define
\begin{align*}
\kappab_{n,\gamma}(x) & := \Phi_{n,\gamma}(x) - \widehat{\Phi}_{n,\gamma}(x)   \\
                                        & = \frac{n}{2\beta_n} \Argtanh\prth{ \frac{x}{n^{1/4}} }^2 + \prth{ \frac{n}{2} + 1 } \ln\prth{1 - \frac{x^2}{\sqrt{n}} }  - \alpha_n^{(1)} \frac{x^4}{12} - \alpha_n^{(2)} \frac{x^2}{2 }.
\end{align*}

The Taylor expansion \eqref{Eq:TaylorRandomisation:beta_n} shows that $ \kappab_{n,\gamma}^{(k)}(0) = 0 $ for $ k \in \intcrochet{0, 5} $, hence 
\begin{align*}
\kappab_{n,\gamma}(x)  = \frac{x^6}{5!} \int_0^1 (1 - \alpha)^5 \kappab_{n,\gamma}^{(6)}(\alpha x) d\alpha.
\end{align*}

An analysis of $ \kappab_{n,\gamma}^{(6)} $ with SageMath \cite{SageMath} in the same vein as for the other cases shows that $ \kappab_{n,\gamma}^{(6)} \geq 0 $ on $ (-\overline{\varepsilon} n^{1/4} \,  \overline{\varepsilon} n^{1/4} ) $, is an odd function and is strictly increasing on $ ( 0, \overline{\varepsilon} n^{1/4}) $. As a result, one can write
\begin{align*}
\deltab_n^{(\gamma,7)}  & :=   \int_{ -\overline{\varepsilon} n^{1/4} }^{ \overline{\varepsilon} n^{1/4} } \abs{ 1 - e^{- \kappab_{n,\gamma} }  } f_{\widehat{\Fb}_{\! n,\gamma}} \\ 
               &  = 2 \int_0^{ \overline{\varepsilon} n^{1/4} } \prth{ 1 - e^{- \kappab_{n,\gamma} }  } f_{\widehat{\Fb}_{\! n,\gamma} } \\
         & \leq   \frac{2}{6!} \, \kappab_{n,\gamma}^{(6)}( \overline{\varepsilon} n^{1/4} ) \, \Esp{ \widehat{\Fb}_{\!\! n, \gamma}^6 } 
                 = \frac{2}{6!} \, \kappab_{n,\gamma}^{(6)}( \overline{\varepsilon} n^{1/4} ) \, \Esp{ \Fb_{\!\! \gamma}^6 }\prth{ 1  + O\prth{ \frac{1}{\sqrt{n} } } }.
\end{align*}

Moreover, one can write
\begin{align*}
\kappab_{n,\gamma}(x) = \frac{n}{\beta_n} \, W_1\prth{ \frac{x}{n^{1/4}} }  + n \, W_2\prth{ \frac{x}{n^{1/4}} } + W_3\prth{ \frac{x}{n^{1/4}} } - \frac{4 - 3 \beta_n}{\beta_n} \frac{x^4}{12} - \frac{\gamma}{\beta_n} \frac{x^2}{2},
\end{align*}
with $ W_1, W_2 $ and $W_3$ explicit and indefinitely differentiable on $ (0, \overline{\varepsilon} n^{1/4} ) $. As a result, 
\begin{align*}
\kappab_{n,\gamma}^{(6)}(x) = \frac{1}{\beta_n \sqrt{n} } \, W_1^{(6)}\prth{ \frac{x}{n^{1/4}} } + \frac{1}{\sqrt{n} } \, W_2^{(6)}\prth{ \frac{x}{n^{1/4}} }  + \frac{1}{n^{3/2} } W_3^{(6)}\prth{ \frac{x}{n^{1/4}} }  
\end{align*}
and 
\begin{align*}
\kappab_{n,\gamma}^{(6)}(\overline{\varepsilon}n^{1/4}) = \frac{1}{\beta_n \sqrt{n} } \, W_1^{(6)}\prth{ \overline{\varepsilon} }  + \frac{1}{\sqrt{n} } \, W_2^{(6)}\prth{ \overline{\varepsilon} }  + \frac{1}{n^{3/2} } W_3^{(6)}\prth{ \overline{\varepsilon} }. 
\end{align*}

Choosing $ \varepsilon = \frac{3}{4} $ for instance gives then
\begin{align*}
\kappab_{n,\gamma}^{(6)}(\overline{\varepsilon}n^{1/4}) & = \frac{1}{\beta_n \sqrt{n} }  \, W_1^{(6)}\prth{ \frac{1}{4} }  + \frac{1}{\sqrt{n} }  \, W_2^{(6)}\prth{ \frac{1}{4} } + O\prth{ \frac{1}{n^{3/2} } } \\ 
               & = \frac{1}{\sqrt{n}} \prth{ W_1^{(6)}\prth{ \frac{1}{4} }  + W_2^{(6)}\prth{ \frac{1}{4} } }+ O\prth{ \frac{1}{n } }
\end{align*}
and
\begin{align}\label{Ineq:DeltaN_gamma_7}
\deltab_n^{(\gamma,7)}  
                & \leq  \frac{1}{\sqrt{n} } \times \frac{2}{6!} \, \prth{ W_1^{(6)}\prth{ \frac{1}{4} }  + W_2^{(6)}\prth{ \frac{1}{4} } }\, \Esp{ \Fb_\gamma^6 } + O\prth{ \frac{1}{n } } \notag \\
                & =: \frac{K_7^\gamma}{\sqrt{n} } + O\prth{ \frac{1}{n } }.
\end{align}

\medskip
\noindent $ \bullet $ \textbf{\underline{Final bound on $ \deltab^\gamma_n $:}} In the end, one has
\begin{align*}
\deltab^\gamma_n(h)  & \leq  \deltab_n^{(\gamma, 1)}(h) + \deltab_n^{(\gamma, 2)}(h) + \deltab_n^{(\gamma, 3)}(h) \\
              & \leq  \deltab_n^{(\gamma, 1)}(h) + \deltab_n^{(\gamma, 2)}(h) +  \norm{h'}_\infty \frac{ \sqrt{ \Esp{ G^{ 2k } } } }{ n^{k + 1/4} } + \deltab^{(\gamma, 4)}_n(h) + \deltab^{(\gamma, 5)}_n(h) \\
              & \leq  \deltab_n^{(\gamma, 1)}(h) + \deltab_n^{(\gamma, 2)}(h) +  \norm{h'}_\infty \frac{ \sqrt{ \Esp{ G^{ 2k } } } }{ n^{k + 1/4} } + \deltab^{(\gamma, 4)}_n(h) \\
              & \hspace{+2cm} + \norm{h}_\infty \prth{ \deltab_n^{(\gamma, 6)}  + \frac{ 2 \gamma \Esp{\Fb_{\!\! \gamma}^2} + 3(\gamma^2 - 2) \Esp{\Fb_{\!\! \gamma}^4}}{ 6 \sqrt{n} } +  O\prth{ \frac{1}{n} } +  \frac{ \Esp{\Fb_{\!\!\gamma}^{4k}} }{(1 - \varepsilon)^{4k} \, n^k }  } \\
              & = \deltab_n^{(\gamma, 1)}(h) + \deltab_n^{(\gamma, 2)}(h) +  \norm{h'}_\infty \frac{ \sqrt{ \Esp{ G^{ 2k } } } }{ n^{k + 1/4} } + \deltab^{(\gamma, 4)}_n(h) \\
              & \hspace{+1cm} + \norm{h}_\infty \prth{ \frac{K_7^\gamma}{\sqrt{n} } + O\prth{ \frac{1}{n  } } + \frac{ 2 \gamma \Esp{\Fb_{\!\! \gamma}^2} + 3(\gamma^2 - 2) \Esp{\Fb_{\!\! \gamma}^4}}{ 6 \sqrt{n} } + O\prth{ \frac{1}{n} } +  \frac{  \Esp{\Fb_{\!\! \gamma}^{4k}} }{(1 - \varepsilon)^{4k} \, n^k } } \\
              & \leq \frac{ \norm{h'}_\infty }{ \sqrt{2 n} } \prth{ 1 + O\prth{ \frac{1}{\sqrt{n} } } } + \frac{ C }{ n^{3/4} } \norm{h'}_\infty + \norm{h'}_\infty \frac{ \sqrt{ \Esp{ G^{ 2k } } } }{ n^{k + 1/4} } +  \norm{h'}_\infty \frac{ \Esp{ \Fb_{\!\! \gamma}^2 } + o(1) }{ n^{3/4} } \\
              & \hspace{+2cm} + \norm{h}_\infty \prth{  \frac{K_7^\gamma}{\sqrt{n} }      + \frac{ 2 \gamma \Esp{\Fb_{\!\! \gamma}^2} + 3(\gamma^2 - 2) \Esp{\Fb_{\!\! \gamma}^4}}{ 6 \sqrt{n} } +  O\prth{ \frac{1}{n} } +  \frac{ \Esp{\Fb_{\!\! \gamma}^{4k}} }{(1 - \varepsilon)^{4k} \, n^k } },
\end{align*}
hence the result.
\end{proof}


\begin{remark}
Note that without using the totally dependent coupling for $ (S_n(p), S_n(q) ) $, one could have taken independent random variables and used the bound 
$$ \Esp{ \abs{S_n(p) - S_n(q) } } \leq \sqrt{ \Esp{ (S_n(p) - S_n(q) )^2 } } = \sqrt{n} \sqrt{\sigma(p)^2 + \sigma(q)^2 },
$$ but this bound does not provide any useful gain. This particular choice of coupling is thus a critical ingredient of the proof.
\end{remark}


\begin{remark}\label{Rk:DominationRandomisation:Smooth}
As announced in Remark~\ref{Rk:DominationCLT:Smooth}, this is the randomisation that dominates the distance in the case $ \beta = 1 \pm \frac{\gamma}{\sqrt{n}} $. This is already visible at the level of the fluctuations, since they are not Gaussian.
\end{remark}


\begin{remark}
It seems interesting to note the discrepancy between the case $ \gamma \geq 0 $ where the derivative of the function $ x \mapsto -\frac{x^4}{12} - \gamma \frac{x^2}{2} $ only vanishes in 0 and the case $ \gamma < 0 $ where the derivative has two additional zeroes in $ \pm \sqrt{-3\gamma} $. As a result, the density $ f_{\Fb_{\! \gamma} } $ has two humps in this case, which is close to the last case treated in \S~\ref{SubSec:Smooth:Beta>1}. 
\end{remark}

\medskip
\subsection{The case $ \beta = 1 $}\label{SubSec:Smooth:Beta=1}

We mention here, as a corollary, the case $ \gamma = 0 $ of Theorem~\ref{Theorem:MagnetisationBeta_n} due to its historical importance.

\medskip
\begin{theorem}[Fluctuations of the \textit{unnormalised} magnetisation for $ \beta = 1 $]\label{Theorem:MagnetisationBetaEqual1}
Let $ \Fb $ be a random variable of law given by $ \Prob{\Fb \in dx } :=  \frac{1}{\Ze_\Fb} e^{-\frac{x^4}{12}} dx $ with $ \Ze_\Fb := \int_{\Rr} e^{-\frac{x^4}{12}} dx = 3^{1/4} 2^{-1/2} \Gamma(1/4)  $. 

Then, for all $ h \in \Ce^1 $ with $ \norm{h}_\infty, \norm{h'}_\infty < \infty $ it holds that 
\begin{align}\label{Eq:Beta=1:Speed}
\abs{ \Esp{ h\prth{ \frac{M_n^{(1 )}}{n^{3/4}}  } } -  \Esp{ h( \Fb ) } } \leq  \prth{ \frac{  C }{ \sqrt{n} } + O\prth{ \frac{ 1 }{ n^{3/4} } } } \prth{ \vphantom{a^{a^a}} \norm{h}_\infty +  \norm{h'}_\infty },   
\end{align}
where $ C > 0 $ is an explicit constant.
\end{theorem}

\medskip
\subsection{The case $ \beta > 1 $}\label{SubSec:Smooth:Beta>1}

We consider the transcendent equation
\begin{align}\label{Def:CriticalPointEquation}
\tanh(x) = \frac{x}{\beta}, \qquad \beta > 1.
\end{align}

An easy study shows that there exist two solutions to this equation denoted by $ \pm x_\beta $ with $ x_\beta > 1 $. We define
\begin{align*}
\Xb_{\!\beta}  \sim \Ber_{\pm x_\beta}\prth{\tfrac{1}{2}}, \qquad \Teb^{(\beta)} \sim \Ber_{\pm t_\beta}\prth{\tfrac{1}{2}}, \qquad t_\beta := \frac{x_\beta}{\beta} = \tanh(x_\beta). 
\end{align*}


\begin{theorem}[Fluctuations of the \textit{unnormalised} magnetisation for $ \beta > 1 $]\label{Theorem:MagnetisationBetaBigger1}
If $ \beta > 1 $, one has for all $ h \in \Ce^1 $ with $ \norm{h}_\infty, \norm{h'}_\infty < \infty $
\begin{align}\label{Eq:Beta>1:Speed}
\abs{ \Esp{ h\prth{ \frac{M_n^{(\beta )}}{n } } } -  \Esp{ h\prth{ \Teb^{(\beta)} } } } \leq \prth{  \frac{C }{\sqrt{n} }  + O_\beta\prth{\frac{1}{n}} } \norm{h'}_\infty 
\end{align}
for an explicit constant $ C  > 0 $. 
\end{theorem}


\begin{proof}
Rescaling $ M_n^{(\beta)} $ by $ n $ and substituting in \eqref{Ineq:RandomisedBerryEsseen} yields
\begin{align*}
\abs{ \Esp{ h\prth{ \frac{M_n^{(\beta)}}{ n } } } - \Esp{ h\prth{ \frac{\Me_n^{(\beta)}}{ n } } } } \leq \frac{C'}{ n } \norm{h'}_\infty
\end{align*}
and the triangle inequality implies the following adaptation of \eqref{Ineq:MagSpeedWithSurrogate}:
\begin{align}\label{Ineq:MagSpeedWithSurrogate:beta>1}
\abs{ \Esp{ h\prth{ \frac{M_n^{(\beta)}}{ n } } } - \Esp{ h\prth{ \Teb^{(\beta)} } } } \leq \frac{C'}{n } \norm{h'}_\infty + \abs{ \Esp{ h\prth{ \frac{\Me_n^{(\beta)}}{n } } } - \Esp{ h\prth{ \Teb^{(\beta)} } } }.
\end{align}

Moreover, one has
\begin{align*}
\frac{\Me_n^{(\beta)} }{ n } =  \frac{G}{\sqrt{n} } \, \sqrt{1 - ( \Tb_{\! n}^{(\beta)} )^2 } +  \Tb_{\! n}^{(\beta)}  
\end{align*}
and, since $ \Tb_{\! n}^{(\beta)} \in [-1, 1] $ a.s., one has $ \frac{\abs{G}}{\sqrt{n} } \, \sqrt{1 - ( \Tb_{\! n}^{(\beta)} )^2 } \leq \frac{\abs{G}}{\sqrt{n} } \to 0 $ in law, hence
\begin{align*}
\abs{ \Esp{ h\prth{ \frac{\Me_n^{(\beta)}}{n } } } - \Esp{ h\prth{ \Teb^{(\beta)} } } } & \leq  \frac{ \norm{h'}_\infty  }{\sqrt{n} } + \abs{ \Esp{ h\prth{ \Tb_{\! n}^{(\beta)} } - \Esp{ h\prth{ \Teb^{(\beta)} } } } }.
\end{align*}

It thus remains to show that $ \Tb_{\! n}^{(\beta)}\! \stackrel{\Le}{\longrightarrow} \Teb^{(\beta)} $ and to control its norm. For this, remark that 
\begin{align*}
\abs{ \Esp{ h\prth{ \Tb_{\! n}^{(\beta)} } - \Esp{ h\prth{ \Teb^{(\beta)} } } } } & = \abs{ \Esp{ h\prth{ \tanh(\Rb_{\! n}^{(\beta)} ) } - \Esp{ h\prth{ \tanh(\Xb_{\!\beta}  }) } } } \\
                 & =: \abs{ \Esp{ \widetilde{h}\prth{ \Rb_n^{(\beta)} } - \Esp{ \widetilde{h}\prth{ \Xb_{\!\beta} } } } }, \qquad \widetilde{h} := h \!\circ\! \tanh,
\end{align*}
where $ \Rb_n^{(\beta)} $ has a law given by
\begin{align}\label{Def:Law:RnBeta}
\mu_{n, \beta}(dy) := \Prob{ \Rb_n^{(\beta)} \in dy } = e^{ - n \varphi_\beta(y) } \frac{dy}{\Ze_{n, \beta} }, \qquad \varphi_\beta(y) := \frac{y^2}{2\beta} - \log\cosh(y).
\end{align}

We now adapt the Laplace method (in the easier case of a global minimum) to show that $ \mu_{n, \beta} \to \frac{1}{2}(\delta_{t_\beta} + \delta_{-t_\beta}) $ weakly. Since $ \int_{\Rr_+}d\mu_{n,\beta} = \int_{\Rr_-}d\mu_{n, \beta} = \frac{1}{2} $, one has 
\begin{align*}
\abs{ \Esp{ \widetilde{h}\prth{ \Rb_n^{(\beta)} } - \Esp{ \widetilde{h}\prth{ \Xb_{\!\beta} } } } } & = \int_{\Rr_+} \crochet{\widetilde{h}(y) - \widetilde{h}(x_\beta) } \mu_{n, \beta}(dy) + \int_{\Rr_-} \crochet{\widetilde{h}(y) - \widetilde{h}(-x_\beta) } \mu_{n, \beta}(dy) \\
              & =: \delta_n(\widetilde{h}) + \delta_n(\widetilde{h}(-\cdot) ).
\end{align*}

It is thus enough to treat the case of $ \delta_n(\widetilde{h}) $. For this, note that \eqref{Def:CriticalPointEquation} is equivalent to $ \varphi_\beta'(x_\beta) = 0 $, hence that for all $ x \geq 0 $
\begin{align*}
\varphi_\beta(x) = \varphi_\beta(x_\beta) +  (x - x_\beta)^2  \int_0^1 \varphi_\beta''(\alpha x + \overline{\alpha} x_\beta) \alpha d\alpha, \qquad \overline{\alpha } := 1 - \alpha.
\end{align*}

As a result
\begin{align*}
\delta_n(\widetilde{h}) & = e^{-n \varphi_\beta(x_\beta) } \int_{\Rr_+} \crochet{\widetilde{h}(x) - \widetilde{h}(x_\beta) } e^{ -n (\varphi_\beta(x) - \varphi_\beta(x_\beta) )  } \frac{dx}{\Ze_{n, \beta} } \\
                 & =  e^{-n \varphi_\beta(x_\beta) } \int_{\Rr_+} \crochet{\widetilde{h}(x) - \widetilde{h}(x_\beta) } e^{ -n (x - x_\beta)^2 \int_0^1 \varphi_\beta''(\alpha x + \overline{\alpha} x_\beta) \alpha d\alpha  } \frac{dx}{\Ze_{n, \beta} } \\
                 & = e^{-n \varphi_\beta(x_\beta) } \int_{-x_\beta \sqrt{n} }^{+\infty} \crochet{\widetilde{h}\prth{ x_\beta + \frac{w}{\sqrt{n} } } - \widetilde{h}(x_\beta) } e^{ - w^2  \int_0^1 \varphi_\beta''(\alpha w/\sqrt{n} + x_\beta) \alpha d\alpha  } \frac{dw}{\Ze_{n, \beta}\sqrt{n} } \\
                 & \leq e^{-n \varphi_\beta(x_\beta) } \frac{|\!|\widetilde{h}'|\!|_\infty }{\sqrt{n} } \int_{-x_\beta \sqrt{n} }^{+\infty} \abs{w}  e^{ - w^2   \int_0^1 \varphi_\beta''(\alpha w/\sqrt{n} + x_\beta) \alpha d\alpha  } \frac{dw}{\Ze_{n, \beta}\sqrt{n} } \\
                 & =: \frac{|\!|\widetilde{h}'|\!|_\infty }{\sqrt{n} }    \frac{  \Ze_{n, \beta} ^{(1)} \sqrt{n} }{\Ze_{n, \beta}\sqrt{n} },
\end{align*}
with \eqref{Def:ZnBetaK}.

Lemma~\ref{Lemma:AsymptoticRenormConstant4} for $ k = 1 $ gives moreover
\begin{align*}
\int_{- x_\beta   \sqrt{n} }^{+\infty} \abs{w}  e^{ - w^2   \int_0^1 \varphi_\beta''(\alpha w/\sqrt{n} + x_\beta) \alpha d\alpha  } dw  \tendvers{n}{+\infty} \int_\Rr \abs{w} e^{- \frac{w^2}{2} \varphi''(x_\beta) } dw = \sqrt{\frac{2\pi}{\varphi_\beta''(x_\beta)^3 } } \Esp{ \abs{G} },
\end{align*}
with $ G \sim \Ns(0, 1) $.

In the end, Lemma~\ref{Lemma:AsymptoticRenormConstant4} for $ k = 1 $ yields 
\begin{align*}
\delta_n(\widetilde{h}) & \leq  \frac{|\!|\widetilde{h}'|\!|_\infty }{\sqrt{n} }  \prth{ \sqrt{\frac{2\pi}{\varphi_\beta''(x_\beta) } } \Esp{ \abs{G} } + o(1) }\times \frac{e^{-n \varphi_\beta(x_\beta) } }{\sqrt{n} \Ze_{n,\beta} } \\
                 & = \frac{|\!|\widetilde{h}'|\!|_\infty }{2 \sqrt{n} } \prth{ \frac{\Esp{ \abs{G} } }{\varphi_\beta''(x_\beta)} + o(1) }.
\end{align*}

Last, $ |\!|\widetilde{h}'|\!|_\infty = \norm{ h' \!\circ  \tanh \times \tanh' }_\infty \leq \norm{ h' \!\circ\! \tanh }_\infty \norm{  \tanh' }_\infty = \norm{ h' }_\infty $ since $ \norm{\tanh'}_\infty = 1 $. Using $ C := C' + 1 $ concludes the proof.
\end{proof}


\begin{remark}\label{Rk:beta>1:conditioning}
One can also consider with \cite{EichelsbacherLoewe} the law of
\begin{align*}
(X_1^{(\beta, +)}, \dots, X_n^{(\beta, +)} ) & :\eqlaw  \prth{ X_1(\widetilde{V}_{n, \beta}), \dots, X_n(\widetilde{V}_{n, \beta}) \Big\vert \widetilde{V}_{n, \beta} \geq \tfrac{1}{2} }, \\
M_n^{(\beta, +)} & := \sum_{k = 1}^n X_k^{(\beta, +)}, 
\end{align*}
that corresponds to conditioning the randomisation $ \Tb_{\! n}^{(\beta)} $ to be positive. Equivalently, one ``zooms in'' a neighbourhood of the positive limiting magnetisation $ t_\beta $, as one only sees the density $ e^{-n \varphi_\beta}/\Ze_{n, \beta} $ on $ \Rr_+ $ and there is only one solution to \eqref{Def:CriticalPointEquation} in this set. 

The conditional version of the previous theorem gives the law of large number $ M_n^{(\beta, +)}/n \to t_\beta $ in probability and in $ L^2 $. The question of interest is then to ``push the chaos expansion to the next order'', i.e. to find the fluctuations, and the surrogate philosophy proves again helpful by considering
\begin{align*}
\Me_n^{(\beta, +)} := \sqrt{n} \, G \, \sqrt{1 - ( \Tb_{\! n}^{(\beta, +)} )^2 } +  n\, \Tb_{\! n}^{(\beta, +)}, \qquad  \Tb_{\! n}^{(\beta, +)} :=  \prth{ \Tb_{\! n}^{(\beta )} \Big\vert \Tb_{\! n}^{(\beta )} \geq 0 }.
\end{align*}

The same methodology applies to show that $ \Me_n^{(\beta, +)}/n \to t_\beta $ in probability and in $ L^2 $, and an easy corollary of the previous arguments shows that $ \Tb_{\! n}^{(\beta, +)} \to t_\beta $ in probability and in $ L^2 $, implying $ \Me_n^{(\beta, +)}/n \approx t_\beta + \sqrt{1 - t_\beta^2} \, G /\sqrt{n} $. In the end, one gets
\begin{align*}
\frac{ M_n^{(\beta, +)} - t_\beta \, n }{ \sqrt{n} \sqrt{1 - t_\beta^2} } \cvlaw{n}{+\infty} \Ns(0, 1),
\end{align*}
with a speed in $ O(1/\sqrt{n}) $ in smooth norm. Details are left to the interested reader.
\end{remark}

\section{Application to the Curie-Weiss magnetisation in Kolmogorov distance}\label{Sec:Kol}
\medskip

Recall the definition of the Kolmogorov distance given in~\eqref{Def:dKol}. 
%
%
%
%
This norm is a functional one in the same vein as the Fortet-Mourier one, using test functions $ h = \Un_{(-\infty, x]} $. The main difference is nevertheless the lack of differentiability of these test functions that prevents the use of \eqref{Ineq:BerryEsseen:Smooth}. An extension of Theorem~\ref{Theorem:MagnetisationBetaSmaller1} to the Kolmogorov distance case requires thus to find an analogue of this inequality for indicator functions. This is furnished by the classical Berry-Ess\'een bound for sums of i.i.d. random variables (see e.g. \cite[thm. 3.39]{Ross}) 
\begin{align}\label{Ineq:BerryEsseen:Kol}
\dKol\prth{ S_n, \, \sigma \sqrt{n} \, G + n \mu } = O\prth{ \frac{ 1}{\sqrt{n} } }.
\end{align}

Here, $ (Z_k)_k $ is a sequence of i.i.d. random variables satisfying $ \Esp{\abs{Z}^3} < \infty $, $ S_n := \sum_{k = 1}^n Z_k $, $ \Var(S_n) =: n \sigma^2 $, $ \Esp{S_n} = n \mu $ and $ G \sim \Ns(0, 1) $. 

Of course, a randomisation of this inequality will give the same result as in \S~\ref{Subsec:Theory:GeneralSurrogate}, since one has just changed test functions.

\medskip
\subsection{The case $ \beta < 1 $}\label{SubSec:Kol:Beta<1}

\begin{theorem}[Kolmogorov distance to the Gaussian for the \textit{unnormalised} magnetisation with $ \beta < 1 $]\label{Theorem:MagnetisationBetaSmaller1:dKol}
With $ \Zb_{\!\beta} \sim \Ns\prth{ 0, \frac{1}{1 - \beta } } $, and $ C > 0 $, we have
\begin{align}\label{Eq:Beta<1:Speed:dKol}
\dKol\prth{ \frac{M_n^{(\beta)} }{ \sqrt{n} }, \Zb_{\!\beta} } = \frac{C}{\sqrt{n}} + O_\beta\prth{ \frac{1}{ n } }.
\end{align}
\end{theorem} 


\begin{proof}
We use \eqref{Ineq:BerryEsseen:Kol} in the particular case of $ \ensemble{\pm 1} $-Bernoulli random variables $ (B_k)_k $ of parameter $ p := \Prob{B = 1} $ (with $ \Esp{B} = 2p - 1 $ and $ \Var(B) = 4 p(1 - p) $) and then randomise $p$. Taking $p$ distributed as in \eqref{Def:PnBeta}, i.e. $ \Pb_n^{(\beta)}  \sim \widetilde{\nu}_{n, \beta} $, or equivalently taking $ t := 2p - 1 $ distributed as in \eqref{Def:TnBeta}, i.e. $ \Tb_{\! n}^{(\beta)} \sim \nu_{n, \beta} $ yields
\begin{align*}
& \dKol\prth{ M_n^{(\beta)}, G \sqrt{n} \sqrt{ 1  - (\Tb_{\! n}^{(\beta)})^2 }  + n \, \Tb_{\! n}^{(\beta)} } = O\prth{ \frac{ 1}{\sqrt{n} } }  \\
 \ \Longleftrightarrow  \ & \dKol\prth{ \frac{M_n^{(\beta)}}{\sqrt{n}} , \frac{\Me_n^{(\beta)}}{\sqrt{n}} } = O\prth{ \frac{ 1}{\sqrt{n} } }
\end{align*}
by invariance of the norm and using the definition of the surrogate $ \Me_n^{(\beta)} $ given in \eqref{Def:SurrogateMnBeta}.

The triangle inequality then yields
\begin{align*}
\dKol\prth{ \frac{M_n^{(\beta)} }{ \sqrt{n} }, \Zb_{\!\beta} } & \leq \dKol\prth{ \frac{M_n^{(\beta)}}{\sqrt{n}} , \frac{\Me_n^{(\beta)}}{\sqrt{n}} } +  \dKol\prth{   \frac{\Me_n^{(\beta)}}{\sqrt{n}} , \Zb_{\!\beta} }     \\
                     & =   \dKol\prth{   \frac{\Me_n^{(\beta)}}{\sqrt{n}} , \Zb_{\!\beta} } + O\prth{\frac{1}{\sqrt{n} } } 
\end{align*}
and one is led to analyse 
\begin{align*}
\dKol\prth{   \frac{\Me_n^{(\beta)}}{\sqrt{n}} , \Zb_{\!\beta} } =: \sup_{x \in \Rr} \delta_n(x),
\end{align*}
with 
\begin{align*}
\delta_n(x) := \abs{ \Prob{ \frac{\Me_n^{(\beta)}}{\sqrt{n}} \leq x   } - \Prob{\Zb_{\!\beta} \leq x } }.
\end{align*}

Recall that $ \Zb_{\!\beta} \eqlaw G + G_\beta $ with $ G \sim \Ns(0, 1) $ independent of $ G_\beta \sim \Ns(0, \beta/(1 - \beta) ) $. One thus has
\begin{align*}
\delta_n(x) = \abs{ \Prob{ G  \sqrt{ 1  - \frac{(\Xb_n^{(\beta)})^2}{n} }  +  \Xb_n^{(\beta)} \leq x  } - \Prob{G + G_\beta \leq x } },
\end{align*}
with $ \Xb_{\! n, \beta} :=  \sqrt{n} \, \Tb_{\! n}^{(\beta)} \to G_\beta $ (in law). 

Integrating on $ G $ and using $ \Phi(x) := \Prob{ G \leq x } $, $ \Xb_{\! n, \beta} \eqlaw - \Xb_{\! n, \beta} $ and $ G_\beta \eqlaw - G_\beta $ yields
\begin{align*}
\delta_n(x) & = \abs{ \Esp{ \Phi\prth{ \frac{x + \Xb_n^{(\beta)} }{  \sqrt{ 1  - \frac{(\Xb_n^{(\beta)})^2}{n} } } } } - \Esp{  \Phi(x + G_\beta) } }.
\end{align*}

Set
\begin{align*}
Y_x & := x + G_\beta , \quad \Psi_{x,n}(u) := \frac{x+ u}{\sqrt{1 - \frac{u^2}{n} }}, \quad \Xi_n(u)  := \frac{1}{\sqrt{1 - \frac{u^2}{n} }}, \\
Y_{x, n} & := \Psi_{x,n}\prth{ \Xb_n^{(\beta)} }, \quad  \Ye_{x, n}  := \Psi_{x,n}(G_\beta) = Y_x \, \Xi_n(G_\beta). 
\end{align*}

Recall that the support of the law of $\Xb_n^{(\beta)} $ is $ (-\sqrt{n}, \sqrt{n}) $. One has
\begin{align*}
\delta_n(x)  & = \abs{ \Esp{\Phi\prth{ Y_x } } - \Esp{ \Phi\prth{Y_{x, n} }  } } \\ 
              & \leq \abs{ \Esp{\Phi\prth{ Y_x } \Unens{ \abs{G_\beta} \leq \sqrt{n} } } - \Esp{\Phi\prth{ \Ye_{x, n} } \Unens{ \abs{G_\beta} \leq \sqrt{n} } } } \\
              & \qquad\qquad + \abs{ \Esp{\Phi\prth{ \Ye_{x, n} } \Unens{ \abs{G_\beta} \leq \sqrt{n} } } - \Esp{\Phi\prth{ Y_{x, n} }  }  } + \abs{ \Esp{\Phi\prth{Y_x} \Unens{ \abs{G_\beta} > \sqrt{n}} }} \\
              & =: \delta_n^{(1)}(x) + \delta_n^{(2)}(x) + \delta_n^{(3)}(x). 
\end{align*}


Since $ \Phi(x) := \Prob{G \leq x} \leq 1 $, the Markov inequality gives for all $ k \geq 1 $
\begin{align*}
\delta_n^{(3)}(x) \leq \Prob{ \abs{G_\beta} > \sqrt{n} \,} \leq \frac{\Esp{\abs{G_\beta}^{2k}}}{n^{k}}.
\end{align*}


Moreover, using the notation $ g_\beta $ (resp. $ g_n $) for the Lebesgue density of $ G_\beta $ (resp. $ \Xb_n^{(\beta)} $) as in the proof of Theorem~\ref{Theorem:MagnetisationBetaSmaller1}, one gets
\begin{align*}
\delta_n^{(2)}(x) & = \abs{ \Esp{ \Phi \circ \Psi_{x,n}(G_\beta) \Unens{ \abs{G_\beta} \leq \sqrt{n} } } -  \Esp{ \Phi \circ \Psi_{x,n}(\Xb_n^{(\beta)}) \Unens{ \abs{\Xb_n^{(\beta)} } \leq \sqrt{n} } } } \\
              & = \abs{ \int_{ (-\sqrt{n}, \sqrt{n} \,)} \Phi \circ \Psi_{x,n} \cdot (g_\beta - g_n) } \\
              & \leq \norm{ \Phi \circ \Psi_{x,n} }_\infty \norm{ (g_\beta - g_n) \Un_{(-\sqrt{n}, \sqrt{n} \,) } }_{L^1(\Rr)} \\
              & =: \norm{ \Phi \circ \Psi_{x,n} }_\infty \delta_n(g_n, g_\beta).
\end{align*}

Since $ 0 \leq \Phi \leq 1 $, one has $ \sup_{x\in \Rr}\norm{ \Phi \circ \Psi_{x,n} }_\infty \leq 1 $, and \eqref{Ineq:EstimateDeltaNg} yields then for all $ x \in \Rr $
\begin{align*}
\delta_n^{(2)}(x) = O_\beta\prth{ \frac{ 1}{ n } }.
\end{align*}

We now estimate $ \delta_n^{(1)}(x) $. Setting $ \varphi := \Phi(\cdot + x) - \Phi\circ\Psi_{x, n}$, we get
\begin{align*}
\delta_n^{(1)}(x) & = \abs{ \Esp{ \Phi(Y_x)\Unens{ \abs{G_\beta} \leq \sqrt{n} } } - \Esp{ \Phi(\Ye_x) \Unens{ \abs{G_\beta} \leq \sqrt{n} }  } } \\
          & = \abs{ \Esp{ \prth{ \Phi(x + G_\beta) - \Phi\circ\Psi_{x, n}(G_\beta) } \Unens{ \abs{G_\beta} \leq \sqrt{n} } } } \\
          & =: \abs{ \Esp{ \varphi(G_\beta) \prth{ \Unens{ \abs{G_\beta} \leq (1 - \varepsilon) \sqrt{n}  } + \Unens{ (1 - \varepsilon) \sqrt{n} \leq \abs{G_\beta} \leq \sqrt{n} } } } } \\
          & \leq \abs{ \Esp{ \varphi(G_\beta)  \Unens{ \abs{G_\beta} \leq (1 - \varepsilon) \sqrt{n} }   } } + \norm{\varphi}_\infty \Prob{ (1 - \varepsilon) \sqrt{n} \leq \abs{G_\beta} } \\
          & \leq \abs{ \Esp{ \varphi(G_\beta) \Unens{ \abs{G_\beta} \leq (1 - \varepsilon) \sqrt{n}}   } } + 2 \, \frac{ \Esp{ \abs{G_\beta}^{2k } } }{ n^k   (1 - \varepsilon)^{2k} },
\end{align*}
for all $ \varepsilon \in (0, 1) $, using the Markov inequality and $ \norm{\varphi}_\infty \leq 2 $.

Recall that $ \Phi'(x) = \frac{1}{\sqrt{2\pi} } e^{- x^2/2} \geq 0 $ for all $ x \in \Rr $ and that $ \Ye_{n, x} = Y_x \, \Xi_n(G_\beta) $. One has then with $ U \sim \Us([0, 1]) $ independent of $ G_\beta $
\begin{align*}
\delta_n^{(4)}(x) & := \abs{ \Esp{ \varphi(G_\beta)  \Unens{ \abs{G_\beta} \leq (1 - \varepsilon) \sqrt{n}} } } = \abs{ \Esp{ \prth{ \vphantom{a^{a^a}} \Phi(Y_x) - \Phi(\Ye_{n, x}) } \Unens{ \abs{G_\beta} \leq (1 - \varepsilon) \sqrt{n}} } } \\
               & \leq  \Esp{ \abs{ \Ye_{x, n} - Y_x } \Phi'\prth{ Y_x + U (\Ye_{x, n} - Y_x) }  \Unens{ \abs{G_\beta} \leq (1 - \varepsilon) \sqrt{n}} }  \\
               & = \Esp{ \abs{ 1 - \Xi_n(G_\beta) } \times \abs{ Y_x } \Phi'\prth{ Y_x + Y_x U (\Xi_n(G_\beta) - 1) }  \Unens{ \abs{G_\beta} \leq (1 - \varepsilon) \sqrt{n}} } \\
               & = \frac{1}{\sqrt{2\pi} } \, \Esp{ \prth{ \vphantom{a^{a^a}} \Xi_n(G_\beta) - 1 } \times \abs{ Y_x } e^{ -\frac{Y_x^2}{2} (1 + U (\Xi_n(G_\beta) - 1) )^2 }  \Unens{ \abs{G_\beta} \leq (1 - \varepsilon) \sqrt{n}} } \\
               & \leq \frac{1}{\sqrt{2\pi} } \, \Esp{ \prth{ \vphantom{a^{a^a}} \Xi_n(G_\beta) - 1 } \times \abs{ Y_x } e^{ -\frac{Y_x^2}{2}  }  \Unens{ \abs{G_\beta} \leq (1 - \varepsilon) \sqrt{n}} },
\end{align*}
since on $ \ensemble{ \abs{G_\beta} \leq (1 - \varepsilon) \sqrt{n}  } $, one has $ \Xi_n(G) - 1 = (1 - G^2/n)^{-1/2} - 1 \geq 0 $. In particular, 
\begin{align*}
\delta_n^{(4)}(x) &  \leq \frac{1}{\sqrt{2\pi} } \, \Esp{ \prth{ \vphantom{a^{a^a}} \Xi_n(G_\beta) - 1 } \times \sup_{x \in \Rr} \ensemble{ \abs{ Y_x } e^{ -\frac{Y_x^2}{2}  } } \Unens{ \abs{G_\beta} \leq (1 - \varepsilon) \sqrt{n}} }  \\
                 & = \frac{1}{\sqrt{2\pi} } \, \Esp{ \prth{ \vphantom{a^{a^a}} \Xi_n(G_\beta) - 1 }\Unens{ \abs{G_\beta} \leq (1 - \varepsilon) \sqrt{n}}   } \times \sup_{y \in \Rr} \ensemble{ \abs{ y } e^{ -\frac{y^2}{2}  } }  \\
                 & = \frac{1}{\sqrt{2\pi e} } \Esp{ \prth{ \vphantom{a^{a^a}} \Xi_n(G_\beta) - 1 }\Unens{ \abs{G_\beta} \leq (1 - \varepsilon) \sqrt{n}}   },
\end{align*}
since $ Y_x = x + G_\beta $ and the supremum of the function $ y \mapsto \abs{y} e^{-y^2/2} $ is easily seen to be reached uniquely in $ y = 1 $. 

Last, the function $ x \mapsto \Xi_n(\sqrt{n} x) - 1 $ is clearly integrable on $  (-1 + \varepsilon, 1 - \varepsilon)   $, and as a result, using $ 0 \leq \frac{1}{\sqrt{1 - t}} - 1 \leq  t \, \sup_{\abs{u} \leq 1 - \varepsilon } \abs{ \frac{d}{du} \frac{1}{\sqrt{1 - u} } } = \frac{t}{2 \varepsilon^{3/2} } $ and $ \Esp{ G_\beta^2} = \frac{\beta}{1 - \beta } $, one finally gets for all $ x \in \Rr $
\begin{align*}
\delta_n^{(4)}(x) \leq  \frac{\beta}{4 \varepsilon^{3/2} (1 - \beta) \sqrt{2\pi e} } \times \frac{1}{n}.
\end{align*}

In the end, we obtain
\begin{align*}
\delta_n(x) & \leq \delta_n^{(1)}(x) + \delta_n^{(2)}(x) + \delta_n^{(3)}(x) \\
             & \leq \delta_n^{(4)}(x) + 2 \, \frac{ \Esp{ \abs{G_\beta}^{2k } } }{ n^k   (1 - \varepsilon)^{2k} } + \delta_n^{(2)}(x) + \delta_n^{(3)}(x) \\
             & \leq   \frac{\beta}{4 \varepsilon^{3/2} (1 - \beta) \sqrt{2\pi e} } \times \frac{1}{n} + 2 \, \frac{ \Esp{ \abs{G_\beta}^{2k } } }{ n^k   (1 - \varepsilon)^{2k} } + O_\beta\prth{ \frac{ 1}{ n } } + \frac{\Esp{\abs{G_\beta}^{2k}}}{n^{k}} \\
             & = O_\beta\prth{ \frac{ 1}{ n } },
\end{align*}
having choosen $ \varepsilon = \frac{3}{4} $ for instance. Taking the supremum over $ x \in \Rr $ ends the proof.
\end{proof}


\begin{remark}\label{Rk:DominationCLT:Kol}
Analogously with the case of a smooth norm analysed in Remark~\ref{Rk:DominationCLT:Smooth}, one sees that $ \dKol\prth{\frac{\Me_n^{(\beta) } }{\sqrt{n}}\, , \Zb_{\!\beta}} = O_\beta\prth{\frac{1}{n} } $ which is faster than the speed coming from the sum of i.i.d.s (i.e. from the CLT). Here again, we can check in Kolmogorov distance the phenomenon of ``CLT domination'' over the randomisation.
\end{remark}

\subsection{The case $ \beta_n = 1 \pm \frac{\gamma}{\sqrt{n}} $, $ \gamma > 0 $}\label{SubSec:Kol:BetaTrans}

\begin{theorem}[Kolmogorov distance to $ \Fb_{\!\! \gamma} $ for the \textit{unnormalised} magnetisation with $ \beta_n = 1 - \frac{\gamma}{\sqrt{n}} $, $ \gamma \in \Rr^* $]\label{Theorem:MagnetisationBeta_n:dKol}
Let $ \Fb_{\!\! \gamma} $ be a random variable of law given in Theorem~\ref{Theorem:MagnetisationBeta_n}. Then,
\begin{align}\label{Eq:Beta_n:Speed:dKol}
\dKol\prth{ \frac{M_n^{(\beta_n)} }{ n^{3/4} }, \Fb_{\!\! \gamma} } = O\prth{ \frac{1}{\sqrt{n} } }.
\end{align}
\end{theorem} 


\begin{proof}
Starting with the approach \eqref{Ineq:BerryEsseen:Kol} with $\beta = \beta_n $ yields
\begin{align}\label{Eq:Beta_n:Speed:dKol2}
\dKol\prth{ \frac{M_n^{(\beta_n)}}{n^{3/4}} , \frac{\Me_n^{(\beta_n)}}{n^{3/4}} } = O\prth{ \frac{ 1}{\sqrt{n} } },
\end{align}
where the surrogate $ \Me_n^{(\beta_n)} $ is given in \eqref{Def:SurrogateMnBeta}. With the coupling \eqref{Def:TotallyDependentCouplingBer} and the notations in the proof of Theorem~\ref{Theorem:MagnetisationBeta_n}, we get  
\begin{align*}
\deltab^\gamma_n(x) & := \abs{ \Pp \prth{ \frac{S_n(\Pb)}{n^{3/4} } \leq x}  - \Pp \prth{\Fb_{\!\! \gamma} \leq x} } \\
               & \leq \, \, \abs{ \Pp \prth{ \frac{S_n(\Pb)}{n^{3/4} } \leq x}  - \Pp \prth{\frac{S_n(\Qb) + \lambdab_n }{n^{3/4} } \leq x} }  \\
               & \hspace{+1.5cm} + \abs{  \Pp \prth{\frac{S_n(\Qb) + \lambdab_n }{n^{3/4} } \leq x} - \Pp \prth{ \frac{ 2\, \Qb (1 - \Qb) \sqrt{n} \, G + n (2\Qb - 1) + \lambdab_n }{n^{3/4} } \leq x}  } \\
               & \hspace{+3cm} + \abs{ \Pp \prth{ \frac{ 2\, \Qb (1 - \Qb) \sqrt{n} \, G + n (2\Qb - 1) + \lambdab_n }{n^{3/4} } \leq x} - \Pp \prth{\Fb_{\!\! \gamma} \leq x} }  \\
               & =: \deltab^{(\gamma,1)}_n(x) + \deltab^{(\gamma,2)}_n(x) + \deltab^{(\gamma,3)}_n(x). 
\end{align*}

\noindent\textbf{$ \bullet $ \underline{Bound on $\deltab^{(\gamma,2)}_n(x)$:}} the Berry-Ess\'een bound \eqref{Eq:Beta_n:Speed:dKol2} gives
\begin{align*}
\deltab^{(\gamma,2)}_n(x) & \leq  \sup_{x \in \Rr} \deltab^{(\gamma,2)}_n(x) = \sup_{y \in \Rr} \abs{  \Prob{ \frac{S_n(\Qb)   }{n^{3/4} } \leq y} - \Prob{ \frac{ 2\, \Qb (1 - \Qb) \sqrt{n} \, G + n (2\Qb - 1)   }{n^{3/4} } \leq y}  } \\
              & = O\prth{ \frac{1}{\sqrt{n} } }.
\end{align*}

\noindent\textbf{$ \bullet $ \underline{Bound on $\deltab^{(\gamma,3)}_n(x)$:}} One has
\begin{align*}
\frac{ 2\, \Qb (1 - \Qb) \sqrt{n} \, G + n (2\Qb - 1) + \lambdab_n }{n^{3/4} }  = \frac{G}{n^{1/4} } \prth{  \sqrt{ 1 - \frac{ \Fb_{\!\! n,\gamma}^2 }{\sqrt{n} } } - 1 } + \Fb_{\!\! n,\gamma}  + \frac{G}{n^{1/4} } \Unens{ \abs{G} > \sqrt{n} }.
\end{align*}

Similarly to the proof of Theorem~\ref{Theorem:MagnetisationBeta_n}, set
\begin{align*}
\Fe_{n,\gamma} & := \frac{G}{n^{1/4} } \prth{  \sqrt{ 1 - \frac{ \Fb_{\!\! n,\gamma}^2 }{\sqrt{n} } } - 1 }, \\
\deltab^{(\gamma,4)}_n(x) & := \abs{ \Pp \prth{ \Fe_{n,\gamma} + \Fb_{\!\! n,\gamma} + \frac{G}{n^{1/4} } \Unens{ \abs{G} > \sqrt{n} } \leq x} - \Pp \prth{\Fe_{n,\gamma} + \Fb_{\!\! n,\gamma} \leq x} }, \\
\deltab^{(\gamma,5)}_n(x) & := \abs{ \Pp \prth{ \Fe_{n,\gamma} + \Fb_{\!\! n,\gamma} \leq x}  - \Pp \prth{ \Fb_{\!\! n,\gamma} \leq x } },\\
\deltab^{(\gamma,6)}_n(x) & := \abs{ \Pp \prth{ \Fb_{n, \gamma} \leq x}  - \Pp \prth{ \Fb_{\!\! \gamma} \leq x } }, 
\end{align*}
so that
\begin{align*}
\deltab^{(\gamma,3)}_n(x) & = \abs{ \Pp \prth{ \Fe_{n,\gamma} + \Fb_{\!\! n,\gamma} + \frac{G}{n^{1/4} } \Unens{ \abs{G} > \sqrt{n} } \leq x} - \Pp \prth{\Fb_{\!\! \gamma} \leq x} } \\
                  & \leq \deltab^{(\gamma,4)}_n(x)  + \deltab^{(\gamma,5)}_n(x) + \deltab^{(\gamma,6)}_n(x).
\end{align*}

\noindent\textbf{$ \bullet $ \underline{Bound on $\deltab^{(\gamma,5)}_n(x)$:}} Integrating on $ \Fb_{\!\! n,\gamma} $ and using $ F_{\Fb_{\! n,\gamma}}(x) := \Prob{ \Fb_{\!\! n,\gamma} \leq x } = \int_{-\infty}^x f_{\! \Fb_{\!\! n,\gamma}} $ and $ \Fe_{n,\gamma} \eqlaw - \Fe_{n,\gamma} $ yields
\begin{align*}
\deltab^{(\gamma,5)}_n(x) & = \abs{ \Esp{ F_{\Fb_{\! n,\gamma}}\prth{ x + \Fe_{n,\gamma}  } } - \Esp{  F_{\Fb_{\! n,\gamma}}(x) } } \\
                   & \leq \norm{ f_{\Fb_{\! n,\gamma}} }_\infty \Esp{ \abs{\Fe_{n,\gamma}} } \\
                   & \leq  \norm{ f_{\Fb_{\! n,\gamma}} }_\infty \frac{ \Esp{ \abs{G} } \Esp{ \Fb_{\!\! n,\gamma}^2 } }{  n^{3/4} }  \\
                   & \leq \prth{ \frac{1}{\Ze_{\Fb_{\! \gamma}}} +  O\prth{\frac{1}{\sqrt{n}}}} \frac{ \Esp{ \Fb_{\!\! \gamma}^2 } + o(1) }{ n^{3/4} },
\end{align*}
%
%
%
where we have used Lemma~\ref{Lemma:MaxDensityFn} for the last inequality.

\noindent\textbf{$ \bullet $ \underline{Bound on $\deltab^{(\gamma,4)}_n(x)$:}} Similarly, using $ G \eqlaw - G $, we get
\begin{align*}
\deltab^{(\gamma,4)}_n(x) & = \abs{ \Esp{ F_{\Fb_{\! n,\gamma} }\prth{ x + \Fe_{n,\gamma} + \frac{G}{n^{1/4} } \Unens{ \abs{G} > \sqrt{n} }  } } - \Esp{  F_{\Fb_{\! n,\gamma}}(x) } } \\
                      & \leq \prth{ \frac{1}{\Ze_{\Fb_{\! \gamma}}} +  O\prth{\frac{1}{\sqrt{n}}}} \prth{ \frac{ \Esp{ \Fb_{\!\! \gamma}^2 } + o(1) }{ n^{3/4} } + \frac{ \sqrt{ \Esp{ G^{ 2k } } } }{ n^{k + 1/4} } }.
\end{align*}

\noindent\textbf{$ \bullet $ \underline{Bound on $\deltab^{(\gamma, 6)}_n(x)$:}} One has for all $ \varepsilon \in (0, 1) $ and setting $ \overline{\varepsilon} := 1 - \varepsilon $ 
\begin{align*}
\deltab^{(\gamma,6)}_n(x) & := \abs{ \Pp \prth{ \Fb_{\!\! n,\gamma} \leq x }  - \Pp (\Fb_{\!\! \gamma} \leq x)  } \\
                   & = \abs{ \int_{- \infty}^x ( f_{\Fb_{\!  n,\gamma} } - f_{\Fb_{\! \gamma}} ) } \leq  \int_\Rr \abs{f_{\Fb_{\!  n,\gamma} } -  f_{\Fb_{\! \gamma}} } \\ 
                   & \leq  \norm{ f_{\Fb_{\!  n,\gamma} } - f_{\Fb_{\! \gamma}}  }_{ L^1([-\overline{\varepsilon} n^{1/4}, \overline{\varepsilon} n^{1/4} ] )  } +  \frac{ 2 \Esp{\Fb_{\!\! \gamma}^{4k}} + o(1) }{(1 - \varepsilon)^{4k} \, n^k },
\end{align*}
using the triangle inequality and the Markov inequality for all $ k \geq 1 $. We then conclude with \eqref{Def:DeltaN_gamma_6}, \eqref{Ineq:DeltaN_gamma_6:Intermediate} and \eqref{Ineq:DeltaN_gamma_7}.

\noindent\textbf{$ \bullet $ \underline{Bound on $\deltab^{(\gamma,1)}_n(x)$:}} Recall that $ \lambdab_n = 2n (\Pb - \Qb) = n^{3/4} (\Fb_{\!\! n,\gamma}' - \Fb_{\!\! n,\gamma} ) $. Then, 
\begin{align*}
\sup_{x \in \Rr} \deltab^{(1)}_n(x) & = \sup_{x \in \Rr} \abs{ \Prob{ \frac{S_n(\Pb)}{n^{3/4} } \leq x}  - \Prob{ \frac{S_n(\Qb) + \lambdab_n }{n^{3/4} } \leq x} }  \\
                & = \sup_{x \in \Rr} \Big\vert \Prob{ S_n(\Pb) - n(2\Pb - 1) \leq n^{3/4} x - n(2\Pb - 1)  } \\
                & \hspace{+3cm} - \Prob{ S_n(\Qb) - n(2\Qb - 1) \leq n^{3/4} x - n(2\Pb - 1)  } \Big\vert  \\
                & \leq \Ee \Bigg( \sup_{x \in \Rr} \Bigg\vert \Prob{ S_n(\Pb) - n(2\Pb - 1) \leq n^{3/4} x - n(2\Pb - 1)  \Big\vert \Pb   }  \\
                    &  \hspace{+3cm} - \Prob{ S_n(\Qb) - n(2\Qb - 1) \leq n^{3/4} x - n(2\Qb - 1)  \Big\vert \Pb, \Qb } \Bigg\vert \Bigg) \\
                    & = \Esp{  \sup_{y \in \Rr} \Big\vert \Prob{ S_n(\Pb) - n(2\Pb - 1)  \leq y  \Big\vert \Pb  }  - \Prob{ S_n(\Qb) - n(2\Qb - 1) \leq y  \Big\vert \Qb }  \Big\vert } \\
                                                             & \leq  \Esp{\abs{\Pb-\Qb} \abs{ \frac{ 1 - (\Pb+\Qb)}{\Pb(1-\Pb)}}} + O\prth{\frac{1}{\sqrt{n} } }, 
\end{align*}
by Lemma~\ref{Lemma:Binomials}.

We recall the following definitions from the proof of Theorem~\ref{Theorem:MagnetisationBeta_n}~:
\begin{align*}
	\Qb = \frac{1}{2} + \frac{1}{2n^{1/4}} \Fb_{\!\! n,\gamma}, \quad
	 \Pb = \frac{1}{2} + \frac{1}{2n^{1/4}} \Fb_{\!\! n,\gamma}^\prime, \\
	2(\Pb - \Qb) = \frac{1}{n^{1/4} } \prth{ \Fb_{\!\! n,\gamma}' - \Fb_{\!\! n,\gamma} } = - \frac{ G }{ \sqrt{n} } \Unens{ \abs{G} \leq \sqrt{n} }.
\end{align*}

As a result
\begin{align*}
\sup_{x \in \Rr} \deltab^{(\gamma,1)}_n(x) & \leq \frac{1}{4 n^{3/4}}  \Esp{\abs{G} \abs{ \frac{ \Fb_{\!\! n,\gamma} + \Fb_{\!\! n,\gamma}^\prime}{\prth{\frac{1}{2} + \frac{1}{2n^{1/4}} \Fb_{\!\! n,\gamma}^\prime} \prth{\frac{1}{2} - \frac{1}{2n^{1/4}} \Fb_{\!\! n,\gamma}^\prime}}}} + O\prth{\frac{1}{\sqrt{n} } } \\
                    & \leq \frac{1}{ n^{3/4} }  \prth{ \Esp{ \prth{   \frac{  \Fb_{\!\! \gamma} + \Fb_{\!\! \gamma}' }{ 1 - \frac{1}{  \sqrt{n} }  \Fb_{\!\! \gamma}^2  } }^{\!\! 2 } } + o(1) }^{1/2} + O\prth{\frac{1}{\sqrt{n} } } \\
                    & \leq \frac{ \sqrt{2 \, \Esp{\Fb_{\!\! \gamma}^2 } } }{n^{3/4} } \prth{ 1 + O\prth{\frac{1}{\sqrt{n} } } } + O\prth{\frac{1}{\sqrt{n} } } =  O\prth{\frac{1}{\sqrt{n} } }.
\end{align*}

\noindent\textbf{$ \bullet $ \underline{Final bound on $\deltab^\gamma_n(x)$:}}

In the end, we obtain  
\begin{align*}
\deltab_n^\gamma(x) & \leq \deltab_n^{(\gamma,1)}(x) + \deltab_n^{(\gamma,2)}(x) + \deltab_n^{(\gamma,3)}(x) \\
             & \leq O\prth{ \frac{ 1}{ \sqrt{n} } } + O\prth{ \frac{ 1}{ \sqrt{n} } } + \deltab_n^{(\gamma,4)}(x) + \deltab_n^{(\gamma,5)}(x) + \deltab_n^{(\gamma,6)}(x) \\
             & \leq O\prth{ \frac{ 1}{ \sqrt{n} } }  + \prth{\frac{1}{\Ze_{\Fb_{\! \gamma}}} +  O\prth{\frac{1}{\sqrt{n}}}} \prth{ \frac{ 2 \Esp{ \Fb_{\!\! \gamma}^2 } + o(1) }{ n^{3/4} } + \frac{ \sqrt{ \Esp{ G^{ 2k } } } }{ n^{k + 1/2} } } + \frac{ \Esp{\Fb_{\!\! \gamma}^{4k}} }{(1 - \varepsilon)^{4k} \, n^k } \\
             & = O\prth{ \frac{ 1}{ \sqrt{n} } },
\end{align*}
which concludes the proof.
\end{proof}

\subsection{The case $ \beta = 1 $}\label{SubSec:Kol:Beta=1}

We give, for completeness (and due to its importance in the literature) the particular case of $ \gamma = 0 $:

\medskip

\begin{theorem}[Kolmogorov distance to $ \Fb $ for the \textit{unnormalised} magnetisation with $ \beta = 1 $]\label{Theorem:MagnetisationBeta=1:dKol}
Let $ \Fb $ be a random variable of law given in Theorem~\ref{Theorem:MagnetisationBetaEqual1}. Then,
\begin{align}\label{Eq:Beta=1:Speed:dKol}
\dKol\prth{ \frac{M_n^{(1)} }{ n^{3/4} }, \Fb } = O\prth{ \frac{1}{\sqrt{n} } }.
\end{align}
\end{theorem} 


\medskip
\subsection{The case $ \beta > 1 $}\label{SubSec:Kol:Beta>1}
Recall that $ \pm x_\beta $ are the solution to the transcendent equation \eqref{Def:CriticalPointEquation} and that $ t_\beta = \tanh(x_\beta) $, with $ \Xb_{\!\beta}  \sim \Ber_{\pm x_\beta}\prth{\tfrac{1}{2}} $ and $ \Teb^{(\beta)} \sim \Ber_{\pm t_\beta}\prth{\tfrac{1}{2}} $.


\begin{theorem}[Fluctuations of the \textit{unnormalised} magnetisation for $ \beta > 1 $]\label{Theorem:MagnetisationBetaBigger1:dKol}
If $ \beta > 1 $, one has  
\begin{align}\label{Eq:Beta>1:Speed}
\dKol\prth{ \frac{M_n^{(\beta )}}{n }, \Teb^{(\beta)} }  = O\prth{\frac{1}{\sqrt{n} } },
\end{align}
for an explicit constant $ C  > 0 $. 
\end{theorem}


\begin{proof}
Using \eqref{Ineq:BerryEsseen:Kol} and the invariance of $ \dKol $ gives
\begin{align}\label{Eq:Beta>1:Speed:dKol2}
& \dKol\prth{ M_n^{(\beta)}, G \sqrt{n} \sqrt{ 1  - (\Tb_{\! n}^{(\beta)})^2 }  + n \, \Tb_{\! n}^{(\beta)} } = O\prth{ \frac{ 1}{\sqrt{n} } } \nonumber \\
 \ \Longleftrightarrow  \ & \dKol\prth{ \frac{M_n^{(\beta)}}{n } , \frac{\Me_n^{(\beta)}}{n } } = O\prth{ \frac{ 1}{\sqrt{n} } }
\end{align}
and the triangle inequality then implies
\begin{align*}
\dKol\prth{ \frac{M_n^{(\beta)} }{ n }, \Teb^{(\beta)} } & \leq \dKol\prth{ \frac{M_n^{(\beta)}}{ n } , \frac{\Me_n^{(\beta)}}{n} } +  \dKol\prth{   \frac{\Me_n^{(\beta)}}{n} , \Tb_{\! n}^{(\beta)} }   +  \dKol\prth{   \Tb_{\! n}^{(\beta)} , \Teb^{(\beta)} }   \\
                     & =  \dKol\prth{   \frac{\Me_n^{(\beta)}}{n} , \Tb_{\! n}^{(\beta)} }    +  \dKol\prth{  \Tb_{\! n}^{(\beta)} , \Teb^{(\beta)} }   + O\prth{\frac{1}{\sqrt{n} } }.
\end{align*}

One has 
\begin{align*}
\dKol\prth{ \frac{\Me_n^{(\beta)}}{n} , \Tb_{\! n}^{(\beta)} }  & = \dKol\prth{  \Tb_{\! n}^{(\beta)}  + \frac{G}{\sqrt{n}} \sqrt{1 - (\Tb_{\! n}^{(\beta)} )^2 }   , \, \Tb_{\! n}^{(\beta)} } \\
                & = \sup_{x \in \Rr}\abs{ \Prob{ \Tb_{\! n}^{(\beta)}  + \frac{G}{\sqrt{n}} \sqrt{1 - (\Tb_{\! n}^{(\beta)} )^2 } \leq x   } - \Prob{ \Tb_{\! n}^{(\beta)} \leq x } } \\
                & = \sup_{x \in \Rr}\abs{ \Prob{  \frac{G}{\sqrt{n}}   \leq \Ye_{n, x, \beta}   } - \Prob{ 0 \leq \Ye_{n, x, \beta} } } = \sup_{x \in \Rr}  \Prob{ 0 \leq \Ye_{n, x, \beta} \leq \frac{G}{\sqrt{n}}  },  
\end{align*}
with 
\begin{align*}
\Ye_{n, x, \beta} & := \frac{ x - \Tb_{\! n}^{(\beta)} }{ \sqrt{1 - (\Tb_{\! n}^{(\beta)} )^2 } } = x \cosh\prth{ \Rb_{  n}^{(\beta)} } - \sinh\prth{ \Rb_{  n}^{(\beta)} } \\
               & = \begin{cases} \sqrt{1 - x^2} \sinh\prth{ \Argtanh(x) - \Rb_n^{(\beta)} } & \mbox{if } \abs{x} < 1 \\
                                 s_x e^{ - s_x \Rb_n^{(\beta)} } & \mbox{if } \abs{x} = 1, \qquad s_x := \operatorname{sign}(x), \\
                             \sqrt{x^2 - 1} \cosh\prth{ \Argtanh(x\inv) - \Rb_n^{(\beta)} } & \mbox{if } \abs{x} > 1.
               \end{cases}
\end{align*}
%
%
%

Thus, for $ \varepsilon > 0 $ small enough so that $ \min_{\abs{x - 1} < \varepsilon} ( x - t_\beta )_+ > 0 $, one has
\begin{align*}
\dKol\prth{ \frac{\Me_n^{(\beta)}}{n} , \Tb_{\! n}^{(\beta)} }  & \leq \sup_{ \abs{x - 1} \leq \varepsilon }  \Prob{ 0 \leq \Ye_{n, x, \beta} \leq \frac{G}{\sqrt{n}}  }  + \sup_{\abs{x - 1} > \varepsilon}\Prob{ 0 \leq \Ye_{n, x, \beta} \leq \frac{G}{\sqrt{n}}  }
\end{align*}

One has 
\begin{align*}
\sup_{\abs{x - 1} > \varepsilon}\Prob{ 0 \leq \Ye_{n, x, \beta} \leq \frac{G}{\sqrt{n}}  } & \leq \Prob{ \frac{\sqrt{  \varepsilon}}{2}  e^{\Argtanh( \varepsilon) - \Rb_n^{(\beta)} } \leq \frac{G}{\sqrt{n}} }  + \Prob{ \sqrt{2 \varepsilon } \leq \frac{G}{\sqrt{n}} } \\
                  & = O\prth{ e^{ - n C   \varepsilon   } }, \qquad C > 0,
\end{align*}
and
\begin{align*}
\sup_{\abs{x - 1} < \varepsilon}\Prob{ 0 \leq \Ye_{n, x, \beta} \leq \frac{G}{\sqrt{n}}  } & \leq  \sup_{\abs{x - 1} < \varepsilon}\Prob{   (\Ye_{n, x, \beta})_+ \leq \frac{G}{\sqrt{n}}  }  \\
               & \leq \sup_{\abs{x - 1} < \varepsilon} \Esp{ e^G } \Esp{ e^{- \sqrt{n} ( x \cosh( \Rb_n^{(\beta)}) - \sinh( \Rb_n^{(\beta)}) )_+ } } \\
               & \leq \sqrt{e} \sup_{\abs{x - 1} < \varepsilon}  e^{- \sqrt{n} [ ( x \cosh(x_\beta) - \sinh(x_\beta) )_+ + o(1) ] }  \\
               & = \sqrt{e} \,   \exp\prth{ - \sqrt{n} \cosh(x_\beta) \prth{  \min_{\abs{x - 1} < \varepsilon} ( x  - \tanh(x_\beta) )_+ + o(1) }  } \\
               & = O\prth{ e^{- C_\varepsilon \sqrt{n} } }, \qquad C_\varepsilon > 0.
\end{align*}

It then remains to analyse  
\begin{align*}
\dKol\prth{  \Tb_{\! n}^{(\beta)} , \Teb^{(\beta)} } & := \sup_{x \in \Rr}\abs{ \Prob{ \Tb_{\! n}^{(\beta)}    \leq x   } - \Prob{ \Teb^{(\beta)} \leq x } } \\
                   & = \sup_{x \in \Rr}\abs{ \Prob{  \tanh(\Rb_n^{(\beta)} )    \leq x   } - \Prob{ \tanh(\Xb_{\! \beta}) \leq x } } \\
                   & = \sup_{y \in \Rr}\abs{ \Prob{   \Rb_n^{(\beta)} \leq y } - \Prob{  \Xb_{\! \beta} \leq y } } \\
                   & = \sup_{y \in \Rr}\abs{ \int_{-\infty}^y e^{-n \varphi_\beta(x)} \frac{dx}{\Ze_{n, \beta} } - \frac{1}{2}\prth{ \Unens{ y \geq -x_\beta } + \Unens{y \geq x_\beta} } }.
\end{align*}

Since $ \max_{A \cup B} f = \max\{ \max_A f, \max_B f \} $, it is enough to consider the following quantities:
\begin{align*}
I_+(y) & := \int_y^{+\infty} e^{-n \varphi_\beta(x)} \frac{dx}{\Ze_{n, \beta} }  , \qquad y > x_\beta, \\
I_0(y) & := \int_{-\infty}^y e^{-n \varphi_\beta(x)} \frac{dx}{\Ze_{n, \beta} } - \frac{1}{2} = \int_0^y e^{-n \varphi_\beta(x)} \frac{dx}{\Ze_{n, \beta} }, \qquad -x_\beta < y \leq x_\beta, \\
I_-(y) & := \int_{-\infty}^y e^{-n \varphi_\beta(x)} \frac{dx}{\Ze_{n, \beta} } = I_+(-y), \qquad y < -x_\beta. 
\end{align*}

By symmetry, it is enough to consider the case $ 0 < y < x_\beta $ in the case of $ I_0(y) $. One thus has for $ y \in (0, x_\beta ) $ 
and with Lemma~\ref{Lemma:AsymptoticRenormConstant4}  
$$ 
I_0(y) = \int_0^y e^{-n \varphi_\beta(x)} \frac{dx}{\Ze_{n, \beta} } \leq y \frac{ e^{- n \varphi_\beta(y) } }{ \Ze_{n, \beta} } = O\prth{  \sqrt{n} \, e^{ -n (\varphi_\beta(x_\beta - \varepsilon) - \varphi_\beta(x_\beta) ) } }   =   O\prth{  \sqrt{n} \, e^{ -n  \varepsilon^2 \varphi''_\beta(x_\beta)/2  } },
$$ 
having set $ y := x_\beta - \varepsilon $ with $ \varepsilon > 0 $.

In the case of $ I_+(y) $, $ y = x_\beta + \varepsilon $, one has with the Markov inequality
\begin{align*}
I_+(x_\beta + \varepsilon) & = \Prob{ \Rb_n^{(\beta)} \geq x_\beta + \varepsilon} \leq \frac{1}{\varepsilon} \Esp{ \abs{ \Rb_n^{(\beta)} - x_\beta} } = O\prth{ \frac{1}{\sqrt{n} } },
\end{align*}
using the computations at the end of the proof of Theorem~\ref{Theorem:MagnetisationBetaBigger1}. This concludes the proof.
\end{proof}

\section{Conclusion and perspectives}\label{Sec:Conclusion}



This work was concerned with the study of complex models having particular symmetries encoded into a structure of mixture. It opens the path to study other such models, at several other levels:
\begin{enumerate}

\item Large or moderate deviations in the vein of \cite{EichelsbacherLoeweMDP, EllisBook}, mod-$ \phi $ convergence \cite{FerayMeliotNikeghbali1, FerayMeliotNikeghbali2}, local CLT \cite{FleermannKirschToth} and functional CLT in the vein of \cite{JeonCW, Papangelou} for the magnetisation with the De Finetti randomisation and the surrogate approach. Interesting questions are e.g.: will the LDP surrogate coming from the randomisation of the precise Cramer/Bahadur-Rao theorem for sums of i.i.d.'s give the correct LDP for the magnetisation? The same questions can be asked for the local CLT or the functional CLT, starting from the one for sums of i.i.d.'s followed by a randomisation. In the case of the LDP, one does not have a surrogate random variable but a surrogate probability (i.e. the randomisation of a probability); one will have to rescale again the randomised LDP to get the correct rate (which should imply a fair piece of analysis). In the case of a local CLT, one can use a surrogate with values in a lattice, such as the Poisson surrogate in Remark~\ref{Rk:PoissonSurrogate}. In the case of the functional CLT, the results of Jeon and Papangelou \cite{JeonCW, Papangelou} can certainly be obtained with the surrogate approach with a modification of the arguments presented here, at least in the Bernoulli case. Nevertheless, the functional rescaling of the i.i.d. sums will give a Brownian motion and not a Brownian bridge, i.e. the functional surrogate will then be of the form $ m_n(\Pb_{\!\! n}) + \sigma_n(\Pb_{\!\! n}) B_t $ where $ B $ is a standard Brownian motion. As a result, the critical case $ \beta = 1 $ will require an additional argument to turn the Brownian motion into a Brownian bridge (the rescaling performed by Papangelou uses a random shift given precisely by the randomisation). As a by-product of such a study, one will get as a bonus the speed of convergence in Kolmogorov distance (say), a question that is not treated in \cite{JeonCW, Papangelou}. 

\item The functional CLTs studied in \cite{ChagantySethuraman1987, JeonCW, Papangelou} or the extension of the Ellis-Newman Theorem~\ref{Theorem:FluctuationsMagClassical} for general measures on spins 
show the emergence of a general condition on the class of random variables considered, called the \textit{class L} in \cite{ChagantySethuraman1987, EllisNewman}: their cumulant generating series $ \Lambda_X(s) := \log\Esp{e^{s X} } $ must satisfy ``$ s \mapsto \frac{s^2}{2} + \Lambda_X(s) $ has a unique minimiser''. On the other hand, the general surrogate approach will be concerned with the replacement of the expectation and variance of the random variables $ X_k $ generalising the Bernoulli random variables $ B_k(p) $, implying a surrogate of the form $ m_X(V_n) + \sigma_X(V_n) G $ where $ V_n $ will be the new De Finetti randomisation. Finding conditions on the expectation and variance $ (m_X, \sigma_X^2) $ so that the (functional) CLT remains valid and linking them with the conditions defining the class \textit{L} random variables is an interesting question. 

\item Change of model with exchangeability, e.g. the Curie-Weiss-Potts model \cite{EichelsbacherMartschink} or the inhomogeneous Curie-Weiss model \cite{DommersEichelsbacher}, the Curie-Weiss-Heisenberg or Curie-Weiss clock model \cite{LiggettSteifToth}. Such models will still be exchangeable and Lemma~\ref{Lemma:DeFinettiCWspins} will give their De Finetti measure in a straightforward way. More generally, all models with exchangeable pairs can be scrutinised to see if a De Finetti measure is present, and the philosophy of this article will then apply.

\item Dynamical Curie-Weiss model and dynamical De Finetti randomisation: replacing the spins $ (X_k)_k $ by $ (X_k(t))_k $ for a good dynamic such as the spin-flip dynamics of \cite{KulskeLeNy} will lead to a dynamical exchangeability randomisation whose dynamic will certainly be interesting. The interplay between $ t\to+\infty $ and $ n \to +\infty $ could lead to interesting phenomena in the underlying surrogate, in particular concerning the marginally relevant disorder. 

\item Change of model with a weight that reads as a Laplace transform, to mimic the proof of Lemma~\ref{Lemma:DeFinettiCWspins} and obtain the De Finetti measure in the same way. For instance, replacing the Gaussian by a Poisson random variable gives the following \textit{Poissonian Curie-Weiss model}:
\begin{align*}
\qquad \Pp_{PoCW(\beta, n) } = \frac{e^{ \psi(S_n) }}{\Esp{ e^{ \psi(S_n) } } } \bullet \Pp_{(X_1, \dots, X_n) }, \quad \psi(x) := e^{\lambda x} - 1 - \mu x = \log\Esp{ e^{x(\Poisson(\lambda) - \mu) } }
\end{align*}
and this time, the De Finetti randomisation will be discrete, of law $ \mu_n(k) = \frac{1}{\Ze_n(\lambda, \mu) \, k!} \cosh(k)^n e^{  k (\ln(\lambda) - \mu) } $, hence the relevance of a Poisson surrogate.

\item Self-Organised Criticality, in short, SOC \cite{CerfGorny} : the SOC-Curie-Weiss model is of the type 
\begin{align*}
\Pp_{SOC(\beta, n) } = \frac{e^{ \beta H(X_1, \dots, X_n) }}{\Esp{ e^{ \beta H(X_1, \dots, X_n) } } } \bullet \Pp_{(X_1, \dots, X_n) }, \qquad H(\xb) = \frac{\prth{ \sum_{k = 1}^n x_k}^2 }{1 + \sum_{k = 1}^n x_k^2 }, \quad \beta = 1.
\end{align*}

$H$ is again a symmetric function, hence, the vector is exchangeable. There is thus a surrogate coming from the CLT (the same as ours), but with a different De Finetti measure, not immediately straightforward to compute.

\end{enumerate}

We hope to come back to some of these questions in the future.

\appendix

\section{Analysis of diverse constants}\label{Sec:Estimates}
\medskip
\subsection{Renormalisation constants}

\subsubsection{Case $ \beta < 1 $}

\begin{lemma}[Asymptotic analysis of $ \Ze_{n, \beta} $ for $ \beta < 1 $]\label{Lemma:AsymptoticRenormConstant}
We have
\begin{align*}
\abs{ \sqrt{ \frac{ C_\beta }{2 \pi } } \times \Ze_{n, \beta} \sqrt{n} - 1 } \leq \frac{1}{ n}   \frac{C_\beta^{-9/2}}{4}.
\end{align*}
\end{lemma}


\begin{proof}
Using \eqref{Eq:DeFinettiMeasureT}, one has
\begin{align*}
\Ze_{n, \beta}  := \int_{(-1, 1) } e^{- \frac{n}{2} F_\beta(t) } \frac{dt}{1 - t^2}, \qquad F_\beta(t) := \frac{1}{\beta} \Argtanh(t)^2 + \ln(1 - t^2).
\end{align*}

Setting $ x = \Argtanh(t) $ and $ y := \sqrt{n} x $ gives 
\begin{align*}
\Ze_{n, \beta}  := \int_\Rr e^{- \frac{n}{2} F_\beta(\tanh(x)) } dx = \int_\Rr e^{- \frac{n}{2} F_\beta(\tanh(y/\sqrt{n})) } \frac{dy}{\sqrt{n} }, 
\end{align*}
with 
\begin{align*}
\frac{n}{2} F_\beta(\tanh(y/\sqrt{n}))  & = \frac{y^2}{2\beta} + \frac{n}{2}  \log\prth{ 1 - \tanh(y/\sqrt{n})^2 } \\
                       & = \frac{y^2}{2}\prth{\frac{1}{\beta} - 1} + \frac{y^2}{2} + \frac{n}{2}  \log\prth{ 1 - \tanh(y/\sqrt{n})^2 } \\
                       & =: \frac{y^2}{2} C_\beta + \psi_n(y),
\end{align*}
where  
\begin{align}\label{Def:CBetaAndPsiN}
\begin{aligned}
C_\beta & := \frac{1}{\beta} - 1, \\
\psi_n(y) & := \frac{y^2}{2} + \frac{n}{2}  \log\,\prth{ 1 - \tanh(y/\sqrt{n})^2 } = \frac{y^2}{2} - n \log\cosh\prth{ \tfrac{y}{\sqrt{n}} } =: n \psi\prth{\frac{y}{\sqrt{n } } }, \\
\psi(y) & := \frac{y^2}{2} - \log\cosh\prth{ y },
\end{aligned}
\end{align}

since $ 1 - \tanh^2 = \cosh^{-2} $. The equality
$
\int_\Rr e^{-\frac{y^2}{2} C_\beta } dy = \sqrt{\frac{2\pi}{C_\beta} }
$ 
implies that
$$ 
\sqrt{\frac{2\pi}{C_\beta} } - \Ze_{n, \beta} \sqrt{n}  = \int_\Rr e^{-\frac{y^2}{2} C_\beta } dy - \int_\Rr e^{- \frac{n}{2} F_\beta(\tanh(y/\sqrt{n})) } dy    
                 =  \int_\Rr \prth{ 1 - e^{ - \psi_n(y) }   } e^{-\frac{y^2}{2} C_\beta } dy. 
$$ 

The study of $ \psi_n $ with SageMath \cite{SageMath} shows that $ \psi_n $ is non negative on $ \Rr $ with only cancelation in $ 0 $. This can also be seen with the inequality $ \cosh(t) \leq \exp(\frac{t^2}{2}) $ that follows from the termwise comparison of the Taylor series of each function. As a result, the previous quantity is positive on $ \Rr^* $. Moreover, in the same vein as for $ \kappa_n $ defined in \eqref{Def:CbetaNandKappaN}, one has $ \psi_n(0) = \psi_n'(0) = \psi_n''(0) = \psi_n'''(0) = 0 $, and the Taylor formula with integral remainder gives at the fourth order
\begin{align*}
\psi_n (y) = \frac{y^4}{6} \int_0^1 (1 - \alpha)^3 \psi_n^{(4)}( \alpha y ) d\alpha. 
\end{align*}

A computation with SageMath \cite{SageMath} gives 
\begin{align*}
\psi_n'(y) & = y - \sqrt{n} \tanh\prth{ \frac{y}{\sqrt{n}} }, \quad \psi_n''(y) = \tanh\prth{\frac{y}{\sqrt{n}} }^2, \\
\psi_n'''(y) & = \frac{2}{\sqrt{n}}\tanh\prth{\frac{y}{\sqrt{n}} } \prth{ 1 - \tanh\prth{ \frac{y}{\sqrt{n}} }^2 },  \\
\psi_n^{(4)}(y) & = \frac{2}{n} \prth{ 1 - \tanh\prth{ \frac{y}{\sqrt{n}} }^2} \prth{ 1 - 3 \tanh\prth{ \frac{y}{\sqrt{n}} }^2}.  
\end{align*}
%
%
%
%
%
%

Moreover, the function $ y \mapsto n \log( 1 - \tanh(y/\sqrt{n})^2 ) = - 2n \log\cosh(y/\sqrt{n}) $ has bounded derivatives. One thus has
\begin{align*}
0 \leq  1 - e^{ - \psi_n(y) } \leq \psi_n (y) \leq \frac{y^4}{24} \norm{\psi_n^{(4)}}_\infty.  
\end{align*}

%
%
%
%
%
%

Since $ \abs{ \psi_n^{(4)}(y) } := \frac{2}{n} \abs{ 1 - \tanh(\frac{y}{\sqrt{n}})^2}\,  \abs{1 - 3 \tanh(\frac{y}{\sqrt{n}})^2} \leq \frac{2}{n}  $, one gets
$$ 
0 \leq \sqrt{\frac{2\pi}{C_\beta} } - \Ze_{n, \beta} \sqrt{n}   = \int_\Rr \prth{ 1 - e^{ - \psi_n(y) }   } e^{-\frac{y^2}{2} C_\beta } dy   
                 \leq \frac{2}{24 \, n} \int_\Rr  y^4 e^{-\frac{y^2}{2} C_\beta } dy = \frac{1}{4 \, n} \sqrt{ 2\pi } C_\beta^{-5},
$$ 
hence the result using the fourth moment of a Gaussian (equal to 3).
\end{proof}

\subsubsection{Case $ \beta = 1 $}

\begin{lemma}[Asymptotic analysis of $ \Ze_{n, 1}  $]\label{Lemma:AsymptoticRenormConstant2}
We have
\begin{align*}
\abs{ n^{1/4} \frac{ \Ze_{n, 1}  }{\Ze_\Fb} - 1 } = O\prth{\frac{1}{\sqrt{n}}}.
\end{align*}
\end{lemma}


\begin{proof}
Using \eqref{Eq:DeFinettiMeasureT}, one has
\begin{align*}
\Ze_{n, 1}  := \int_{(-1, 1) } e^{- \frac{n}{2} F_1(t) } \frac{dt}{1 - t^2}, \qquad F_1(t) := \Argtanh(t)^2 + \ln(1 - t^2).
\end{align*}

Setting $ x = \Argtanh(t) $ and $ y := n^{1/4}  x $ gives 
\begin{align*}
\Ze_{n, 1}  := \int_\Rr e^{- \frac{n}{2} F_1(\tanh(x)) } dx = \int_\Rr e^{- \frac{n}{2} F_1(\tanh(y/n^{1/4} )) } \frac{dy}{n^{1/4}  } =: \int_\Rr e^{- n \, \psi(y/n^{1/4} ) } \frac{dy}{n^{1/4}  },
\end{align*}
with $ \psi $ defined in \eqref{Def:CBetaAndPsiN}.

Since $ \Ze_\Fb = \int_\Rr e^{-\frac{y^4}{12} } dy $, one has
$$ 
n^{1/4} \Ze_{n, 1}  - \Ze_\Fb  = \int_\Rr e^{- n \, \psi(y/n^{1/4} ) } dy - \int_\Rr e^{-\frac{y^4}{12} } dy    
                =  \int_\Rr \prth{ 1 - e^{ - \widetilde{\psi}_n(y) }   } e^{- n \, \psi(y/n^{1/4} ) } dy, 
$$ 
where 
\begin{align}\label{Def:PsiTilde}
\begin{aligned}
\widetilde{\psi}_n(y) & := \frac{y^4}{12} - n \psi\prth{ \frac{y}{ n^{1/4} } } = n \prth{  \frac{1}{12} \prth{ \frac{y}{ n^{1/4} } }^4 - \psi\prth{ \frac{y}{ n^{1/4} } } } =: n \widetilde{\psi}\prth{ \frac{y}{ n^{1/4} } }, \\
\widetilde{\psi}(y) & := \frac{y^4}{12} - \psi(y) = \frac{y^4}{12} - \frac{y^2}{2} + \log\cosh(y).
\end{aligned}
\end{align}

The study of $ \widetilde{\psi}_n $ with SageMath \cite{SageMath} shows that it is non negative on $ \Rr $ with only cancelation in $ 0 $ (see also the Taylor formula at the fourth order below). As a result, the previous quantity is positive on $ \Rr^* $.  Moreover, in the same vein as for $ \kappa_n $ defined in \eqref{Def:CbetaNandKappaN}, one has $ \widetilde{\psi}_n^{(k)}(0) = 0 $ for all $ k \in \intcrochet{0, 5} $, and the Taylor formula with integral remainder gives at the sixth order
\begin{align*}
\widetilde{\psi}_n (y) = \frac{y^6}{120} \int_0^1 (1 - \alpha)^5 \widetilde{\psi}_n^{(6)}( \alpha y ) d\alpha  = \frac{y^6}{120 \sqrt{n} } \int_0^1 (1 - \alpha)^5 \widetilde{\psi}^{(6)}( \alpha y n^{-1/4} ) d\alpha.
\end{align*}

A computation with SageMath \cite{SageMath} gives 
\begin{align*}
\widetilde{\psi}'(y) &=  \tanh(y)  + \frac{y^3}{3} - y, \quad  
\widetilde{\psi}''(y) = - \tanh(y)^2 + y^2,   \\
\widetilde{\psi}'''(y) &=  2 \tanh(y) \prth{ 1 - \tanh(y)^2 } + 2y, \quad
\widetilde{\psi}^{(4)}(y) = 2 \, \frac{\sinh(y)^4 + 4\sinh(y)^2}{\cosh(y)^4} \ \geq 0 \\
\widetilde{\psi}^{(5)}(y) &= - 8 \, \frac{\prth{\cosh(y)^4 - 3} \sinh(y) }{ \cosh(y)^5}, \quad 
\widetilde{\psi}^{(6)}(y) = 4\, \frac{4\cosh(y)^4 - 30\cosh(y)^2 +  30}{  \cosh(y)^6}. 
\end{align*}

%
%
%
%
%
%

Note that the Taylor formula at the fourth order gives $ \psi_n (y) = \frac{y^4}{6} \int_0^1 (1 - \alpha)^3 \widetilde{\psi}^{(4)}( \alpha y /n^{1/4} ) d\alpha $ and since $ \widetilde{\psi}^{(4)} \geq 0 $, it is easily seen that $ \widetilde{\psi}_n $ is non negative.


One thus has
\begin{align*}
0 \leq  1 - e^{ - \widetilde{\psi}_n(y) } \leq \widetilde{\psi}_n (y) \leq \frac{y^6}{720} \norm{\widetilde{\psi}_n^{(6)}}_\infty   = \frac{y^6}{720\sqrt{n} } \norm{\widetilde{\psi}^{(6)}}_\infty.
\end{align*}

Since 
\begin{align*}
\widetilde{\psi}^{(6)}(y) & = 4\, \frac{4\cosh(y)^4 - 30\cosh(y)^2 +  30}{  \cosh(y)^6} \\
                 & = 4 \prth{ 4(1 - \tanh(y)^2) - 30 (1 - \tanh(y)^2)^2 + 30 (1 - \tanh(y)^2)^3  } \leq 136,  
\end{align*}
one gets
\begin{align*}
0 \leq n^{1/4} \Ze_{n, 1} - \Ze_\Fb & = \int_\Rr \prth{ 1 - e^{ - \widetilde{\psi}_n(y) } } e^{- n \psi ( y/n^{1/4})  } dy \\
                & \leq \frac{136}{720 \, \sqrt{n}} \int_\Rr  y^6 e^{-n \psi ( y/n^{1/4}) } dy = \frac{136}{720 \,  \sqrt{n}}  \,  n^{1/4} \Ze_{n, 1}\Esp{ \Fb_{\!\! n}^6 },
\end{align*}
hence
\begin{align*}
0 \leq 1 - \frac{\Ze_\Fb }{n^{1/4} \Ze_{n, 1} } &   \leq \frac{136}{720 \, \sqrt{n}}  \prth{ \Esp{ \Fb^6 } + o(1) },
\end{align*}
which concludes the proof.
%
%
%
%
%
%
%
%
%
%
%
%
\end{proof}

\subsubsection{Case $ \beta_n = 1 \pm \frac{\gamma}{\sqrt{n}}$, $ \gamma \in \Rr $.}

\begin{lemma}[Asymptotic analysis of $ \norm{f_{\! \Fb_{\!\! n, \gamma} } }_\infty $]\label{Lemma:MaxDensityFn}
With $ \Fb_{\!\! n, \gamma} $ defined in \eqref{Def:FnGamma} and $ \beta_n := 1 - \frac{\gamma}{\sqrt{n}} $, we have
\begin{align*}
\max_{x \in \Rr} f_{\! \Fb_{\!\! n, \gamma} }(x) = \frac{1}{\Ze_{\Fb_{\!\gamma}}} +  O\prth{\frac{1}{\sqrt{n}}}.
\end{align*}
\end{lemma}


\begin{proof}
Using \eqref{Def:FnGamma}, \eqref{Def:TnBeta} and \eqref{Eq:DeFinettiMeasureT}, one gets
\begin{align*}
f_{\! \Fb_{\!\! n, \gamma} }(x) & = \frac{1}{n^{1/4} } f_{n, \beta_n}\prth{ \frac{x }{n^{1/4} } } \\
                                  & := \frac{1}{n^{1/4} \Ze_{n, \beta_n} }  e^{ -\frac{n}{2 } \Argtanh\,\prth{\frac{x}{n^{1/4}}}^2 - (\frac{n}{2} + 1) \ln\,\prth{1 - \frac{x^2}{\sqrt{n}}} } \Unens{\abs{x} < n^{1/4}} \\
                                  & =: \frac{e^{\xi_n(x) } }{n^{1/4} \Ze_{n, \beta_n} } \Unens{\abs{x} < n^{1/4}}
\end{align*}
and 
\begin{align*}
\xi_n'(x) = - n^{3/4} \, \frac{\beta_n\inv\Argtanh\prth{\frac{x}{n^{1/4}}} - \prth{1 + \frac{2}{n}} \frac{x}{n^{1/4}} }{1 - (\frac{x}{n^{1/4}})^2  }.
\end{align*}

We have $ \xi_n'(0) = 0 $, and $\xi_n''(0) = \frac{2}{\sqrt{n}} $, hence, $0$ is a minimum of $ \xi_n $ and $ f_{\! \Fb_{\!\! n, \gamma} } $. To find the maxima, set
$y := \Argtanh\prth{\frac{x}{n^{1/4}}}$. One needs to analyse the solutions of the equation
\begin{align}\label{Eq:MagBeta>1}
\frac{\tanh(y)}{y} = \frac{\beta_n}{1 + \frac{2}{n}}.  
\end{align}

This equation is well known in the study of the Curie-Weiss model as it gives the limiting magnetisation when $ \beta > 1 $ (see e.g. \cite[prop. 8]{KirschKriecherbauer}). An easy study shows that \eqref{Eq:MagBeta>1} has a unique solution $y_n$ on $\Rr_+$ and by symmetry a unique solution $-y_n$ on $\Rr_-$, both being global maxima. 

Define
\begin{align*}
G(w) & := 1 - \frac{\tanh(\sqrt{w})}{\sqrt{w} } = \frac{w}{3} + O(w^2) \qquad\mbox{when } w \to 0, \\
\varepsilon_n & = 1 - \frac{\beta_n}{1 + \frac{2}{n}} = 1 - \frac{1 - \frac{\gamma}{\sqrt{n}}}{1 + \frac{2}{n}} \ \ \equivalent{n \to +\infty} \ \frac{\gamma}{\sqrt{n}}\Unens{\gamma\neq 0} + \frac{2}{n}\Unens{\gamma = 0}.
\end{align*}

Then, the solution $y_n$ of \eqref{Eq:MagBeta>1} is such that $ y_n = \sqrt{w_n} $, where $G(w_n) = \varepsilon_n$. 
Since $G$ is bijective on $ \Rr_+ $, one can define its inverse $ G\inv $ for the composition $ \circ $ of functions, and, both functions being $ \Ce^\infty $,
\begin{align*}
w_n & = G\inv(\varepsilon_n) = G\inv(0) + (G\inv)'(0) \varepsilon_n + O(\varepsilon_n^2) = 3 \varepsilon_n + O(\varepsilon_n^2) , 
\end{align*}
hence
\begin{align*}
y_n & = \sqrt{w_n} = \sqrt{3\varepsilon_n}\prth{1 + O\prth{\varepsilon_n^2} } = \sqrt{3\varepsilon_n} + O\prth{ \varepsilon_n^{3/2}   }, \\
\frac{x_n}{ n^{1/4} } & :=  \tanh(y_n) = \tanh\prth{ \sqrt{3\varepsilon_n} + O\prth{ \varepsilon_n^{3/2}   } } =  \sqrt{3\varepsilon_n} + O\prth{ \varepsilon_n   }. 
\end{align*}

In the end, using $ H(x) := x + \log\prth{1 - \tanh(\sqrt{x})^2 } = \frac{x^2}{6} + O(x^3) $, one gets
\begin{align*}
\norm{ f_{\! \Fb_{\!\! n, \gamma} } }_\infty & = f_{\! \Fb_{\!\! n, \gamma} }(x_n)  
                    = \frac{1}{n^{1/4} \Ze_{n, \beta_n} } e^{ -\frac{n}{2\beta_n } \Argtanh\,\prth{\frac{x_n}{n^{1/4}}}^2 - (\frac{n}{2} + 1) \ln\,\prth{1 - \frac{x_n^2}{\sqrt{n}}} } \\
                  & = \frac{1}{n^{1/4} \Ze_{n, \beta_n} }  e^{ -\frac{n}{2 \beta_n} w_n - (\frac{n}{2} + 1) \ln\,\prth{1 - \,\tanh(\sqrt{w_n})^2 } }  \\
                  & = \frac{1}{n^{1/4} \Ze_{n, \beta_n} } e^{ -\frac{n}{2 \beta_n} H(w_n) - \ln\,\prth{1 - \,\tanh(\sqrt{w_n})^2 } } \\
                  & = \frac{1}{n^{1/4} \Ze_{n, \beta_n} }  e^{ -\frac{n}{12 \beta_n} w_n^2 + O(n w_n^3) + \frac{6}{n} + O(n^{-3/2})   } \\
                  & = \frac{1}{n^{1/4} \Ze_{n, \beta_n} }  e^{ \frac{\gamma^2}{4} \Unens{\gamma \neq 0} + \frac{9}{n} \Unens{\gamma = 0} + o(n\inv)	 } = \frac{1}{n^{1/4} \Ze_{n, \beta_n} } \prth{1 + O\prth{\frac{1}{\sqrt{n}} } }
\end{align*}
and Lemma~\ref{Lemma:AsymptoticRenormConstant3} gives $ n^{1/4} \Ze_{n, \beta_n} = \Ze_{\Fb_{\!\gamma}} + O\prth{\frac{1}{\sqrt{n}}} $, concluding the proof.
\end{proof}

\begin{lemma}[Asymptotic analysis of $ \Ze_{n, \beta_n}  $]\label{Lemma:AsymptoticRenormConstant3}
Set $ \Ze_{\Fb_{\! \gamma}} := \int_\Rr e^{-\frac{y^4}{12} - \gamma \frac{ y^2}{2}} dy $. Then,
\begin{align*}
\abs{ n^{1/4} \frac{ \Ze_{n, \beta_n}  }{\Ze_{\Fb_{\! \gamma}}} - 1 } = O\prth{\frac{1}{\sqrt{n}}}.
\end{align*}
\end{lemma}


\begin{proof}
Using \eqref{Eq:DeFinettiMeasureT}, one has
\begin{align*}
\Ze_{n, \beta_n}  & := \int_{(-1, 1) } e^{- \prth{\frac{n}{2 \beta_n} \Argtanh(t)^2 + \frac{n}{2 } \ln(1 - t^2) }} \frac{dt}{1 - t^2} \\ 
                & = \int_\Rr e^{ - \prth{\frac{ \sqrt{n} }{2 \beta_n} y^2 - n \ln\,\prth{\cosh\,\prth{y/n^{1/4}} } } } \frac{dy}{n^{1/4} } \\
                & =  \int_\Rr e^{ - \prth{ \sqrt{n} (\frac{ 1 }{\beta_n} - 1) \frac{y^2}{2} - n \, \prth{  \ln\,\prth{\cosh\,\prth{y/n^{1/4}} }  - \, \prth{y/n^{1/4}}^2 /2 } } } \frac{dy}{n^{1/4} },
\end{align*}
having set $ y := n^{1/4} \Argtanh(t) $ and used $ 1 - \tanh(x)^2 = \cosh(x)^{-2} $.

\medskip 

Recall that $ \psi(y) := \frac{y^2}{2} - \ln\prth{\cosh(y) } $ is defined in \eqref{Def:CBetaAndPsiN} and set 
\begin{align*}
\gamma_n & := \sqrt{n} \prth{ \frac{ 1 }{\beta_n} - 1 } = \frac{\gamma}{\beta_n}, 
\end{align*}
so that 
\begin{align*}
\Ze_{n, \beta_n}  & =   \int_\Rr e^{ - \gamma_n \frac{y^2}{2} - n\, \psi( y/n^{1/4} ) }  \frac{dy}{ n^{1/4} }, \\
n^{1/4} \Ze_{n, \beta_n} - \Ze_{\Fb_{\! \gamma}} & = \int_\Rr \prth{ e^{ - \gamma_n \frac{y^2}{2} - n\, \psi( y/n^{1/4} ) } - e^{ - \gamma \frac{y^2}{2} - \frac{y^4}{12}  }  } dy.
\end{align*}

One has moreover with $ \gamma_n \geq \gamma  $
\begin{align*}
0 \leq \Ze_{\Fb_{\! \gamma}} - \Ze_{\Fb_{\! \gamma_n} } & = \int_\Rr \prth{ e^{ - \gamma \frac{y^2}{2} - \frac{y^4}{12}  } - e^{ - \gamma_n \frac{y^2}{2} - \frac{y^4}{12}  }  } dy = \int_\Rr \prth{ e^{ - \gamma \frac{y^2}{2}   } - e^{ - \gamma_n \frac{y^2}{2}   }  } e^{- \frac{y^4}{12} } dy \\
             & \leq (\gamma_n - \gamma) \int_\Rr \frac{y^2}{2} e^{- \frac{y^4}{12} } dy \leq \frac{\gamma^2}{2 \sqrt{n} } \Ze_\Fb  \Esp{\Fb^2}
\end{align*}
and the same inequality holds if $ \gamma \geq \gamma_n $ but with $ 0 \leq \Ze_{\Fb_{\! \gamma_n} } - \Ze_{\Fb_{\! \gamma}} $.

Last, the analysis of $  n^{1/4}\Ze_{n, \beta_n} - \Ze_{\Fb_{\! \gamma_n} } $ is similar to the previous one with $ \beta < 1 $, using exactly the same function $ \psi $ but with a different rescaling. One forms
\begin{align*}
0 \leq  n^{1/4}\Ze_{n, \beta_n} - \Ze_{\Fb_{\! \gamma_n} } & = \int_\Rr \prth{ e^{ - \gamma_n \frac{y^2}{2} - n \, \psi(y/n^{1/4}) } - e^{ - \gamma_n \frac{y^2}{2} - \frac{y^4}{12}  }  } dy =: \int_\Rr \prth{ 1 - e^{- \widetilde{\psi}_n(y) } } \widetilde{f}_{\Fb_{\!\gamma_n}}(y) dy  
\end{align*}
with $ \widetilde{f}_{\Fb_{\!\gamma_n}} := \Ze_{\Fb_{\!\gamma_n} } f_{\Fb_{\!\gamma_n}} $ and $ \widetilde{\psi}_n := n \widetilde{\psi}(\cdot/n^{1/4}) $ is defined in \eqref{Def:PsiTilde}.
%
%
%

Since $ \widetilde{\psi}_n \geq 0 $ and $ 0 \leq 1 - e^{-\widetilde{\psi}_n(y)} \leq \widetilde{\psi}_n(y) \leq \frac{y^6}{  \sqrt{n} } \norm{\widetilde{\psi}^{(6)} }_\infty \leq 136 \, \frac{y^6}{ \sqrt{n}} $, 
%
%
%
%
%
%
one gets
\begin{align*}
0 \leq  n^{1/4}\Ze_{n, \beta_n} - \Ze_{\Fb_{\!\gamma_n} } & \leq  \frac{136}{720} \, \frac{1}{ \sqrt{n}} \int_\Rr y^6 \widetilde{f}_{\Fb_{\!\gamma_n}}(y) dy  =  \frac{136}{720} \, \frac{\Ze_{\Fb_{\!\gamma_n} } \Esp{ \Fb_{\!\! \gamma_n}^6 } }{  \sqrt{n}} =  \frac{136}{720} \, \frac{\Ze_{\Fb_{\!\gamma } } \Esp{ \Fb_{\!\! \gamma}^6 } + o(1) }{ \sqrt{n}}.
\end{align*}

In the end, 
\begin{align*}
\abs{ n^{1/4} \Ze_{n, \gamma } - \Ze_{\Fb_{\! \gamma}}}  & \leq \abs{ n^{1/4} \Ze_{n, \gamma } - \Ze_{\Fb_{\gamma_n} } } + \abs{ \Ze_{\Fb_{\gamma_n} } - \Ze_{\Fb_{\! \gamma}} }  \\
              & \leq \frac{1}{ \sqrt{n}} \prth{ \frac{136}{ 720 } \Ze_{\Fb_{\!\gamma } } \Esp{  \Fb_{\!\! \gamma}^4 }  +  \frac{\gamma^2 }{2} \Ze_\Fb  \Esp{\Fb^2} + o(1) },
\end{align*}
which concludes the proof.
\end{proof}

\medskip
\subsubsection{Case $ \beta > 1 $}

Define for $ k \in \Nn $
\begin{align}\label{Def:ZnBetaK}
\Ze_{n, \beta}^{(k)} := 2\int_{\Rr_+} \abs{x - x_\beta}^k e^{-n \varphi_\beta(x)} dx , \qquad \varphi_\beta(x) := \frac{x^2}{2 \beta} - \log\cosh(x).
\end{align}

We will only be concerned with the cases $ k \in \{ 0, 1 \} $. Notice in particular that for $ k = 0 $, by symmetry of $ \varphi_\beta $, one has
\begin{align*}
\Ze_{n, \beta}^{(0)} := 2\int_{\Rr_+}  e^{-n \varphi_\beta(x)} dx = \int_\Rr  e^{-n \varphi_\beta(x)} dx = \Ze_{n, \beta}.
\end{align*}


\begin{lemma}[Asymptotic analysis of $ \Ze_{n, \beta}^{(k)} $ for $ \beta > 1 $]\label{Lemma:AsymptoticRenormConstant4}
For all $ k \in \Nn $ and with $ G \sim \Ns(0, 1) $, one has
\begin{align*}
\sqrt{n} \, e^{ n \varphi_\beta(x_\beta) } \Ze_{n, \beta}^{(k)}  & =   \prth{2 \sqrt{\frac{2\pi}{\varphi_\beta''(x_\beta)^{k + 1} } } \, \Esp{ \abs{G}^k }   +  O\prth{\frac{1}{\sqrt{n} } } }.
\end{align*}
\end{lemma}


\begin{proof}
One has
\begin{align*}
\sqrt{n} \, e^{ n \varphi_\beta(x_\beta) } \Ze_{n, \beta}^{(k)}   
                 & = 2 \sqrt{n}  \int_{\Rr_+}  \abs{x - x_\beta}^k e^{ -n (\varphi_\beta(x) - \varphi_\beta(x_\beta) )  } dx \\
                 & = 2 \sqrt{n} \int_{\Rr_+} \abs{x - x_\beta}^k e^{ -n (x - x_\beta)^2 \int_0^1 \varphi_\beta''(\alpha x + \overline{\alpha} x_\beta) \alpha d\alpha  } dx \\
                 & = 2 \int_{-x_\beta \sqrt{n} }^{+\infty}  \abs{w}^k e^{ - \frac{w^2}{2} \int_0^1 \varphi_\beta''\,\prth{  \alpha w / \sqrt{n} + x_\beta } \alpha d\alpha  } dw. 
\end{align*}

As $\varphi_\beta'' : x \mapsto -\frac{ \beta - 1}{\beta} + \tanh(x)^2$
is bounded and continuous, dominated convergence and continuity imply
\begin{align*}
\int_0^1 \varphi_\beta''(\alpha w/\sqrt{n} + x_\beta) \alpha d\alpha \tendvers{n}{+\infty} \int_0^1 \varphi_\beta''( x_\beta) \alpha d\alpha = \frac{\varphi_\beta''( x_\beta) }{2} 
\end{align*}
and $ \varphi_\beta''( x_\beta) > 0 $ by an easy study. It is also easy to see that $ x_\beta $ is the global minimum of $ \varphi_\beta $ on $ \Rr_+ $, hence that $ \varphi(x) - \varphi(x_\beta) > 0 $ on $ \Rr_+\!\!\setminus\!\ensemble{x_\beta} $. As a result, dominated convergence applies on this set to give
\begin{align*}
\int_{- x_\beta   \sqrt{n} }^{+\infty} \abs{w}^k  e^{ - w^2   \int_0^1 \varphi_\beta''(\alpha w/\sqrt{n} + x_\beta) \alpha d\alpha  } dw  
                 \rightarrow \int_\Rr \abs{w}^k e^{- \frac{w^2}{2} \varphi''(x_\beta) } dw 
                 = 2 \sqrt{\frac{2\pi}{\varphi_\beta''(x_\beta)^{k + 1} } } \, \Esp{\abs{G}^k},
\end{align*}
which is the result.
\end{proof}

\medskip
\subsection{Kolmogorov distance estimates}


\begin{lemma}[Kolmogorov Distance between two centered binomials]\label{Lemma:Binomials}
For all $ t \in (0, 1) $, define $\centered{S_n}(t) := S_n(t) - n(2t-1) $ with $ S_n(t) = \sum_{k = 1}^n B_k(t) $ and $ B_k(t) \sim \mathrm{i.i.d.}\Ber_{\pm 1}(t) $. Then, for all $p,q \in (0, 1) $, one has
\begin{align*}
\dKol\prth{ \centered{S_n}(p) , \centered{S_n}(q)  } \leq \abs{p-q} \abs{ \frac{ 1 - (p+q)}{p(1-p)}} + O\prth{ \frac{1}{\sqrt{n} } }.
\end{align*}
\end{lemma}


\begin{proof}
Recall that $ \Esp{\centered{S_n}(t)} = 0 $ and set $ \sigma_n^2(q) := \Var(\centered{S_n}(q)) = 4q(1-q) n $. 
%
%
%
We define $ G_p \sim \Ns(0,\sigma_n^2(p))$ and $ G_q \sim \Ns(0,\sigma_n^2(q)) $. The triangle inequality yields
\begin{align*}
\dKol\prth{ \centered{S_n}(p) , \centered{S_n}(q)  } & \leq \dKol\prth{ \centered{S_n}(p) , G_p  }  + \dKol\prth{ \centered{S_n}(q) , G_q  } + \dKol\prth{ G_p , G_q  } \\
                 & \leq \dKol\prth{ G_p , G_q  } + O\prth{ \frac{1}{\sqrt{n} } }
\end{align*}
by the Berry-Ess\'een theorem \eqref{Ineq:BerryEsseen:Kol}. It remains to prove that
\begin{align*}
\dKol\prth{ G_p , G_q  } & \leq \abs{p-q} \abs{ \frac{ 1 - (p+q)}{p(1-p)}}.
\end{align*}

Stein's method for $G_p$ gives \cite[(2.2)]{Ross}
\begin{align*}
\dKol\prth{ G_p , G_q  } & := \sup_{w \in \Rr} \abs{ \Prob{G_p \leq w} - \Prob{G_q \leq w} } \\
                         & = \sup_{w \in \Rr} \abs{\Esp{f_{w, p}'(G_q) - \frac{G_q}{\sigma_n^2(p)} f_{w, p}(G_q)}},
\end{align*}
where
\begin{align*}
f_{w, p}(x) := \frac{1}{f_{G_p}(x)} \Esp{  \Unens{ G_p \leq x } \prth{ \Unens{ G_p \leq w } - \Prob{ G_p \leq w } } }, \quad f_{G_p}(x) := \frac{1}{\sigma_n(p) \sqrt{2\pi}} e^{ - \frac{x^2}{2 \sigma_n(p)^2 } }
\end{align*}
is the solution of the Stein equation for $G_p$ (\cite[(19)]{SteinApprox}/\cite[(2.1)]{Ross}).

Moreover, the Gaussian integration by parts gives for $ f_{w, p} \in \Ce^1 $
\begin{align*}
\Esp{G_q f_{w, p}(G_q)} = \sigma_n^2(q) \Esp{f_{w, p}'(G_q)},
\end{align*}
implying
\begin{align*}
\dKol\prth{ G_p , G_q  } & = \sup \limits_{w \in \Rr} \abs{\Esp{f_{w, p}'(G_q) - \frac{\sigma_n^2(q)}{\sigma_n^2(p)} f_{w, p}'(G_q)}} \\
                                       & = \abs{ 1 - \frac{\sigma_n^2(q)}{\sigma_n^2(p)}} \sup_{w \in \Rr} \abs{\Esp{f_{w, p}'(G_q)}} \\
                                       & \leq \abs{ 1 - \frac{\sigma_n^2(q)}{\sigma_n^2(p)}} \sup_{w \in \Rr} \norm{f_{w, p}'}_\infty 
\end{align*}
and (see e.g. \cite[(22)]{SteinApprox}) $\sup_{w \in \Rr} \norm{f_{w, p}'}_\infty \leq 1$.
Finally
\begin{align*}
\abs{ 1 - \frac{\sigma_n^2(q)}{\sigma_n^2(p)}} & = \abs{ 1 - \frac{4 q(1-q)n}{4 p(1-p)n}} = \abs{p-q} \abs{ \frac{ 1 - (p+q)}{p(1-p)}},  
\end{align*}
%
%
which gives the result.
\end{proof}

\section{An identity in law}\label{Sec:IdentityInLaw}

We define the \textit{logistic function} $ \psi $ and its bijective inverse $ \psi\inv $ by
\begin{align*}
\psi(\alpha) & := \frac{e^\alpha }{e^{\alpha} + e^{-\alpha} } = \frac{1}{1 + e^{-2 \alpha } } = \frac{1}{2}\prth{ 1 + \tanh(\alpha) }, \quad\alpha \in \Rr, \\
\psi\inv(p) & := -\frac{1}{2} \log\prth{ p\inv - 1 } = \Argtanh(2p - 1), \quad p \in \crochet{0, 1}.
\end{align*}

Noting that $ \psi(0) = \frac{1}{2} $, we define
\begin{align}\label{Convention:Bernoulli}
X \equiv X\prth{\tfrac{1}{2}} = X\crochet{0} \sim \Ber_{\pm 1}\prth{\tfrac{1}{2} }, \quad X(\psi(\alpha)) = X\crochet{\alpha} \sim \Ber_{\pm 1}\prth{ \psi(\alpha) }.
\end{align}

It is well known that the exponential bias of a Bernoulli random variable $ X\crochet{\beta} $ is given by \cite{BarhoumiModOmega}
\begin{align}\label{Eq:BernoulliBias}
\frac{e^{\alpha X[\beta] } }{ \Esp{ e^{ \alpha X[\beta] } } } \bullet \Pp_{ X[\beta] } = \Pp_{ X[\alpha + \beta] }. 
\end{align}
Define $ S_n\crochet{\alpha} := \sum_{k = 1}^n X_k\crochet{\alpha} $ and $ S_n \equiv S_n\crochet{0} $. The Bernoulli bias \eqref{Eq:BernoulliBias} implies the Binomial bias:
\begin{align}\label{Eq:BinomialBias}
\frac{e^{\alpha S_n[\beta] } }{ \Esp{ e^{ \alpha S_n[\beta] } } } \bullet \Pp_{ S_n[\beta] }  = \Pp_{ S_n[\alpha + \beta] }  
\end{align}
and one also has the Gaussian bias, with $ G  \sim \Ns(0, 1)$,
\begin{align}\label{Eq:GaussianBias}
\frac{e^{\gamma G  } }{ \Esp{ e^{ \gamma G  } } } \bullet \Pp_{ G  } = \Pp_{ G + \gamma } .
\end{align}

We start by deriving the De Finetti measure \eqref{Eq:DeFinettiCW:Meas:P} and \eqref{Eq:DeFinettiMeasureP}:

\begin{lemma}[The De Finetti measure of the Curie-Weiss spins]\label{Lemma:DeFinettiCWspins}
Recall from \eqref{Def:Law:RnBeta} that
\begin{align*}
\varphi_\beta(u) := \frac{u^2}{2\beta} - \log\cosh(u), \qquad \Ze_{n, \beta} := \int_\Rr e^{-n \varphi_\beta}.
\end{align*}

Then,
\begin{align}\label{Eq:DeFinettiCW:Desintegration}
\begin{aligned}
\Pp_{(X_1^{(\beta)}, \dots, X_n^{(\beta)})} & = \int_\Rr  \Pp_{(X_1[u], \dots, X_n[u])} \mu_n(du), \qquad \mu_n(du) = \frac{1}{\Ze_{n, \beta} }   e^{ - n \varphi_\beta(u) } du, \\
                   & = \int_\Rr  \Pp_{(X_1(p), \dots, X_n(p))} \widetilde{\nu}_n(dp), \qquad \widetilde{\nu}_n = \mu_n\circ\psi\inv,
\end{aligned}
\end{align}
where $ \widetilde{\nu}_n $ is defined explicitely in \eqref{Eq:DeFinettiMeasureP}.
\end{lemma}


\begin{proof}
Set 
$$
S_n := \sum_{k = 1}^n X_k, \quad \Zg_{n, \beta} := \Esp{ e^{ \frac{\beta^2}{2n} S_n^2 } }. 
$$ 

Then, for $ f $ bounded, we have
\begin{align*}
\Esp{ f\prth{ X_1^{(\beta)}, \dots, X_n^{(\beta)} } } & := \frac{1}{\Zg_{n, \beta}} \Esp{ e^{ \frac{\beta^2}{2n} S_n^2 } f\prth{ X_1 , \dots, X_n  } } \\ 
                & = \frac{\sqrt{n}}{\beta \Zg_{n, \beta}\sqrt{2\pi} } \int_\Rr \Esp{ e^{ u S_n  } f\prth{ X_1 , \dots, X_n  } } e^{ - n \frac{u^2}{2\beta } } du   \\
                & = \frac{\sqrt{n}}{\beta \Zg_{n, \beta}\sqrt{2\pi} } \int_\Rr \Esp{ e^{ u S_n  } } \Esp{  f\prth{ X_1 \crochet{u } , \dots, X_n\crochet{u }   } } e^{ - n \frac{u^2}{2\beta } } du \mbox{  with \eqref{Eq:BernoulliBias} } \\ 
                & = \frac{\sqrt{n}}{\beta \Zg_{n, \beta}\sqrt{2\pi} } \int_\Rr \Esp{  f\prth{ X_1 \crochet{ u } , \dots, X_n\crochet{u }   } } \cosh(u)^n e^{ - n \frac{u^2}{2\beta } } du \\
                & = \int_\Rr \Esp{  f\prth{ X_1 \crochet{ u } , \dots, X_n\crochet{u } } } \mu_n(du),
\end{align*}
as $ \Ze_{n, \beta} = \frac{\beta \Zg_{n, \beta}\sqrt{2\pi}}{\sqrt{n}} $ by setting $ f = 1 $. The other case is obtained by taking the image measure with $ \psi $, using the convention \eqref{Convention:Bernoulli} that writes $ X(p) = X\crochet{\psi\inv(p)} $.
\end{proof}

\begin{remark}
Notice the difference between the bias definition \eqref{Def:CurieWeiss} and the mixture property \eqref{Eq:DeFinettiCW:Desintegration}. Another way to rephrase the whole philosophy of the surrogate approach would be: when adding a particular penalisation to an i.i.d. sum is equivalent to perform an independent randomisation, one should only use the randomisation framework. Such a situation is of course extremely particular, but one could try to describe more complicated systems as approximation of ``mixture systems'', with a replacement method for instance. We leave this study for a future work.
\end{remark}

We now show that the magnetisation is solution to a fixed point equation in law, which allows to simulate it with a Markov chain approximating the fixed point: 

\begin{lemma}[The magnetisation as a solution to an equation in law]\label{Lemma:MagnetisationWithRandomisation}
We have
\begin{align}\label{Eq:CWMagnetisationAsARandomisation}
M_n^{(\beta)} \eqlaw \widetilde{S}_n\crochet{  \sqrt{\frac{\beta}{n} } G + \mu +  \frac{\beta}{n} M_n^{(\beta)}},  
\end{align}
where $ (\widetilde{S}_n, G, M_n^{(\beta)}) $ are independent.
\end{lemma}


\begin{proof}
Let $ \widetilde{S}_n, S_n, G $ independent random variables and $ f $ a bounded function. One can write with $ \frac{\beta}{n} := \alpha^2 $
\begin{align*}
\Esp{ f\prth{ M_n^{(\beta)} } } & = \frac{ \Esp{ e^{ \frac{\beta}{2n} S_n^2 + \mu S_n } f(S_n) } }{ \Esp{ e^{ \frac{\beta}{2n} S_n^2 + \mu S_n } } }  \\
                 & = \frac{ \Esp{ e^{  (\alpha G + \mu) S_n} f(S_n) } }{ \Esp{ e^{(\alpha G + \mu) S_n } } } \quad\mbox{with $G$ independent of $ S_n $}\\
                 & = \frac{ \Esp{  \Espr{S_n \!}{ e^{  (\alpha G + \mu) S_n} }  f(S_n\crochet{   \alpha G + \mu    }  ) } }{ \Esp{ e^{(\alpha G + \mu) S_n } } } \quad\mbox{by \eqref{Eq:BernoulliBias}} \\
                 & = \frac{ \Esp{ e^{  (\alpha G + \mu) \widetilde{S}_n} f\prth{ S_n\crochet{   \alpha G + \mu    } } } }{ \Esp{ e^{(\alpha G + \mu) \widetilde{S}_n } } } \quad\mbox{with $ \widetilde{S}_n \eqlaw S_n $ independent of } G, S_n  \\
                 & =  \frac{ \Esp{  e^{  (\alpha G + \mu) \widetilde{S}_n}  f\prth{ S_n\crochet{   \alpha G + \mu    } } } }{ \Esp{ e^{(\alpha G + \mu) \widetilde{S}_n } } } \\
                 & = \frac{ \Esp{  \Espr{ G \,}{ e^{(\alpha G + \mu) \widetilde{S}_n } } f\prth{ S_n\crochet{  \alpha G + \mu + \alpha^2 \widetilde{S}_n } } } }{ \Esp{ e^{(\alpha G + \mu) \widetilde{S}_n } } }  \quad\mbox{by \eqref{Eq:GaussianBias} with $ \gamma = \alpha \widetilde{S}_n $,  } \\
                 & = \frac{ \Esp{  e^{ \alpha \widetilde{S}_n^2 + \mu \widetilde{S}_n }   f\prth{ S_n\crochet{  \alpha G + \mu + \alpha^2 \widetilde{S}_n } } } }{ \Esp{ e^{ \alpha \widetilde{S}_n^2 + \mu \widetilde{S}_n  } } } \\
                 & = \Esp{ f\prth{ S_n\crochet{  \alpha G + \mu + \alpha^2 \widetilde{S}^{(\beta)}_n } } }, 
\end{align*}
which gives the result.
\end{proof}

This result gives a way to simulate $ M_n^{(\beta)} $ (hence the whole sequence of spins) with a sequence of i.i.d. Binomial and Gaussian random variables $ (\Bin_k)_k $ and $ (G_k)_k $. One can form the Markov chain $ (M_k)_{k \geq 0} $ given, with obvious notations, by:
\begin{align*}
M_{k + 1} \eqlaw  \Bin_k\prth{ n,  \, \psi\prth{  \sqrt{\frac{\beta}{n} } G_k + \mu +  \frac{\beta}{n} M_k } }, \qquad M_0 \sim \Bin\prth{n, \tfrac{1}{2}}.
\end{align*}

An interesting question is to find its rate of convergence to equilibrium (i.e. to $ M_n^{(\beta)} $) that would give the precise number of iterations to simulate it, and in particular if a cut-off phenomenon appears in total variation distance.

\section*{Acknowledgements}

The authors thank J. Heiny for the reference \cite{KirschMoments}, O. H\'enard for the correction of an early mistake and W. Kirsch and Y. Velenik for interesting discussions. A particular thanks is given to J. Steif for bringing the attention of the authors to \cite{LiggettSteifToth} and the numerous references on the use of the De Finetti mixture in the study of the Curie-Weiss model. Thanks are also due to the two anonymous referees for pointing out several omissions and errors in addition to giving interesting insights and references. 

\bibliographystyle{amsplain}


\end{document}